\documentclass[a4paper,12pt]{amsart}
\usepackage{amssymb, amsmath, graphicx}
\usepackage[curve]{xypic}
\usepackage{enumerate}
\usepackage{tikz-cd}  
\usepackage[indexonlyfirst,sort=use]{glossaries}
\usepackage{url}
\DeclareRobustCommand{\gobblefour}[4]{}
\newcommand*{\SkipTocEntry}{\addtocontents{toc}{\gobblefour}}

\providecommand{\abs}[1]{\lvert#1\rvert}

\providecommand{\id}{\textnormal{id}}

\providecommand{\Ker}{\textnormal{Ker}}
\providecommand{\IIm}{\textnormal{Im}}

\providecommand{\Hom}{\textnormal{Hom}}

\providecommand{\Ker}{\textnormal{Ker}}

\providecommand{\Int}{\textnormal{int}}
\providecommand{\IInt}{\textnormal{INT}}
\providecommand{\ch}{\textnormal{ch}}
\providecommand{\PD}{\textnormal{PD}}
\providecommand{\hatPD}{\widehat{\textnormal{PD}}}
\providecommand{\hatL}{\widehat{\textnormal{L}}}
\providecommand{\cpt}{\textnormal{cpt}}

\providecommand{\chr}{\textnormal{char}}

\providecommand{\Oo}{\mathcal{O}}

\providecommand{\vol}{\textnormal{vol}}

\providecommand{\fl}{\textnormal{fl}}

\providecommand{\Z}{\mathbb{Z}}

\providecommand{\R}{\mathbb{R}}
\providecommand{\C}{\mathbb{C}}
\providecommand{\Ii}{\mathbb{I}}

\providecommand{\h}{\mathfrak{h}}

\providecommand{\cl}{\textnormal{cl}}

\providecommand{\dR}{\textnormal{dR}}
\providecommand{\Td}{\textnormal{Td}}

\providecommand{\cov}{\textnormal{cov}}

\providecommand{\Tau}{\mathcal{T}}
\providecommand{\ScT}{\textnormal{T}}
\providecommand{\Thom}{\textnormal{Th}}
\providecommand{\hatcap}{\,\hat{\cap}\,}
\providecommand{\hatcup}{\,\hat{\cup}\,}
\providecommand{\hattimes}{\,\hat{\times}\,}
\providecommand{\hatslant}{\,\hat{/}\,}

\providecommand{\BM}{\textnormal{BM}}
\providecommand{\vcpt}{\textnormal{vcpt}}
\providecommand{\vncpt}{\textnormal{vncpt}}

\providecommand{\pt}{\textnormal{pt}}

\tikzset{
	curvarr/.style={
		to path={ -- ([xshift=2ex]\tikztostart.east)
			|- (#1) [near end]\tikztonodes
			-| ([xshift=-2ex]\tikztotarget.west)
			-- (\tikztotarget)}
	}
}

\makeatletter
\newcommand\tint{\mathop{\mathpalette\tb@int{t}}\!\int}
\newcommand\bint{\mathop{\mathpalette\tb@int{b}}\!\int}
\newcommand\tb@int[2]{%
  \sbox\z@{$\m@th#1\int$}%
  \if#2t%
    \rlap{\hbox to\wd\z@{%
      \hfil
      \vrule width .35em height \dimexpr\ht\z@+1.4pt\relax depth -\dimexpr\ht\z@+1pt\relax
      \kern.05em 
    }}
  \else
    \rlap{\hbox to\wd\z@{%
      \vrule width .35em height -\dimexpr\dp\z@+1pt\relax depth \dimexpr\dp\z@+1.4pt\relax
      \hfil
    }}
  \fi
}
\makeatother

\allowdisplaybreaks

\begin{document}

\newtheorem{Theorem}{Theorem}[section]
\newtheorem{Lemma}[Theorem]{Lemma}
\newtheorem{Prop}[Theorem]{Proposition}
\newtheorem{Corollary}[Theorem]{Corollary}
\newtheorem{ThmDef}[Theorem]{Theorem - Definition}

\theoremstyle{definition}
\newtheorem{Rmk}[Theorem]{Remark}
\newtheorem*{Rmk*}{Remark}
\newtheorem{Rmks}[Theorem]{Remarks}
\newtheorem{Rmks*}[Theorem]{Remarks}
\newtheorem{Def}[Theorem]{Definition}
\newtheorem{Def2}[Theorem]{Definition}
\newtheorem{Not}[Theorem]{Notation}

\AtEndEnvironment{Def}{\hfill$\diamondsuit$}
\AtEndEnvironment{Rmk}{\hfill$\diamondsuit$}


\title{Differential Homology}
\author{Fabio Ferrari Ruffino and Gabriel Longatto Clemente}
\address{Departamento de Matem\'atica - Universidade Federal de S\~ao Carlos - Rod.\ Washington Lu\'is, Km 235 - C.P.\ 676 - 13565-905 S\~ao Carlos, SP, Brasil}
\email{ferrariruffino@ufscar.com, glclemente@estudante.ufscar.br}
\thanks{This study was financed, in part, by the S\~ao Paulo Research Foundation (FAPESP), Brazil, Process Number 2022/00676-3.}

\begin{abstract}
We first construct the differential refinement of singular homology, verifying that it satisfies the axioms dual to their cohomological counterparts. Then, we define the differential cap product, leading to Poincar\'e duality. The essential uniqueness of both differential homology and the cap product is proven. Moreover, we construct the relative, non-compact, and local coefficient versions, so that we can state Poincar\'e and Lefschetz dualities for every smooth manifold with our without boundary. Afterwards, we develop the analogous refinement of any rationally-even homology theory, with the same properties. We also define a ``differentiation map'' for fibre bundles, dual to the integration map in cohomology. We conclude by sketching some possible applications of this theory in mathematical physics.
\end{abstract}

\maketitle


\tableofcontents


\section*{Introduction}

The notion of differential refinement of a cohomology theory is widely studied in the mathematical literature, and its application in mathematical physics are well-known. The whole theory began by introducing the concept of \emph{Cheeger-Simons character} in \cite{CS}, with the aim of finding obstructions to conformal embeddings of Riemannian manifolds into a Euclidean space. Then, other equivalent models were constructed or recognised, a remarkable one being the Deligne cohomology (see \cite{Brylinski}). For this reason, it was natural to inquire about the essential uniqueness of the differential theory under suitable axioms, which was proven by Simons and Sullivan in \cite{SS}.

Beyond singular cohomology, various authors dealt with the differential extension of K-theory and its physical applications (for instance, see \cite{BS3}, \cite{BS2}, \cite{FL}, \cite{SS2}). This suggests the possibility of refining every cohomology theory to the differential setting. Hopkins and Singer achieved this result in the seminal paper \cite{HS} by starting from any theory represented by a spectrum. Their model has been completed in \cite{Upmeier} by describing in detail the $S^{1}$-integration and the multiplicative structure. The issue of uniqueness then arose in this general setting as well, and it was settled by Bunke and Schick in \cite{BS} under reasonable hypotheses.

In the present paper, we develop the dual theory in the homological framework. By using currents instead of forms, we first define differential singular homology. We state the axioms dual to the corresponding cohomological ones, and we verify that they are satisfied. Then, we construct the differential cap product, which allows to state and prove the refined Poincar\'e duality. The problem of uniqueness is solved both for the differential refinement and the cap product by adapting the techniques of \cite{SS}. Furthermore, we construct the analogous differential extension for relative, non-compact, and local coefficient homologies. In this way, we obtain the necessary tools to state the various versions of Poincar\'e and Lefschetz dualities, even without compactness or orientability assumptions.

Afterwards, we generalise the previous constructions to any rationally-even homology theory represented by a spectrum, achieving analogous results in this general context. Moreover, we dualise the integration map (or Gysin map) for fibre bundles, and we call ``differentiation map'' the dual version. We conclude by sketching how differential homology can be used in mathematical physics to describe charge sources and the minimal coupling (or Wess-Zumino action) without relying on a fixed representative cycle.

The paper is organised as follows. In section \ref{SecDiffHom}, we construct differential singular homology. In section \ref{ProdDualitySec}, we define the cap product and deal with Poincar\'e duality; we also define the external and slant products. Section \ref{SecUnique} is dedicated to the issue of uniqueness, and in section \ref{CapCSSec} we relate differential homology to the holonomy map in cohomology (equivalently, the evaluation of a differential character on a cycle). The relative, non-compact, and local coefficient versions are built in section \ref{RelNonCptLocCSec}, leading to the corresponding dualities. In section \ref{GenHomSec}, we construct the differential refinement of a generalised homology theory, reaching analogous results in this setting. Section \ref{GysinMapHomSec} is dedicated to dualise the integration map and to the problem of uniqueness. We conclude with section \ref{FurthPer}, in which we summarise why this theory may be interesting in mathematical physics and indicate some possible future developments.


\section{Differential Singular Homology}\label{SecDiffHom}

\SkipTocEntry \subsection{Notation on Currents}

We denote by $X$ a smooth manifold (with or without boundary) of dimension $n$ and by $\Omega^{k}(X)$ the space of smooth forms of degree $k$ on $X$. We define the space of compactly-supported currents of degree $k$ on $X$, which we denote by $\Tau^{k}(X)$, as the dual of $\Omega^{n-k}(X)$.\footnote{The reader can find the basic theory of currents in \cite{GH}, chapter 3, section 1.} Moreover, $d \colon \Tau^{k}(X) \to \Tau^{k+1}(X)$ denotes the exterior differential, that is, given $T^{k} \in \Tau^{k}(X)$ and $\omega^{n-k-1} \in \Omega^{n-k-1}(X)$:
\begin{equation}\label{ExtDiffCurrents}
	dT^{k}(\omega^{n-k-1}) := (-1)^{k+1}T^{k}(d\omega^{n-k-1})
\end{equation}
Since we work in the homological context, we set $\Tau_{k}(X) := \Tau^{n-k}(X)$---that is, $\Tau_{k}(X)$ is the dual of $\Omega^{k}(X)$---and we define
\begin{eqnarray}
	\partial \colon \, \Tau_{k}(X) & \to & \Tau_{k-1}(X) \label{DefBoundaryCurr} \\
	T_{k} & \mapsto & (-1)^{n-k+1} dT_{k}. \nonumber
\end{eqnarray}
It follows that
\begin{equation}\label{BoundaryDifferential}
	(\partial T_{k})(\omega^{k-1}) \overset{\eqref{DefBoundaryCurr}}= (-1)^{n-k+1} (dT_{k})(\omega^{k-1}) \overset{\eqref{ExtDiffCurrents}}= T_{k}(d\omega^{k-1}).
\end{equation}
We denote by $C_{k}(X)$, $Z_{k}(X)$, and $B_{k}(X)$ respectively the groups of \emph{smooth} singular chains, cycles, and boundaries on $X$ of degree $k$ with coefficients in $\Z$. The group $H_{k}(X) := Z_{k}(X)/B_{k}(X)$ is naturally isomorphic to the singular homology group of $X$ of degree $k$. When $G$ is the coefficient group, we write $C_{k}(X; G)$, $Z_{k}(X; G)$, etc.\ Moreover, we use the analogous notation with upper indices in the cohomological setting.

Given a real chain $\gamma_{k} \in C_{k}(X; \R)$, we denote by $T_{\gamma_{k}} \in \Tau_{k}(X)$ the current $\omega^{k} \mapsto \int_{\gamma_{k}} \omega^{k}$. We obtain the abelian group morphism
\begin{eqnarray*}
	\ScT_{k} \colon \, C_{k}(X; \R) & \to & \Tau_{k}(X) \\
	\gamma_{k} & \mapsto & T_{\gamma_{k}}
\end{eqnarray*}
whose kernel is formed by \emph{thin} chains. We have:
\begin{align*}
	(\partial T_{\gamma_{k}})(\omega^{k-1}) & \overset{\eqref{BoundaryDifferential}}= T_{\gamma_{k}}(d\omega^{k-1}) = \int_{\gamma_{k}} d\omega^{k-1} = \int_{\partial \gamma_{k}} \omega^{k-1} = T_{\partial \gamma_{k}}(\omega^{k-1})
\end{align*}
Therefore:
\begin{equation}\label{BoudaryCurrents}
	\partial T_{\gamma_{k}} = T_{\partial \gamma_{k}}
\end{equation}
Projecting the morphism $\ScT_{k}$ to homology, we obtain the isomorphism
\begin{equation}\label{BarPsiIso}
	\bar{\ScT}_{k} \colon H_{k}(X; \R) \overset{\!\simeq}\longrightarrow H\Tau_{k}(X).
\end{equation}
Given $T_{k} \in \Tau_{k}(X)$ and $\eta^{h} \in \Omega^{h}(X)$, we define $T_{k} \wedge \eta^{h} \in \Tau_{k-h}(X)$ by $(T_{k} \wedge \eta^{h})(\omega^{k-h}) := T_{k}(\eta^{h} \wedge \omega^{k-h})$. For future reference, we compute the corresponding boundary:
\begin{align}
	\partial (T_{k} \wedge \eta^{h}) & \overset{\eqref{DefBoundaryCurr}}= (-1)^{n-k+h+1}d(T_{k} \wedge \eta^{h}) \nonumber \\
	& = (-1)^{n-k+h+1}\bigl(dT_{k} \wedge \eta^{h} + (-1)^{n-k} T_{k} \wedge d\eta^{h}\bigr) \nonumber \\
	& = (-1)^{h}\bigl(\partial T_{k} \wedge \eta^{h} - T_{k} \wedge d\eta^{h}\bigr) \label{BoudaryWedge}
\end{align}
This formula is analogous to the one for $\partial(\gamma_{k} \cap \eta^{h})$. Moreover, we define \emph{integral} currents as follows:
\begin{equation}\label{DefIntCurr}
	\Tau_{k}^{\Int}(X) := \bigl\{ T_{\gamma_{k}} + \partial T'_{k+1}: \gamma_{k} \in Z_{k}(X), T'_{k+1} \in \Tau_{k+1}(X) \bigr\} \subset \Tau_{k}(X)
\end{equation}
This means that an integral current is a representative of the real image of an integral homology class. Equivalently, a current $T_{k} \in \Tau_{k}(X)$ is integral if and only if $T_{k}(\omega^{k}) \in \Z$ for every integral form $\omega^{k}$.

\SkipTocEntry \subsection{Definition of Differential Homology}

\begin{Def}\label{DefDiffHom} The \emph{differential singular homology group} of degree $k$ on $X$, denoted by $\hat{H}_{k}(X)$, is the quotient of $Z_{k}(X) \times \Tau_{k+1}(X)$ by the subgroup whose elements are of the form
\begin{equation}\label{EqRelDiffHom}
	(\partial \Gamma_{k+1}, - T_{\Gamma_{k+1}} - \partial T'_{k+2})
\end{equation}
with $\Gamma_{k+1} \in C_{k+1}(X)$ and $T'_{k+2} \in \Tau_{k+2}(X)$.
\end{Def}
This means that an element of $\hat{H}_{k}(X)$ is a class $[\gamma_{k}, T_{k+1}]$, where $[\partial \Gamma_{k+1}, 0] = [0, T_{\Gamma_{k+1}}]$ and $[0, \partial T'_{k+2}] = 0$. In particular, $\gamma_{k}$ is relevant only up to boundaries of thin chains and $T_{k+1}$ is relevant only up to exact currents. Furthermore, given a smooth map $f \colon X \to Y$, the natural pushforward $f_{*} \colon \hat{H}_{k}(X) \to \hat{H}_{k}(Y)$ is defined by $f_{*}[\gamma_{k}, T_{k+1}] := [f_{*}\gamma_{k}, f_{*}T_{k+1}]$, where $(f_{*}T_{k+1})(\omega^{k+1}) := T_{k+1}(f^{*}\omega^{k+1})$.

\SkipTocEntry \subsection{Natural Morphisms}

We define the following natural transformations:
\begin{align}
	& I \colon \hat{H}_{k}(X) \to H_{k}(X), & & [\gamma_{k}, T_{k+1}] \mapsto [\gamma_{k}] \label{DefI} \\
	& R \colon \hat{H}_{k}(X) \to \Tau_{k}(X), & & [\gamma_{k}, T_{k+1}] \mapsto T_{\gamma_{k}} + \partial T_{k+1} \label{DefR} \\
	& a \colon \Tau_{k+1}(X) \to \hat{H}_{k}(X), & & T_{k+1} \mapsto [0, T_{k+1}] \label{DefA}
\end{align}
Let us check that they are well-defined, that is, they vanish on trivial classes, whose representatives are of the form \eqref{EqRelDiffHom}. We have:
\begin{align*}
	& I\bigl[\partial \Gamma_{k+1}, - T_{\Gamma_{k+1}} - \partial T'_{k+2}\bigr] = [\partial \Gamma_{k+1}] = 0 \\
	& R\bigl[\partial \Gamma_{k+1}, - T_{\Gamma_{k+1}} - \partial T'_{k+2}\bigr] = T_{\partial \Gamma_{k+1}} - \partial T_{\Gamma_{k+1}} \overset{\eqref{BoudaryCurrents}}= 0
\end{align*}
Nothing needs to be checked about $a$. We also have the following transformation, that will be proved to be injective:
\begin{eqnarray}
	\iota \colon \, H_{k+1}(X; \R/\Z) & \to & \hat{H}_{k}(X) \label{MorfIota} \\
	{} [\hspace{1pt}\Gamma_{k+1}] & \mapsto & \bigl[-\partial \tilde{\Gamma}_{k+1}, T_{\tilde{\Gamma}_{k+1}} \bigr] \nonumber
\end{eqnarray}
where $\tilde{\Gamma}_{k+1} \in C_{k+1}(X; \R)$ is any lift of $\Gamma_{k+1} \in Z_{k+1}(X; \R/\Z)$. In fact, if we choose another lift $\tilde{\Gamma}_{k+1}' = \tilde{\Gamma}_{k+1} + \Theta_{k+1}$, where $\Theta_{k+1} \in C_{k+1}(X)$, then the contribution of $\Theta_{k+1}$ is $[-\partial \Theta_{k+1}, T_{\Theta_{k+1}}]$, that vanishes because it is of the form \eqref{EqRelDiffHom} (since $\Theta_{k+1}$ is integral, while $\tilde{\Gamma}_{k+1}$ is real). By fixing $\Gamma_{k+1}$, this shows that the result does not depend on the chosen lift. Lastly, if we fix another representative $\Gamma'_{k+1} = \Gamma_{k+1} + \partial \Xi_{k+2}$, then we choose the lift $\tilde{\Gamma}'_{k+1} = \tilde{\Gamma}_{k+1} + \partial \tilde{\Xi}_{k+2}$. The contribution of $\partial \tilde{\Xi}_{k+2}$ is $[0, T_{\partial\tilde{\Xi}_{k+2}}] = [0, \partial T_{\partial\tilde{\Xi}_{k+2}}] = 0$.

\begin{Prop}\label{PropAxioms} The following square is well-defined and commutative:
\begin{equation}\label{CommSquare}
	\xymatrix{
	\hat{H}_{k}(X) \ar@{->>}[rr]^{I} \ar@{->>}[d]_(.45){R} & & H_{k}(X) \ar[d]^(.475){- \otimes 1} \\
	\Tau_{k}^{\Int}(X) \ar[rr]^{\bar{\ScT}_{k}^{\,-1}} & & H_{k}(X; \R)
}\end{equation}
The isomorphism $\bar{\ScT}_{k}$ was defined in \eqref{BarPsiIso} and we are applying implicitly the projection from $\Tau_{k}^{\Int}(X)$ to currential homology. Moreover:
\begin{equation}\label{ThreeFormulas}
	R \circ a = \partial \hspace{50pt} I \circ \iota = -\beta \hspace{50pt} \iota \circ \pi_{\R/\Z} = a \circ \bar{\ScT}_{k}
\end{equation}
where $\beta \colon H_{k+1}(X; \R/\Z) \to H_{k}(X)$ and $\pi_{\R/\Z} \colon H_{k}(X; \R) \to H_{k}(X; \R/\Z)$ are parts of the long exact sequence induced by $0 \to \Z \to \R \to \R/\Z \to 0$. Lastly, the following sequences are exact:
\begin{align}
	& \xymatrix{0 \ar[r] & \Tau_{k+1}^{\Int}(X) \ar[r] & \Tau_{k+1}(X) \ar[r]^{a} & \hat{H}_{k}(X) \ar[r]^{I} & H_{k}(X) \ar[r] & 0} \label{ExSeq1} \\
	& \xymatrix{0 \ar[r] & H_{k+1}(X; \R/\Z) \ar[r]^(.6){\iota} & \hat{H}_{k}(X) \ar[r]^(.45){R} & \Tau_{k}^{\Int}(X) \ar[r] & 0} \label{ExSeq2}
\end{align}
\end{Prop}
\begin{proof} About the square \eqref{CommSquare}, the upper-right path is $[\gamma_{k}, T_{k+1}] \mapsto [\gamma_{k}] \mapsto [\gamma_{k}] \otimes 1$ and the left-lower path is $[\gamma_{k}, T_{k+1}] \mapsto T_{\gamma_{k}} + \partial T_{k+1} \mapsto \bar{\ScT}_{k}^{\,-1}[T_{\gamma_{k}}] = [\gamma_{k}] \otimes 1$. The morphism $I$ is surjective, since, given $[\gamma_{k}] \in H_{k}(X)$, we have $[\gamma_{k}] = I[\gamma_{k}, 0]$. The surjectivity of $R$ immediately follows from definitions \eqref{DefR} and \eqref{DefIntCurr}.

About \eqref{ThreeFormulas}, we have $R \circ a(T_{k+1}) = R[0, T_{k+1}] = \partial T_{k+1}$, and the other two formulas immediately follow from definition \eqref{MorfIota}.

About \eqref{ExSeq1}, we have $a(T_{k+1}) = [0, T_{k+1}] = 0$ if and only if $(0, T_{k+1})$ is of the form \eqref{EqRelDiffHom}, which means that there exist $\Gamma_{k+1}$ and $T'_{k+2}$ such that $(0, T_{k+1}) = (\partial \Gamma_{k+1}, -T_{\Gamma_{k+1}} - \partial T'_{k+2})$. Equivalently, there exist an integral cycle $\Gamma_{k+1}$ and a current $T'_{k+2}$ such that $T_{k+1} = -T_{\Gamma_{k+1}} - \partial T'_{k+2}$, that is, $T_{k+1}$ is integral by definition \eqref{DefIntCurr}. Moreover, $I \circ a(T_{k+1}) = I[0, T_{k+1}] = [0] = 0$. Conversely, if $I[\gamma_{k}, T_{k+1}] = [\gamma_{k}] = 0$, then $\gamma_{k} = \partial \Gamma_{k+1}$, thus $[\gamma_{k}, T_{k+1}] = [\partial \Gamma_{k+1}, T_{k+1}] = [0, T_{\Gamma_{k+1}} + T_{k+1}] = a(T_{\Gamma_{k+1}} + T_{k+1})$. Lastly, we already proved the surjectivity of $I$.

About \eqref{ExSeq2}, we first prove that $\iota$ is injective. If $\iota[\Gamma_{k+1}] = 0$, then the second formula in \eqref{ThreeFormulas} implies that $[\Gamma_{k+1}]$ lifts to a real homology class, meaning that we can choose a lift $\tilde{\Gamma}_{k+1} \in Z_{k+1}(X; \R)$ that is a cycle. Thus, our hypothesis becomes $\iota[\Gamma_{k+1}] = [0, T_{\tilde{\Gamma}_{k+1}}] = 0$. It follows from \eqref{EqRelDiffHom} that $(0,  T_{\tilde{\Gamma}_{k+1}}) = (\partial \tilde{\Theta}_{k+1}, T_{\tilde{\Theta}_{k+1}} + \partial T'_{k+2})$. Therefore, $\tilde{\Theta}_{k+1}$ is an integral cycle and $T_{\tilde{\Gamma}_{k+1}} = T_{\tilde{\Theta}_{k+1}} + \partial T''_{k+2}$. The isomorphism \eqref{BarPsiIso} implies that there exists $\tilde{\Xi}_{k+2} \in C_{k+2}(X; \R)$ such that $\tilde{\Gamma}_{k+1} = \tilde{\Theta}_{k+1} + \partial \tilde{\Xi}_{k+2}$. Projecting to $\R/\Z$, we obtain $\Gamma_{k+1} = \partial \Xi_{k+2}$, that is, $[\Gamma_{k+1}] = 0$. This proves injectivity. It immediately follows from definitions \eqref{MorfIota} and \eqref{DefR} that $R \circ \iota = 0$. Conversely, let us suppose that $R[\gamma_{k}, T_{k+1}] = 0$, which means that $T_{\gamma_{k}} + \partial T_{k+1} = 0$. This implies that $T_{\gamma_{k}}$ is exact, thus $\gamma_{k}$ is exact as a real cycle. We set $\gamma_{k} = \partial \tilde{\Gamma}_{k+1}$. Projecting to $\R/\Z$, we get $\Gamma_{k+1}$, which is a cycle because $\partial \tilde{\Gamma}_{k+1} = \gamma_{k}$ is integral. Since $\partial T_{k+1} = -T_{\gamma_{k}} = -T_{\partial \tilde{\Gamma}_{k+1}} = -\partial T_{\tilde{\Gamma}_{k+1}}$, we have $T_{k+1} = -T_{\tilde{\Gamma}_{k+1}} + T'_{k+1}$ with $T'_{k+1}$ closed. Hence, there exists a real cycle $\tilde{\Theta}_{k+1}$ such that $T'_{k+1} = -T_{\tilde{\Theta}_{k+1}} + \partial T''_{k+2}$. It follows that $[\gamma_{k}, T_{k+1}] = [\partial \tilde{(\Gamma}_{k+1} + \tilde{\Theta}_{k+1}), -T_{\tilde{\Gamma}_{k+1} + \tilde{\Theta}_{k+1}}] = -\iota(\Gamma_{k+1} + \Theta_{k+1})$, as required. This concludes the proof, since we already showed that $R$ is surjective.
\end{proof}

\begin{Corollary}\label{CorExSeq3} The following sequence is exact:
\begin{equation}\label{ExSeq3}
	\xymatrix{0 \ar[r] & \frac{H_{k+1}(X; \R)}{H_{k+1}(X)} \ar[r]^(.52){a \circ \bar{\ScT}_{k}} & \hat{H}_{k}(X) \ar[r]^(.34){(I, \, R)} & H_{k}(X) \oplus \Tau_{k}^{\Int}(X)
}\end{equation}
The morphism $a \circ \bar{\ScT}_{k}$ could be replaced by $\iota \circ \pi_{\R/\Z}$ because of \eqref{ThreeFormulas}, and the image of $(I, R)$ is formed by the pairs $(\alpha, T)$ such that $\alpha \otimes 1 = \bar{\ScT}_{k}^{-1}[T]$. 
\end{Corollary}

Putting altogether, we obtain the following commutative hexagon with exact diagonals, dual to the cohomological one: \pagebreak
\begin{equation}\label{CommHex}
	\resizebox{0.8\textwidth}{!}{
	\xymatrix{
	& \frac{\Tau_{\bullet+1}(X)}{\Tau^{\Int}_{\bullet+1}(X)} \ar[rr]^{\partial} \ar@{^(->}[dr]^{a} & & \Tau^{\Int}_{\bullet}(X) \ar[dr]^{\bar{\ScT}^{-1}} \\
	\frac{H_{\bullet+1}(X; \R)}{H_{\bullet+1}(X)} \ar@{^(->}[ur]^{\bar{\ScT}} \ar@{^(->}[dr]^{\pi_{\R/\Z}} & & \hat{H}_{\bullet}(X) \ar@{->>}[ur]^{R} \ar@{->>}[dr]^{I} & & H_{\bullet}(X; \R) \\
	& H_{\bullet+1}(X; \R/\Z) \ar@{^(->}[ur]^{\iota} \ar[rr]^{-\beta} & & H_{\bullet}(X) \ar[ur]^{- \otimes 1}
}}
\end{equation}
The next corollary will be useful later on.
\begin{Corollary}\label{LemmaH} The following assertions hold true for any $n$-dimensional manifold $X$.
\begin{itemize}
	\item[(i)] The morphism $I \colon \hat{H}_{n}(X) \to H_{n}(X)$ is an isomorphism, which means that a top-degree differential homology class is completely determined by its underlying topological class.
	\item[(ii)] The morphism $\iota \colon H_{0}(X; \R/\Z) \to \hat{H}_{-1}(X)$ is an isomorphism, so $\hat{H}_{-1}(X) \simeq \R/\Z$ provided that $X$ is connected.
\end{itemize}
\end{Corollary}

\section{Products and Poincar\'e Duality}\label{ProdDualitySec}

\SkipTocEntry \subsection{Differential Cap Product}

We describe differential singular cohomology as in \cite{HS}, section 2.3. The group of $k$-cocycles $\hat{Z}^{k}(X)$ is formed by the triples $(\mu^{k}, c^{k-1}, \omega^{k}) \in Z^{k}(X) \times C^{k-1}(X; \R) \times \Omega^{k}_{\Int}(X)$ such that
\begin{equation}\label{FormC}
	\omega^{k} - \mu^{k} = \delta c^{k-1}.
\end{equation}
Here, we implicitly applied the embeddings $\Omega^{k}(X) \hookrightarrow C^{k}(X; \R)$ and $C^{k}(X) \subset C^{k}(X; \R)$. The group of $k$-coboundaries $\hat{B}^{k}(X)$ is the image of
\begin{eqnarray}
	\hat{\delta} \colon \, C^{k-1}(X) \times C^{k-2}(X; \R) & \to & \hat{Z}^{k}(X) \label{DefHatDelta} \\
	(\nu^{k-1}, d^{k-2}) & \mapsto & (\delta \nu^{k-1}, -\nu^{k-1} - \delta d^{k-2}, 0). \nonumber
\end{eqnarray}
The differential cohomology groups are then defined as usual:
\begin{equation}\label{DefDiffCoh}
	\hat{H}^{k}(X) := \hat{Z}^{k}(X)/\hat{B}^{k}(X)
\end{equation}
These groups are naturally isomorphic to the ones of differential characters of the corresponding degree on $X$, since $[\mu^{k}, c^{k-1}, \omega^{k}]$ can be identified with $\xi^{k} \colon Z_{k-1}(X) \to \R/\Z$ such that $\xi^{k}(\gamma_{k-1}) = c^{k-1}(\gamma_{k-1}) \mod \Z$. Coherently, we set $I[\mu^{k}, c^{k-1}, \omega^{k}] := [\mu^{k}]$, $R[\mu^{k}, c^{k-1}, \omega^{k}] := \omega^{k}$, and $a(\eta^{k-1}) := [0, \eta^{k-1}, d\eta^{k-1}]$ for any $\eta^{k-1} \in \Omega^{k-1}(X)$. Moreover, given $[\varphi] \in H^{k-1}(X; \R/\Z)$, we choose any lift $\tilde{\varphi} \in C^{k-1}(X; \R)$ and set $\iota[\varphi] := [-\delta \tilde{\varphi}, \tilde{\varphi}, 0]$.

We recall that the topological cap product satisfies the relation
\begin{equation}\label{BoundaryCap}
	\partial(\gamma_{l} \,\cap\, \varphi^{s}) = (-1)^{s}(\partial \gamma_{l} \,\cap\, \varphi^{s} - \gamma_{l} \,\cap\, \delta \varphi^{s}).
\end{equation}
Furthermore, by setting $\Omega_{h}(X) := \Omega^{-h}(X)$ with boundary $\partial_{h} := (-1)^{h-1}d^{-h}$, we define the chain complex $C_{\bullet}(X; \R) \otimes_{\R} \Omega_{\star}(X)$ with boundary
\begin{align}
	\partial(\gamma_{k} \otimes \omega_{h}) &:= (-1)^{h}\partial \gamma_{k} \otimes \omega_{h} + \gamma_{k} \otimes \partial \omega_{h} \nonumber \\
	& \phantom{:} = (-1)^{h}(\partial \gamma_{k} \otimes \omega_{h} - \gamma_{k} \otimes d\omega_{h}). \label{BoundaryTensor}
\end{align}
The two maps $C_{\bullet}(X; \R) \otimes_{\R} \Omega_{\star}(X) \to \Tau_{\bullet+\star}(X)$ defined respectively by $\gamma_{k} \otimes \omega_{h} \mapsto T_{\gamma_{k}} \wedge \omega_{h}$ and $\gamma_{k} \otimes \omega_{h} \mapsto T_{\gamma_{k} \cap \omega_{h}}$ are chain maps; the explicit computation is shown in appendix \ref{SecChMaps}. They are chain homotopic in an essentially unique way by a standard argument in homological algebra. We thus fix any $B \colon C_{\bullet}(X; \R) \otimes_{\R} \Omega_{\star}(X) \to \Tau_{\bullet+\star+1}(X)$ such that
\begin{equation}\label{HomotopyCapWedge}
	T_{\gamma} \wedge \omega - T_{\gamma \cap \omega} = (B\partial + \partial B)(\gamma \otimes \omega).
\end{equation}

\begin{Def2} We define the \emph{differential cap product}
	\[\hatcap \colon \hat{H}_{l}(X) \times \hat{H}^{s}(X) \to \hat{H}_{l-s}(X)
\]
as follows:
\begin{flalign}\label{DiffCapP}
	& & & \phantom{X} [\gamma_{l}, T_{l+1}] \hatcap [\mu^{s}, c^{s-1}, \omega^{s}] := \bigl[\gamma_{l} \cap \mu^{s}, (-1)^{s}(T_{l+1} \wedge \omega^{s} + T_{\gamma_{l} \cap c^{s-1}}) + B(\gamma_{l} \otimes \omega^{s})\bigr] & & \diamondsuit
\end{flalign}
\end{Def2}
We have to show that \eqref{DiffCapP} is well-defined. It is clearly bi-additive on representatives, hence we just have to verify that, if one of the two classes vanishes, then their product vanishes as well. The reader can find the details in appendix \ref{AppProdWD}, and the proof of the following proposition in appendix \ref{AppProof}.

\begin{Prop}\label{AxiomsDiffCap} The differential cap product \eqref{DiffCapP} satisfies the following properties:
\begin{enumerate}
	\item it is bi-additive;
	\item it is compatible with the natural transformations $I$, $R$, $a$, and $\iota$ as follows:
	\begin{align*}
		& I(\lambda \hatcap \xi) = I(\lambda) \cap I(\xi) & & R(\lambda \hatcap \xi) = R(\lambda) \wedge R(\xi) \\
		& a(T) \hatcap \xi = (-1)^{s}a(T \wedge R(\xi)) & & \lambda \hatcap a(\omega) = (-1)^{s}a(R(\lambda) \wedge \omega) \\
		& \iota(\alpha) \hatcap \xi = (-1)^{s}\iota(\alpha \cap I(\xi)) & & \lambda \hatcap \iota(\rho) = (-1)^{s}\iota(I(\lambda) \cap \rho)
	\end{align*}
	\item denoting by $\hatcup\!$ the usual product in differential cohomology, we have
	\[\lambda \hatcap (\xi \hatcup \chi) = (\lambda \hatcap \xi) \hatcap \chi;
\]
	\item it is functorial, that is, given a smooth function $f \colon X \to Y$, we have
	\[f_{*}(\lambda \hatcap f^{*}\xi) = (f_{*}\lambda) \hatcap \xi.
\]
\end{enumerate}
\end{Prop}
We will show in section \ref{SecUnique} that these properties, stated as axioms, uniquely characterise the cap product. 

\SkipTocEntry \subsection{Differential Poincar\'e Duality}\label{DiffPDSec}

We suppose $X$ compact, oriented, and without boundary. As above, we set $n := \dim(X)$. In this case, a form $\eta^{k} \in \Omega^{k}(X)$ induces the (compactly-supported) current $T_{\eta^{k}} \in \Tau^{k}(X) = \Tau_{n-k}(X)$ defined by $T_{\eta^{k}}(\omega^{n-k}) := \int_{X} \eta^{k} \wedge \omega^{n-k}$. We get the embedding:
\begin{eqnarray}
	\ScT^{k} \colon \, \Omega^{k}(X) & \hookrightarrow & \Tau^{k}(X) \label{EmbJK} \\
	\eta^{k} & \mapsto & T_{\eta^{k}} \nonumber
\end{eqnarray}
Since $X$ has no boundary, we have
\begin{align*}
	dT_{\eta^{k}}(\omega^{n-k-1}) &\overset{\eqref{ExtDiffCurrents}}= (-1)^{k+1}T_{\eta^{k}}(d\omega^{n-k-1}) = (-1)^{k+1}\int_{X} \eta^{k} \wedge d\omega^{n-k-1} \\
	&\,= \int_{X} d\eta^{k} \wedge \omega^{n-k-1} = T_{d\eta^{k}}(\omega^{n-k-1}),
\end{align*}
that is, $dT_{\eta^{k}} = T_{d\eta^{k}}$. Projecting \eqref{EmbJK} to cohomology, we obtain the isomorphism
\begin{equation}\label{BarJIso}
	\bar{\ScT}^{k} \colon H^{k}_{\dR}(X) \overset{\!\simeq}\longrightarrow H\Tau^{k}(X).
\end{equation}
Since $\Tau^{k}(X) = \Tau_{n-k}(X)$, by composing \eqref{BarJIso} with \eqref{BarPsiIso} we get $\bar{\ScT}_{n-k}^{\,-1} \circ \bar{\ScT}^{k} \colon H^{k}_{\dR}(X) \overset{\!\simeq}\longrightarrow H_{n-k}(X; \R)$, which coincides with Poincar\'e Duality in real (co)homology through the de-Rham isomorphism $H^{k}_{\dR}(X) \simeq H^{k}(X; \R)$.

Because of corollary \ref{LemmaH}-(i), the fundamental class $[X] \in H_{n}(X)$ can be identified with $I^{-1}[X] \in \hat{H}_{n}(X)$. Concretely, if $\gamma_{n} \in Z_{n}(X)$ represents $[X]$, then $I^{-1}[X] = [\gamma_{n}, 0]$. From now on we denote $I^{-1}[X]$ simply by $[X]$. We define the \emph{differential Poincar\'e duality} as follows:
\begin{eqnarray}
	\hatPD \colon \, \hat{H}^{k}(X) & \to & \hat{H}_{n-k}(X) \label{DiffPD} \\
	\xi & \mapsto & [X] \hatcap \xi \nonumber
\end{eqnarray}

\begin{Prop}\label{ProofPD} The natural morphism \eqref{DiffPD} is injective and its image is formed by the elements of $\hat{H}_{n-k}(X)$ whose curvature is a form.
\end{Prop}
\begin{proof} The following diagram is commutative:
\begin{equation}\label{DiagPD2}
	\xymatrix{
	0 \ar[r] & H^{k-1}(X; \R/\Z) \ar[r]^(.6){\iota} \ar[d]^(.46){(-1)^{k}\PD}_(.46){\simeq} & \hat{H}^{k}(X) \ar[r]^{R} \ar[d]^(.43){\hatPD} & \Omega^{k}_{\Int}(X) \ar[r] \ar@{^(->}[d]^(.45){\ScT^{k}} & 0 \\
	0 \ar[r] & H_{n-k+1}(X; \R/\Z) \ar[r]^(.6){\iota} & \hat{H}_{n-k}(X) \ar[r]^{R} & \Tau_{n-k}^{\Int}(X) \ar[r] & 0
}\end{equation}
In fact, about the left square, $\hatPD(\iota(\rho^{k-1})) = [X] \hatcap \iota(\rho^{k-1}) = (-1)^{k}\iota([X] \cap \rho^{k-1}) = (-1)^{k}\iota(\PD(\rho^{k-1}))$. About the right square, we first observe that, if $[X] = [\gamma_{n}, 0]$, then $R[X] = T_{\gamma_{n}}$. Therefore,
	\[R[X](\eta^{n}) = \int_{\gamma_{n}} \eta^{n} = \int_{X} \eta^{n} = T_{1}(\eta^{n})
\]
where $1 \in \Omega^{0}(X)$ is the unit constant 0-form, so that $R[X] = T_{1} = \ScT^{0}(1)$. Hence, $R(\hatPD(\xi)) = R([X] \hatcap \xi) = T_{1} \wedge R(\xi) = T_{R(\xi)} = \ScT^{k}(R(\xi))$. The four lemma implies that $\hatPD$ is injective.

The commutativity of the right square of \eqref{DiagPD2} ensures that $R(\hatPD(\xi))$ belongs to the image of $\ScT^{k}$ for every $\xi$, that is, the curvature of $\hatPD(\xi)$ is a form. Conversely, let us fix $\lambda \in \hat{H}_{n-k}(X)$ such that $R(\lambda) = T_{\omega^{k}}$. Because of the surjectivity of $R$, there exists $\xi' \in \hat{H}^{k}(X)$ such that $R(\xi') = \omega^{k}$. Therefore, $R(\hatPD(\xi')) = \ScT^{k}(\omega^{k}) = T_{\omega^{k}}$ by commutativity. It follows that there exists $\alpha \in H_{n-k+1}(X; \R/\Z)$ such that $\lambda - \hatPD(\xi') = \iota(\alpha)$. By setting $\rho := (-1)^{k}\PD^{-1}(\alpha)$ and $\xi := \xi' + \iota(\rho)$, we obtain $\hatPD(\xi) = \hatPD(\xi') + \iota(\alpha) = \lambda$.
\end{proof}

\begin{Rmk*} In the previous proof, instead of using diagram \eqref{DiagPD2} to show injectivity, one could alternatively start from:
\begin{equation}\label{DiagPD1}
	\xymatrix{
	0 \ar[r] & \frac{\Omega^{k-1}(X)}{\Omega^{k-1}_{\Int}(X)} \ar[r]^{a} \ar@{^(->}[d]^(.45){(-1)^{k}\ScT^{k-1}} & \hat{H}^{k}(X) \ar[r]^{I} \ar[d]^(.43){\hatPD} & H^{k}(X) \ar[r] \ar[d]^(.46){\PD}_(.46){\simeq} & 0 \\
	0 \ar[r] & \frac{\Tau_{n-k+1}(X)}{\Tau_{n-k+1}^{\Int}(X)} \ar[r]^{a} & \hat{H}_{n-k}(X) \ar[r]^{I} & H_{n-k}(X) \ar[r] & 0
}\end{equation}
\end{Rmk*}

In order to turn \eqref{DiffPD} into an isomorphism, we have two possibilities: (i) restricting $\hat{H}_{\bullet}(X)$ to classes whose curvature is a form; (ii) extending $\hat{H}^{\bullet}(X)$ to classes whose curvature is a current. The first case is straightforward: we define the \emph{restricted differential homology groups} as $\check{H}_{k}(X) := R^{-1}(\IIm\,\ScT^{n-k})$ and, by restricting the codomain of the injective map \eqref{DiffPD} to its image, we obtain the isomorphism $\hatPD \colon \, \hat{H}^{k}(X) \to \check{H}_{n-k}(X)$. The second case, based on the next definition, is more meaningful.
\begin{Def}\label{DefRedHom} The \emph{extended differential cohomology group} of degree $k$ on $X$, denoted by $\check{H}^{k}(X)$, is the quotient of $\hat{H}^{k}(X) \times \Tau^{k-1}(X)$ by the subgroup whose elements are of the form
\begin{equation}\label{EqRelDiffHomCurrents}
	\bigl(a(\omega^{k-1}), -T_{\omega^{k-1}} - dT^{k-2}\bigr)
\end{equation}
with $\omega^{k-1} \in \Omega^{k-1}(X)$ and $T^{k-2} \in \Tau^{k-2}(X)$.
\end{Def}
This means that $[a(\omega^{k-1}), 0] = [0, T_{\omega^{k-1}}]$ and $[0, dT^{k+2}] = 0$. In the setting of restricted homology, the natural transformations $I$, $R$, and $a$ are well-defined and satisfy the same axioms replacing currents with forms. In the framework of extended cohomology, we define
\begin{align*}
	& I \colon \check{H}^{k}(X) \to H^{k}(X), & & [\xi^{k}, T^{k-1}] \mapsto I(\xi^{k}) \\
	& R \colon \check{H}^{k}(X) \to \Tau^{k}(X), & & [\xi^{k}, T^{k-1}] \mapsto R(\xi^{k}) + dT^{k-1} \\
	& a \colon \Tau^{k-1}(X) \to \check{H}_{k}(X), & & T^{k-1} \mapsto [0, T^{k-1}]
\end{align*}
Again, these transformations satisfy all of the usual axioms replacing forms with currents, and we have the natural embedding $\hat{H}^{k}(X) \hookrightarrow \check{H}^{k}(X)$, $\xi^{k} \mapsto [\xi^{k}, 0]$, which is compatible with $I$, $R$, and $a$.

Given a smooth map $f \colon X \to Y$, we can neither restrict the pushforward $f_{*} \colon \hat{H}_{k}(X) \to \hat{H}_{k}(Y)$ nor extend the pull-back $f^{*} \colon \hat{H}^{k}(Y) \to \hat{H}^{k}(X)$, since we neither have a natural pushforward of forms nor a natural pull-back of currents. Nevertheless, functoriality holds with respect to the Gysin map, as shown in section \ref{GysinMapHomSec} below (for any generalised homology theory). Moreover, we have the cap product
	\[\hat{\cap} \colon \check{H}_{l}(X) \times \check{H}^{s}(X) \to \hat{H}_{l-s}(X)
\]
defined by
\begin{equation}\label{DiffCapP2}
	\lambda \hatcap [\xi, T] := \lambda \hatcap \xi + (-1)^{s}a(R(\lambda) \wedge T).
\end{equation}
Both \eqref{DiffCapP} and \eqref{DiffCapP2} restrict to a product from $\check{H}_{l}(X) \times \hat{H}^{s}(X)$ to $\check{H}_{l-s}(X)$. However, we have no product between $\hat{H}_{l}(X)$ and $\check{H}^{s}(X)$ since we cannot multiply currents.

The fundamental class $[X]$ belongs to reduced homology since its curvature is the unit 0-form. Hence, we can use \eqref{DiffCapP2} to define $\hatPD \colon \check{H}^{k}(X) \overset{\!\simeq}\longrightarrow \hat{H}_{n-k}(X)$. It is an isomorphism by the five lemma, since diagrams \eqref{DiagPD2} and \eqref{DiagPD1} become respectively
\begin{equation}\label{DiagPD2B}
	\xymatrix{
	0 \ar[r] & H^{k-1}(X; \R/\Z) \ar[r]^(.6){\iota} \ar[d]^(.46){(-1)^{k}\PD}_(.46){\simeq} & \hat{H}^{k}(X) \ar[r]^{R} \ar[d]^(.43){\hatPD} & \Tau^{k}_{\Int}(X) \ar[r] \ar@{=}[d] & 0 \\
	0 \ar[r] & H_{n-k+1}(X; \R/\Z) \ar[r]^(.6){\iota} & \hat{H}_{n-k}(X) \ar[r]^{R} & \Tau_{n-k}^{\Int}(X) \ar[r] & 0
}\end{equation}
and
\begin{equation}\label{DiagPD1B}
	\xymatrix{
	0 \ar[r] & \frac{\Tau^{k-1}(X)}{\Tau^{k-1}_{\Int}(X)} \ar[r]^{a} \ar@{=}[d] & \hat{H}^{k}(X) \ar[r]^{I} \ar[d]^(.43){\hatPD} & H^{k}(X) \ar[r] \ar[d]^(.46){\PD}_(.46){\simeq} & 0 \\
	0 \ar[r] & \frac{\Tau_{n-k+1}(X)}{\Tau_{n-k+1}^{\Int}(X)} \ar[r]^{a} & \hat{H}_{n-k}(X) \ar[r]^{I} & H_{n-k}(X) \ar[r] & 0.
}\end{equation}

\SkipTocEntry \subsection{Duality in Low Cohomological Degree}\label{LowDimSec}

A differential cohomology class of degree one or two has a quite natural geometric interpretation. Here we provide a description of its dual homology class with a similar language, deferring proofs and generalization to any degree to a future paper.

\vspace{3pt} \emph{Degree 2.} A geometric model for the group $\hat{H}^{2}(X)$ consists of Hermitian line bundles with connection up to isomorphism. Therefore, fixing a line bundle $L$ with connection $\nabla$, we aim to describe the dual class $\widehat{\PD}[L, \nabla] \in \hat{H}_{n-2}(X)$.\footnote{We thank Prof.\ Dirk T\"oben for rising this question.} We fix a global section $s \colon X \to L$ transversal to the zero-section. The corresponding zero-locus $Z := s^{-1}(0)$ is a 2-codimensional submanifold of $X$. Denoting by $[Z] \in H_{n-2}(Z)$ its fundamental class, $i \colon Z \hookrightarrow X$ the inclusion, and $[L] \in H^{2}(X)$ the first Chern class of $L$, the \emph{Gauss-Bonnet formula}
	\[\PD[L] = i_{*}[Z]
\]
holds.\footnote{For instance, see \cite{GH}, p.\ 413, Gauss-Bonnet Formula II with $k = r = 1$.} In $X \setminus Z$, the connection $\nabla$ and the non-vanishing section $s$ induce the 1-form $A$ defined by $\nabla_{V}s = 2\pi i A(V)s$ for any section $V$ of the tangent bundle. The form $A$ induces the current $\bar{A} \in \Tau_{n-1}(X)$ defined by $\bar{A}(\omega) := \int_{X \setminus Z} A \wedge \omega$. By choosing any representative $Z$ of the fundamental class, we obtain the \emph{differential Gauss-Bonnet Formula} for line bundles:
\begin{equation}\label{DiffGB}
	\hatPD[L, \nabla] = [i_{*}Z, \bar{A}]
\end{equation}
We set $\delta_{Z} := T_{i_{*}Z} \in \Tau_{n-2}(X)$, that is, $\delta_{Z}(\omega) = \int_{Z} \omega$. The main issue for the proof consists in showing that $d\bar{A} = F - \delta_{Z}$, so that $R[i_{*}Z, \bar{A}] = \delta_{Z} + d\bar{A} = F = R[L, \nabla]$. The compatibility with $I$ and $a$ is straightforward, hence the result follows from the argument of \cite[Fact 1.2]{SS}.

\vspace{3pt} \emph{Degree 1.} An element of $\hat{H}^{1}(X)$ can be identified with a smooth function $f \colon X \to S^{1}$. With this language, $I(f) := [f]$ is the corresponding homotopy class, which is also a cohomology class because $S^{1}$ is the classifying space for $H^{1}$. It follows that $I(f) = 0$ if and only if $f$ admits a global logarithm, that is, there exists $\phi \colon X \to \R$ such that $f = \exp \circ \, \phi$ where $\exp \colon \R \to S^{1} \subset \C$, $t \mapsto e^{2\pi i t}$. We also set $R(f) := \frac{1}{2\pi i} f^{-1}df$ and $a(\phi) := \exp \circ \, \phi$. Let us fix a regular value $z_{0} \in S^{1}$ of $f \colon X \to S^{1}$. The $z_{0}$-locus $Z := f^{-1}(z_{0})$ is a 1-codimensional submanifold of $X$ and we have:
	\[\PD[f] = i_{*}[Z]
\]
The class $[f]$ is trivial in the complement of $Z$, hence there exists $\phi \colon X \setminus Z \to \R$ such that $f = \exp \circ \, \phi$. The 0-form $\phi$ induces the current $\bar{\phi} \in \Tau_{n}(X)$ defined by $\bar{\phi}(\omega) := \int_{X \setminus Z} \phi \, \omega$. Again, by choosing any representative $Z$ of the fundamental class, we obtain:
\begin{equation}\label{DiffGBFunc}
	\hatPD(f) = [i_{*}Z, \bar{\phi}]
\end{equation}
The proof is similar to the previous case.

\vspace{3pt} \emph{Other degrees.} We believe there exists an analogous interpretation for any degree $k$ between $3$ and $n$ by using the language of abelian $(k-2)$-gerbes with connection. In the extreme cases $k = 0$ and $k = n+1$, the duality coincides with the underlying topological one, with coefficients respectively in $\Z$ and $\R/\Z$, because of corollary \ref{LemmaH} and its cohomological counterpart.

\SkipTocEntry \subsection{External Products and $S^{1}$-Differentiation}

We recall that the Eilenberg-Zilber map $EZ \colon C_{\bullet}(X) \times C_{\star}(Y) \to C_{\bullet+\star}(X \times Y)$ is a chain map, therefore it induces the exterior product $H_{\bullet}(X) \times H_{\star}(Y) \to H_{\bullet+\star}(X \times Y)$. One can define the analogous product of currents $\Tau_{\bullet}(X) \times \Tau_{\star}(Y) \to \Tau_{\bullet+\star}(X \times Y)$ as follows: given $T_{l} \in \Tau_{l}(X)$, $T_{s} \in \Tau_{s}(Y)$, $\omega^{k} \in \Omega^{k}(X)$, and $\omega^{h} \in \Omega^{h}(Y)$, we set $\omega^{k} \times \omega^{h} := \pi_{X}^{*}\omega^{k} \wedge \pi_{Y}^{*}\omega^{h}$ and
	\[(T_{l} \times T_{s})(\omega^{k} \times \omega^{h}) := \begin{cases} T_{l}(\omega^{k}) \cdot T_{s}(\omega^{h}) & \textnormal{if } $l = k$ \textnormal{ and } $s = h$ \\ 0 & \textnormal{otherwise}. \end{cases}
\]
Since the subspace of split forms is dense in $\Omega^{l+s}(X \times Y)$, the current $T_{l} \times T_{s}$ is uniquely defined by continuity.

\begin{Def2} The \emph{differential external product}
	\[\hat{\times} \colon \hat{H}_{l}(X) \times \hat{H}_{s}(Y) \to \hat{H}_{l+s}(X \times Y)
\]
is defined as follows:
\begin{flalign}
	& & \hspace{10pt} [\gamma_{l}, T_{l+1}] \hattimes [\gamma_{s}, T_{s+1}] := \bigl[\gamma_{l} \times \gamma_{s}, (-1)^{l}T_{\gamma_{l}} \times T_{s+1} + T_{l+1} \times T_{\gamma_{s}} + T_{l+1} \times \partial T_{s+1} \bigr] & & & \diamondsuit \label{DiffExtPr}
\end{flalign}
\end{Def2}

We also have the slant product $/ \colon H_{\bullet+\star}(X \times Y) \times H_{\bullet}(Y) \to C_{\star}(X)$ defined by $\alpha / \rho := (\pi_{X})_{*}(\alpha \cap \pi_{Y}^{*}\rho)$. Similarly, $/ \colon \Tau_{\bullet+\star}(X \times Y) \times \Omega_{\bullet}(Y) \to \Tau_{\star}(X)$ can be defined by $T/\omega := (\pi_{X})_{*}(T \wedge \pi_{Y}^{*}\omega)$, that is, $(T/\omega)(\eta) := (T \wedge \pi_{Y}^{*}\omega)(\pi_{X}^{*}\eta) = T(\eta \times \omega)$.

\begin{Def} The \emph{differential slant product}
	\[\hatslant \colon \hat{H}_{l}(X \times Y) \times \hat{H}^{s}(Y) \to \hat{H}_{l-s}(X)
\]
is defined by $\lambda \hatslant \xi := (\pi_{X})_{*}(\lambda \hatcap \pi_{Y}^{*}\xi)$.
\end{Def}

The proof of the following proposition is left to the reader.

\begin{Prop}\label{PropExtSlant} The differential external and slant products satisfy the following axioms:
\begin{enumerate}
	\item they are bi-additive;
	\item they are compatible with the natural transformations $I$, $R$, $a$, and $\iota$ as follows:
	\begin{align*}
		& I(\lambda \hattimes \mu) = I(\lambda) \times I(\mu) & & I(\lambda \hatslant \xi) = I(\lambda) / I(\xi) \\
		& R(\lambda \hattimes \mu) = R(\lambda) \times R(\mu) & & R(\lambda \hatslant \xi) = R(\lambda) / R(\xi) \\
		& a(T) \hattimes \mu = a(T \times R(\mu)) & & a(T) \hatslant \mu = a(T / R(\mu)) \\
		& \lambda \hattimes a(\omega) = (-1)^{l}a(R(\lambda) \times \omega) & & \lambda \hatslant a(\omega) = (-1)^{l}a(R(\lambda) / \omega) \\
		& \iota(\alpha) \hattimes \mu = \iota(\alpha \times I(\mu)) & & \iota(\alpha) \hatslant \mu = \iota(\alpha / I(\mu)) \\
		& \lambda \hattimes \iota(\rho) = (-1)^{l}\iota(I(\lambda) \times \rho) & & \lambda \hatslant \iota(\rho) = (-1)^{l}\iota(I(\lambda) / \rho)
	\end{align*}
	\item the external product is associative and compatible with cohomology, and the slant product is mixed-associative with cohomology, that is:
	\begin{align*}
		& (\lambda \hattimes \mu) \hatcap (\xi \hattimes \chi) = (\lambda \hatcap \xi) \hattimes (\mu \hatcap \chi) & & (\lambda \hatcap \pi_{Y}^{*}\xi) \hatslant \chi = \lambda \hatslant (\xi \hatcup \chi)
	\end{align*}
	\item they are functorial, that is, given smooth functions $f \colon X \to X'$ and $g \colon Y \to Y'$, we have
	\begin{align*}
		& (f, g)_{*}(\lambda \hattimes \mu) = f_{*}\lambda \hattimes g_{*}\mu & & f_{*}(\lambda \hatslant g^{*}\mu) = (f, g)_{*}\lambda \hatslant \mu.
	\end{align*}
\end{enumerate}
\end{Prop}

In cohomology, we have the integration map over $S^{1}$. In homology, as an application of the external product just defined, we have the following dual map, that we call \emph{$S^{1}$-differentiation}:
\begin{eqnarray}
	\partial_{S^{1}} \colon \hat{H}_{\bullet}(X) & \to & \hat{H}_{\bullet+1}(S^{1} \times X) \label{DiffS1} \\
	\lambda & \mapsto & [S^{1}] \hattimes \lambda \nonumber
\end{eqnarray}
We define the underlying topological map with coefficients in $\Z$ or $\R/\Z$ in the analogous way, and the one on currents as follows: $\partial_{S^{1}} \colon \Tau_{\bullet}(X) \to \Tau_{\bullet+1}(S^{1} \times X)$, $T \mapsto T \circ \int_{S^{1}}$.

\begin{Prop}\label{PropS1Diff} The following properties of $S^{1}$-differentiation hold:
\begin{enumerate}
	\item it commutes with the transformations $I$, $R$, $a$, and $\iota$;
	\item considering the map $t \colon S^{1} \to S^{1}$, $e^{i\theta} \mapsto e^{-i\theta}$, we have $(t \times \id_{X})_{*} \circ \partial_{S^{1}} = -\partial_{S^{1}}$;
	\item considering the projection $\pi \colon S^{1} \times X \to X$, we have $\pi_{*} \circ \partial_{S^{1}} = 0$;
	\item for every $\lambda \in \hat{H}_{\bullet}(X)$ and $\xi \in \hat{H}^{\star}(S^{1} \times X)$:
	\begin{equation}\label{IntDer}
		\textstyle \lambda \hatcap \int_{S^{1}} \xi = \pi_{*}(\partial_{S^{1}}\lambda \hatcap \xi)
	\end{equation}
	\item if $X$ is oriented of dimension $n$, then the following diagrams commute:
	\[\resizebox{0.9\textwidth}{!}{\xymatrix{
	\hat{H}^{\bullet}(S^{1} \times X) \ar[rr]^{\hatPD} & & \hat{H}_{n-\bullet+1}(S^{1} \times X) \\
	\hat{H}^{\bullet}(X) \ar[rr]^{\hatPD} \ar[u]^{\pi^{*}} & & \hat{H}_{n-\bullet}(X) \ar[u]_{\partial_{S^{1}}}
} \qquad\quad \xymatrix{
	\hat{H}^{\bullet+1}(S^{1} \times X) \ar[rr]^{\hatPD} \ar[d]_{\int_{S^{1}}} & & \hat{H}_{n-\bullet}(S^{1} \times X) \ar[d]^{\pi_{*}} \\
	\hat{H}^{\bullet}(X) \ar[rr]^{\hatPD} & & \hat{H}_{n-\bullet}(X)
}}\]
\end{enumerate}
\end{Prop}

The proof can be found in appendix \ref{AppProofS1}. Furthermore, the map \eqref{DiffS1} will be extended later on to differentiation with respect to an oriented fibre bundle, dual to the corresponding integration (see section \ref{GysinMapHomSec}).

\section{Uniqueness}\label{SecUnique}

Let us show that differential homology and the cap product are completely characterised by propositions \ref{PropAxioms} and \ref{AxiomsDiffCap} stated as axioms.

\SkipTocEntry \subsection{Uniqueness of Differential Homology}

We denote by $\hat{H}'_{\bullet}$ any other model of differential singular homology, endowed with the corresponding transformations $I'$, $R'$, and $a'$. Given a class $\lambda_{k} \in \hat{H}'_{k}(X)$, we fix any cycle $\gamma_{k}$ representing $I'(\lambda_{k}) \in H_{k}(X)$. Following \cite[Fact 2.1]{SS}, we choose a \emph{good} neighbourhood $U$ of the image of $\gamma_{k}$, that is, a neighbourhood with $H_{h}(U) = 0$ for every $h > k$. Because of the exact sequence \eqref{ExSeq3}, the intersection between the kernels of $I'$ and $R'$ in $U$ is isomorphic to $H_{k+1}(U; \R)/H_{k+1}(U) = 0$. Hence, there exists a unique class $\hat{\gamma}_{k} \in H'_{k}(U)$ such that $I'(\hat{\gamma}_{k}) = [\gamma_{k}]_{U}$ and $R'(\hat{\gamma}_{k}) = T_{\gamma_{k}}$. Denoting by $i \colon U \hookrightarrow X$ the inclusion, we have $I(i_{*}\hat{\gamma}_{k}) = i_{*}[\gamma_{k}]_{U} = [\gamma_{k}] = I(\lambda_{k})$. Thus, there exists $T_{k+1} \in \Tau_{k+1}(X)$, unique up to integral currents, such that
\begin{equation}\label{FormulaLambdaK}
	\lambda_{k} = i_{*}\hat{\gamma}_{k} + a'(T_{k+1}).
\end{equation}
We obtain the following morphism, which will be proved to be bijective and natural:
\begin{eqnarray}
	\phi \colon \hat{H}'_{k}(X) & \overset{\!\simeq}\longrightarrow & \hat{H}_{k}(X) \label{IsoUniqueness} \\
	\lambda_{k} & \mapsto & [\gamma_{k}, T_{k+1}] \nonumber
\end{eqnarray}
Let us show that \eqref{IsoUniqueness} does not depend on the choices of $\gamma_{k}$ and $U$. With respect to $U$, if $U'$ is another choice, then it is easy to see that the result does not change by fixing a good neighbourhood $U'' \subset U \cap U'$, which is allowed by \cite[Fact 2.1]{SS}. Let us now suppose that $I'(\lambda_{k}) = [\gamma_{k}] = [\gamma'_{k}]$. We fix a good neighbourhood $U$ of $\gamma_{k} - \gamma'_{k}$. In $U$, the cycle $\gamma_{k} - \gamma'_{k}$ is homologous to a pseudo-manifold $P_{k}$ by \cite[Fact 2.2]{SS}. We set
\begin{equation}\label{FirstEqPseudoM}
	\gamma_{k} - \gamma'_{k} = P_{k} + \partial \Theta_{k+1}
\end{equation}
in $U$, so that $\hat{\gamma}_{k} - \hat{\gamma}'_{k} = \hat{P}_{k} + a(T_{\Theta_{k+1}})$, where $a(T_{\Theta_{k+1}})$ is the unique class in the kernel of $I'$ with curvature $T_{\partial \Theta_{k+1}}$. Therefore:
\begin{equation}\label{SecondEqPseudoM}
	i_{*}\hat{\gamma}_{k} - i_{*}\hat{\gamma}'_{k} = i_{*}\hat{P}_{k} + a(T_{\Theta_{k+1}})
\end{equation}
Since $\gamma_{k} - \gamma'_{k}$ is a boundary, $P_{k}$ is a boundary as well. By \cite[Fact 2.3]{SS}, there exists a good neighbourhood $V$ of $P_{k}$ is which $P_{k}$ bounds. We set $P_{k} = \partial \Gamma_{k+1}$ in $V$, so that $\hat{P}_{k} = a(T_{\Gamma_{k+1}})$, again because this is the unique class in the kernel of $I'$ with curvature $T_{\partial \Gamma_{k+1}}$. Since the choice of a good neighbourhood of $P_{k}$ is immaterial, by using $V$ instead of $U$ in formula \eqref{SecondEqPseudoM} we obtain
\begin{equation}\label{ThirdEqPseudoM}
	i_{*}\hat{\gamma}_{k} - i_{*}\hat{\gamma}'_{k} = a(T_{\Gamma_{k+1} + \Theta_{k+1}}).
\end{equation}
Formula \eqref{FormulaLambdaK}, with $\gamma'_{k}$ instead of $\gamma_{k}$, becomes $\lambda_{k} = i_{*}\hat{\gamma}'_{k} + a'(T'_{k+1})$ for a suitable $T'_{k+1}$. It follows that $i_{*}\hat{\gamma}_{k} - i_{*}\hat{\gamma}'_{k} = a'(T_{k+1} - T'_{k+1})$. Formula \eqref{ThirdEqPseudoM} implies
\begin{equation}\label{EqTBoundary}
	a'(T'_{k+1}) = a'(T_{k+1} + T_{\Gamma_{k+1} + \Theta_{k+1}}).
\end{equation}
By definition \eqref{IsoUniqueness}, we have both $\phi(\lambda_{k}) = [\gamma_{k}, T_{k+1}]$ and $\phi(\lambda_{k}) = [\gamma'_{k}, T'_{k+1}]$. Coherently:
\begin{align*}
	[\gamma'_{k}, T'_{k+1}] & \overset{\eqref{FirstEqPseudoM}, \eqref{EqTBoundary}}= [\gamma_{k} - P_{k} - \partial \Theta_{k+1}, T_{k+1} + T_{\Gamma_{k+1} + \Theta_{k+1}}] \\
	&\hspace{12pt} = \hspace{11pt} [\gamma_{k} - \partial(\Gamma_{k+1} - \Theta_{k+1}), T_{k+1} + T_{\Gamma_{k+1} + \Theta_{k+1}}] = [\gamma_{k}, T_{k+1}].
\end{align*}
This show that \eqref{IsoUniqueness} is well-defined as a function. About additivity, we fix $\lambda_{k}, \mu_{k} \in H'_{k}(X)$ and we set $I'(\lambda_{k}) = [\gamma_{k}]$ and $I'(\mu_{k}) = [\sigma_{k}]$. We fix a good neighbourhood $U$ of $\gamma_{k} + \sigma_{k}$, so that $\lambda_{k} = i_{*}\hat{\gamma}_{k} + a(T_{k+1})$ and $\mu_{k} = i_{*}\hat{\sigma}_{k} + a(U_{k+1})$ for suitable $T_{k+1}$ and $U_{k+1}$. It follows that $\lambda_{k} + \mu_{k} = i_{*}(\hat{\gamma}_{k} + \hat{\sigma}_{k}) + a(T_{k+1} + U_{k+1})$. Trivially $\widehat{\gamma_{k} + \sigma_{k}} = \hat{\gamma}_{k} + \hat{\sigma}_{k}$, thus $T_{k+1} + U_{k+1}$ fits in formula \eqref{FormulaLambdaK} about $\lambda_{k} + \mu_{k}$. Hence, definition \eqref{IsoUniqueness} implies $\phi(\lambda_{k} + \mu_{k}) = [\gamma_{k} + \sigma_{k}, T_{k+1} + U_{k+1}] = \phi(\lambda_{k}) + \phi(\sigma_{k})$.

It easily follows from the construction of $\phi$ that it commutes with the natural transformations $I$, $R$, and $a$. Therefore, the sequence \eqref{ExSeq1} and the five lemma imply that it is an isomorphism. About naturality, given a smooth function $f \colon X \to Y$ and a class $\lambda_{k} \in \hat{H}'_{k}(X)$ satisfying $I'(\lambda_{k}) = [\gamma_{k}]$, we fix two good neighbourhoods $U$ of $f_{*}\gamma_{k}$ and $V \subset f^{-1}U$ of $\gamma_{k}$. Considering the restriction $f \colon V \to U$, we have $f_{*}\hat{\gamma}_{k} = \widehat{f_{*}\gamma_{k}}$. Moreover, $\lambda_{k} = i_{*}\hat{\gamma}_{k} + a(T_{k+1})$ implies
	\[f_{*}\lambda_{k} = f_{*}i_{*}\hat{\gamma}_{k} + a(f_{*}T_{k+1}) = j_{*}f_{*}\hat{\gamma}_{k} + a(f_{*}T_{k+1}) = j_{*}\widehat{f_{*}\gamma_{k}} + a(f_{*}T_{k+1})
\]
where $j \colon V \hookrightarrow Y$ is the inclusion. Hence, $a(f_{*}T_{k+1})$ fits in formula \eqref{FormulaLambdaK} about $f_{*}\lambda_{k}$, so that $\phi(f_{*}\lambda_{k}) = [f_{*}\gamma_{k}, f_{*}T_{k+1}] = f_{*}\phi(\lambda_{k})$.

About uniqueness, let $\phi' \colon \hat{H}'_{k}(X) \to \hat{H}_{k}(X)$ be any natural morphism commuting with $I$, $R$, and $a$. Using the notation of formula \eqref{FormulaLambdaK}, we have $\phi'(\lambda_{k}) = \phi'(i_{*}\hat{\gamma_{k}} + a'(T_{k+1})) = i_{*}\phi'(\hat{\gamma}_{k}) + a(T_{k+1})$. The class $\phi'(\hat{\gamma}_{k})$ is the unique one satisfying $I(\phi'(\hat{\gamma}_{k})) = [\gamma_{k}]_{U}$ and $R(\phi'(\hat{\gamma}_{k})) = T_{\gamma_{k}}$, so that $\phi'(\hat{\gamma}_{k}) = [\gamma_{k}, 0]_{U}$. Thence, $\phi'(\lambda_{k}) = i_{*}[\gamma_{k}, 0] + a(T_{k+1}) = [\gamma_{k}, T_{k+1}] = \phi(\lambda_{k})$.

Lastly, we show that $\phi$ commutes with the natural transformation $\iota$. Since the latter was not used neither in the construction of $\phi$ nor in the proof of its uniqueness, this means that $\iota$ is completely determined by $I$, $R$, and $a$. Indeed, let us apply formula \eqref{FormulaLambdaK} to $\lambda_{k} = \iota'([\Gamma_{k+1}])$ where $\Gamma_{k+1} \in Z_{k+1}(X; \R/\Z)$. We lift $\Gamma_{k+1}$ to $\tilde{\Gamma}_{k+1} \in C_{k+1}(X; \R)$, so that $\gamma_{k} := -\partial \tilde{\Gamma}_{k+1} \in Z_{k}(X)$ and $I'(\lambda_{k}) = [\gamma_{k}]$. The class $[\gamma_{k}]$ is torsion, and we call $r$ its order. According to \cite[Fact 2.3]{SS}, by adding a suitable integral chain to $\tilde{\Gamma}_{k+1}$, we are allowed to suppose that $\gamma_{k}$ is the fundamental cycle of an embedded pseudo-manifold. Also, we can assume without loss of generality that $\dim(X) \geqslant k+2$.\footnote{When $\dim(X) = k+1$, we have two possibilities: if $X$ is compact and orientable, then $H_{k}(X)$ has no torsion, hence $\iota$ is completely determined by $a$; otherwise, $H_{k+1}(X; \R) = 0$, therefore $\iota$ is completely determined by $I$ and $R$.} Then, the same proof of \cite[Fact 2.3]{SS} shows that there exists a good neighbourhood $U$ of $\gamma_{k}$ in which $r\gamma_{k}$ bounds, thus we set $r\gamma_{k} = \partial \Theta_{k+1}$ in $U$. We have
	\[\textstyle \hat{\gamma}_{k} = \iota'([-\frac{1}{r}\Theta_{k+1}]_{\R/\Z, U}) + a'(T_{\frac{1}{r}\Theta_{k+1}}),
\]
since $I'(\hat{\gamma}_{k}) = [\partial(\frac{1}{r}\Theta_{k+1})]_{U} = [\gamma_{k}]_{U}$ and $R'(\hat{\gamma_{k}}) = 0 + \partial T_{\frac{1}{r}\Theta_{k+1}} = T_{\gamma_{k}}$. Hence, the naturality of $\iota$ implies $i_{*}\hat{\gamma}_{k} = \iota'([-\frac{1}{r}\Theta_{k+1}]_{\R/\Z}) + a'(T_{\frac{1}{r}\Theta_{k+1}})$. Formula \eqref{FormulaLambdaK} then becomes
	\[\textstyle \iota'([\Gamma_{k+1}]) = \iota'([-\frac{1}{r}\Theta_{k+1}]_{\R/\Z}) + a'(T_{\frac{1}{r}\Theta_{k+1}} + T_{k+1})
\]
which is equivalent to $\iota'([\tilde{\Gamma}_{k+1}+\frac{1}{r}\Theta_{k+1}]_{\R/\Z}) = a'(T_{\frac{1}{r}\Theta_{k+1}} + T_{k+1})$. Since $\tilde{\Gamma}_{k+1} + \frac{1}{r}\Theta_{k+1}$ is a real cycle, we obtain $a'(T_{\tilde{\Gamma}_{k+1}+\frac{1}{r}\Theta_{k+1}}) = a'(T_{\frac{1}{r}\Theta_{k+1}} + T_{k+1})$, that reduces to $a'(T_{\tilde{\Gamma}_{k+1}}) = a'(T_{k+1})$. Summarizing, in formula \eqref{IsoUniqueness}, we have $\gamma_{k} = -\partial \tilde{\Gamma}_{k+1}$ and $a'(T_{k+1}) = a'(T_{\tilde{\Gamma}_{k+1}})$. Thence:
	\[\phi(\iota'([\Gamma_{k+1}])) = [-\partial \tilde{\Gamma}_{k+1}, T_{\tilde{\Gamma}_{k+1}}] = \iota([\Gamma_{k+1}])
\]

\SkipTocEntry \subsection{Uniqueness of the Cap Product}

Let us suppose that $\hatcap'$ is another cap product. Because of the compatibility with $R$, for every $\lambda_{l} \in \hat{H}_{l}(X)$ and $\xi^{s} \in \hat{H}^{s}(X)$, there exists $\rho_{l-s+1} \in H_{l-s+1}(X; \R/\Z)$ satisfying
	\[\lambda_{l} \hatcap \xi^{s} - \lambda_{l} \hatcap' \xi^{s} = \iota(\rho_{l-s+1}).
\]
Moreover, because of the compatibility with $a$, the class $\rho_{l-s+1}$ only depends on $I(\lambda_{l})$ and $I(\xi^{s})$, thus we obtain the natural morphism
	\[\Delta \colon H_{\bullet}(X) \times H^{\star}(X) \to H_{\bullet-\star+1}(X; \R/\Z).
\]
The thesis consists in showing that $\Delta$ necessarily vanishes. We recall that a cohomology class $[f] \in H^{s}(X)$ admits a representative $f$ whose image is contained in a finite skeleton of the Eilenberg-MacLance space $K(\Z, s)$. By embedding this skeleton in a Euclidean space and forming a regular neighbourhood, the codomain of $f$ is replaced by a smooth manifold and $f$ can be assumed to be smooth. In this way, we can use the Eilenberg-MacLance spaces as if they belonged to the category of smooth manifolds.\footnote{For instance, see \cite[Fact 1.1]{SS}.} We call $\theta^{s}$ the fundamental class of $K(\Z, s)$, so that $[f] = f^{*}\theta^{s}$. We also recall that homology can be expressed as follows:\footnote{For instance, see \cite[$\mathcal{x}$8.33]{Switzer}.}
	\[H_{l}(X) = \varinjlim \pi_{l+q} \bigl(X_{+} \wedge K(\Z, q)\bigr)
\]
Here $X_{+}$ denotes the union of $X$ and a separated basepoint. This means that a class $[f] \in H_{l}(X)$ is represented by a function $f \colon S^{l+q} \to X_{+} \wedge K(\Z, q)$, and $[f] = f_{*}1_{l+q}/\theta^{q}$ where $1_{l+q}$ is the canonical generator of $H_{l+q}(S^{l+q})$. The space $X_{+} \wedge K(\Z, q)$ is a finite-dimensional countable CW-complex, hence it can be replaced by a smooth manifold as above. By naturality, 
\begin{align*}
	\Delta([f], [g]) &= \Delta(f_{*}1_{l+q}/\theta^{q}, g^{*}\theta^{s}) = \Delta(f_{*}1_{l+q}, \theta^{q} \cdot g^{*}\theta^{s}) = f_{*}\Delta(1_{l+q}, f^{*}(\theta^{q} \cdot g^{*}\theta^{s}))
\end{align*}
for every $[f] \in H_{l}(X)$ and $[g] \in H^{s}(X)$, where $\theta^{q} \cdot g^{*}\theta^{s}$ is an instance of the product $\tilde{H}^{q}(K(\Z, q)) \times H^{s}(X) \to \tilde{H}^{q+s}(X_{+} \wedge K(\Z, q))$. The last term shows that the morphism $\Delta$ is completely determined by the particular cases
	\[\Delta \colon H_{l+q}(S^{l+q}) \times H^{s+q}(S^{l+q}) \to H_{l-s+1}(S^{l+q}; \R/\Z).
\]
The group $H_{s+q}(S^{l+q})$ is non-vanishing only in the cases $s+q = 0$ and $s = l$. By choosing $q > 1$ in the direct limit, the former case is excluded, and the latter leads to $H_{1}(S^{l+q}; \R/\Z) = 0$.

\begin{Rmk*} The arguments of uniqueness become more natural by extending the theory to stratifolds, as in \cite{BB} within the cohomological framework, because the bordism theory of stratifolds is naturally isomorphic to singular homology. We will develop this viewpoint in a future paper.
\end{Rmk*}

\section{Cap Product and Cheeger-Simons Characters}\label{CapCSSec}

We recall that there exists the natural pairing:
\begin{eqnarray}
	Z_{k}(X) \times \hat{H}^{k+1}(X) & \to & \R/\Z \label{Holonomy} \\
	(\gamma_{k}, \xi^{k+1}) & \mapsto & \xi^{k+1}(\gamma_{k}) \nonumber
\end{eqnarray}
If we represent $\hat{H}^{\bullet}(X)$ through the model of differential characters, this pairing is part of the definition; in general, it is the holonomy of $\xi^{k+1}$ on $\gamma_{k}$. Furthermore, corollary \ref{LemmaH} implies:
\begin{equation}\label{CapHolonomy}
	\hatcap \colon \hat{H}_{k}(X) \times \hat{H}^{k+1}(X) \to \R/\Z
\end{equation}
Definition \ref{DefDiffHom} suggests the following natural transformation:
\begin{eqnarray}
	b \colon Z_{k}(X) & \to & \hat{H}_{k}(X) \label{DefTransfB} \\
	\gamma_{k} & \mapsto & (-1)^{k+1}[\gamma_{k}, 0] \nonumber
\end{eqnarray}
The next proposition shows that this transformation links formulas \eqref{Holonomy} and \eqref{CapHolonomy}, so that holonomy can be computed on a differential homology class. This may be interesting for applications in mathematical physics, since it allows to describe a charge source without fixing a representative cycle.
\begin{Prop}\label{PropCapXi} The natural transformation \eqref{DefTransfB} is the unique one satisfying
\begin{equation}\label{AxiomB}
	b(\gamma_{k}) \hatcap \xi^{k+1} = \xi^{k+1}(\gamma_{k})
\end{equation}
for every $\xi^{k+1} \in \hat{H}^{k+1}(X)$.
\end{Prop}
\begin{proof} We have
\begin{align*}
	b(\gamma_{k}) \hatcap \xi^{k+1} & \overset{\eqref{DefTransfB}}= (-1)^{k+1}[\gamma_{k}, 0] \hatcap [\mu^{k+1}, c^{k}, \omega^{k+1}] \\
	&\overset{\eqref{DiffCapP}}= [0, T_{\gamma_{k} \cap c^{k}}] = [0, T_{c^{k}(\gamma_{k})}] \overset{\eqref{MorfIota}}= \iota([c^{k}(\gamma_{k})]_{\R/\Z})
\end{align*}
In the second equality, $\gamma_{k} \cap \mu^{k+1}$ and $\gamma_{k} \otimes \omega^{k+1}$ vanish since they are elements of degree $-1$. The last term coincides with $\xi^{k+1}(\gamma_{k})$ because $[\mu^{k+1}, c^{k}, \omega^{k+1}]$ corresponds to $\gamma_{k} \mapsto [c^{k}(\gamma_{k})]_{\R/\Z}$ as a differential character.

About uniqueness, let us suppose that $b' \colon Z_{k}(X) \to \hat{H}_{k}(X)$ is another transformation satisfying \eqref{AxiomB}. We have $b'(\gamma_{k}) = (-1)^{k+1}[\gamma'_{k}, T'_{k+1}]$ for suitable $\gamma'_{k}$ and $T'_{k}$. Then:
	\[b'(\gamma_{k}) \hatcap \xi^{k+1} = \bigl(b(\gamma'_{k}) + (-1)^{k+1}a(T'_{k+1})\bigr) \hatcap \xi^{k+1} = \xi^{k+1}(\gamma'_{k}) + a\bigl(T'_{k+1} \wedge R(\xi^{k+1})\bigr)
\]
By hypothesis, $b'(\gamma_{k}) \hatcap \xi^{k+1} = \xi^{k+1}(\gamma_{k})$. Thus, for any flat character $\xi^{k+1}$, we have $\xi^{k+1}(\gamma'_{k}) = \xi^{k+1}(\gamma_{k})$. The group of flat characters is $H^{k}(X; \R/\Z) \simeq \Hom(H_{k}(X); \R/\Z)$, and morphisms to $\R/\Z$ separate elements of an abelian group, so $[\gamma'_{k}] = [\gamma_{k}]$. Hence, $b'(\gamma_{k}) = (-1)^{k+1}[\gamma_{k}, T''_{k+1}]$ for a suitable $T''_{k+1}$, and $a(T''_{k+1} \wedge R(\xi^{k+1})) = 0$ for any character $\xi^{k+1}$. We have $T''_{k+1} \wedge R(\xi^{k+1}) = T''_{k+1}(R(\xi^{k+1}))$ where $R(\xi^{k+1})$ is any integral form. It follows that $T''_{k+1}(\omega^{k+1}) \in \Z$ for every integral form $\omega^{k+1}$, that is, $T''_{k+1}$ is an integral current. Thus, $b'(\gamma_{k}) = (-1)^{k+1}[\gamma_{k}, 0] = b(\gamma_{k})$.
\end{proof}

Formula \eqref{AxiomB} characterises the transformation $b$ axiomatically. It immediately implies that, if $b(\gamma_{k}) = 0$, then the value of any differential character in $\gamma_{k}$ is zero. The following proposition shows that the converse holds as well. This means that $b(\gamma_{k}) = b(\gamma'_{k})$ if and only if $\gamma_{k}$ and $\gamma'_{k}$ induce the same holonomy map.

\begin{Prop}\label{LemmaZeroChar} The following conditions are equivalent for every cycle $\gamma_{k} \in Z_{k}(X)$:
\begin{enumerate}
	\item $\xi^{k+1}(\gamma_{k}) = 0$ for every differential character $\xi^{k+1}$;
	\item there exists $\Gamma_{k+1} \in C_{k+1}(X)$ such that $\gamma_{k} = \partial \Gamma_{k+1}$ and $T_{\Gamma_{k+1}}$ is an integral current;
	\item $b(\gamma_{k}) = 0$.
\end{enumerate}
If these conditions hold, then $\gamma_{k}$ is a thin chain.
\end{Prop}
\begin{proof} (1) $\Rightarrow$ (2) If $\gamma_{k}$ is not a boundary, then there exists a flat character $\xi^{k+1} \in H^{k}(X; \R/\Z) \simeq \Hom(H_{k}(X); \R/\Z)$ such that $\xi^{k+1}[\gamma_{k}] \neq 0$. The hypothesis then implies $\gamma_{k} = \partial \Gamma_{k+1}$ for a suitable chain $\Gamma_{k+1}$, so that
	\[\textstyle 0 = \xi^{k+1}(\gamma_{k}) = \bigl[\int_{\Gamma_{k+1}} R(\xi^{k+1})\bigr]_{\R/\Z} = [T_{\Gamma_{k+1}}(R(\xi^{k+1}))]_{\R/\Z}
\]
for any character $\xi^{k+1}$. This means that $T_{\Gamma_{k+1}}(\omega^{k+1}) \in \Z$ for every integral form $\omega^{k+1}$, that is, $\Gamma_{k+1}$ is integral.

\vspace{3pt} (2) $\Rightarrow$ (3) We have $(-1)^{k+1} b(\gamma_{k}) = [\partial \Gamma_{k+1}, 0] = [0, T_{\Gamma_{k+1}}] = 0$ since $T_{\Gamma_{k+1}}$ is integral.

\vspace{3pt} (3) $\Rightarrow$ (1) Obvious from \eqref{AxiomB}.

\vspace{3pt} Assuming (2), since an exact form is integral, $T_{k+1}(d\eta^{k}) \in \Z$ for every $\eta^{k}$. This is equivalent to $\int_{\gamma_{k}} \eta^{k} \in \Z$ for every $\eta^{k}$, which happens if and only if $\int_{\gamma_{k}} \eta^{k} = 0$. Hence, $\gamma_{k}$ is thin.
\end{proof}

\begin{Rmk} Let $Y \subset X$ be an embedded oriented submanifold of dimension $k$. If $\gamma_{k}$ and $\gamma'_{k}$ are two representative cycles of the fundamental class $[Y]$, then $b(\gamma_{k}) = b(\gamma'_{k})$. Indeed, we have $\gamma_{k} - \gamma'_{k} = \partial \Gamma_{k+1}$ for a suitable chain $\Gamma_{k+1}$, and $T_{\Gamma_{k+1}} = 0$ since it is a $(k+1)$-current on a $k$-dimensional manifold. Hence, item (2) of proposition \ref{LemmaZeroChar} implies $b(\gamma_{k} - \gamma'_{k}) = 0$. It follows that, with a little abuse of notation, the class $b(Y) \in \hat{H}_{k}(X)$ is well-defined. Proposition \ref{LemmaZeroChar} implies $b(Y) \neq 0$, since a representative of $[Y]$ cannot be thin; similarly, if $Z$ is an embedded oriented submanifold of dimension $k$ different from $Y$, then the difference between any two representatives of $[Y]$ and $[Z]$ is not thin, hence $b(Y) \neq b(Z)$. It follows that the function $Y \mapsto b(Y)$ identifies the set of embedded oriented submanifolds of $X$ with a subset of $\hat{H}_{\bullet}(X)$.
\end{Rmk}

For completeness, we show that the transformation $b$ can be characterised axiomatically within the homological side (that is, without relying on the cap product), and fits in a long exact sequence analogous to \eqref{ExSeq1}. We define \emph{chain-integral currents} as follows:
\begin{equation}\label{DefINTCurr}
	\Tau_{k}^{\IInt}(X) := \bigl\{ T_{\gamma_{k}} + \partial T'_{k+1}: \gamma_{k} \in C_{k}(X), T'_{k+1} \in \Tau_{k+1}(X) \bigr\} \subset \Tau_{k}
\end{equation}
The difference with respect to integral currents, defined in \eqref{DefIntCurr}, is that $\gamma_{k}$ is allowed to be a chain, not necessarily a cycle. We obtain the following natural transformation:
\begin{eqnarray}
	J \colon \hspace{5pt} \hat{H}_{k}(X) & \to & \Tau_{k}(X)/\Tau_{k}^{\IInt}(X) \label{DefTransfJ} \\
	\lbrack \gamma_{k}, T_{k+1} \rbrack & \mapsto & [T_{k+1}] \nonumber
\end{eqnarray}
Furthermore, we use the following notation: $N_{k}(X)$ is the subgroup of $C_{k}(X)$ formed by chains satisfying condition (2) of proposition \ref{LemmaZeroChar}; $\pi \colon Z_{\bullet}(X) \to H_{\bullet}(X)$ and $p \colon \Tau_{k}(X) \to \Tau_{k}(X)/\Tau_{k}^{\IInt}(X)$ are the projections to the quotient; $T \colon C_{\bullet}(X; \R) \to \Tau_{\bullet}(X)$ is the morphism $\gamma_{k} \to T_{\gamma_{k}}$; and, in order to simplify some formulas below, $\bar{b}(\gamma_{k}) := (-1)^{k+1} b(\gamma_{k}) = [\gamma_{k}, 0]$.

\begin{Prop} The transformation \eqref{DefTransfB} is the unique one with the following properties:
\begin{enumerate}
	\item $I \circ \bar{b} = \pi$;
	\item $R \circ \bar{b} = T$;
	\item $\bar{b} \circ \partial = a \circ T$.
\end{enumerate}
Moreover, \eqref{DefTransfJ} is the unique transformation satisfying $J \circ a = p$ and making the following sequence exact:
\begin{equation}\label{ExSeqB}
	\xymatrix{
	0 \ar[r] & N_{k}(X) \ar[r] & Z_{k}(X) \ar[r]^(.48){b} & \hat{H}_{k}(X) \ar[r]^(.35){J} & \Tau_{k}(X)/\Tau_{k}^{\IInt}(X) \ar[r] & 0
}\end{equation}
\end{Prop}
\begin{proof} It is straightforward to verify that \eqref{DefTransfB} satisfies (1)--(3). About uniqueness, if $b'$ is another transformation satisfying the same axioms, then (1) implies $b'(\gamma_{k}) = b(\gamma_{k}) + a(T_{k+1})$ for a suitable $T_{k+1}$, and (2) implies $\partial T_{k+1} = 0$. Moreover, (3) implies that $b$ and $b'$ coincide on boundaries, thus $a(T_{k+1})$ only depends on the homology class $[\gamma_{k}]$. Therefore, we get the natural transformation $B := b - b' \colon H_{\bullet}(X) \to H_{\bullet+1}(X; \R)/H_{\bullet+1}(X)$, $[\gamma_{k}] \mapsto [T_{k+1}]$. Such a transformation necessarily vanishes. In fact, following again \cite[Fact 2.1]{SS}, we choose a good neighbourhood $U$ of the image of a cycle $\gamma_{k}$ and we call $i \colon U \hookrightarrow X$ the inclusion. We have $B([\gamma_{k}]) = B \circ i_{*}[\gamma_{k}]_{U} = i_{*} \circ B[\gamma_{k}]_{U} = 0$, since $B[\gamma_{k}]_{U} \in H_{\bullet+1}(U; \R)/H_{\bullet+1}(U) = 0$.

The exactness of \eqref{ExSeqB} until $Z_{k}(X)$ follows from the definition of $N_{k}$ and proposition \ref{LemmaZeroChar}. Moreover, $J \circ \bar{b}(\gamma_{k}) = J[\gamma_{k}, 0] = 0$. Conversely, if $J[\gamma_{k}, T_{k+1}] = 0$, then there exists an integral chain $\Gamma_{k+1}$ such that $[\gamma_{k}, T_{k+1}] = [\gamma_{k}, T_{\Gamma_{k+1}}] = [\gamma_{k} + \partial \Gamma_{k+1}, 0] = \bar{b}(\gamma_{k} + \partial \Gamma_{k+1})$. Lastly, for any current $T_{k+1}$, we have $J \circ a(T_{k+1}) = J[0, T_{k+1}] = [T_{k+1}]$, which proves both surjectivity of $J$ and the identity $J \circ a = p$.

About the uniqueness of $J$, if $J'$ is another choice, then $J'([\gamma_{k}, T_{k+1}]) = J' \circ \bar{b}(\gamma_{k}) + J' \circ a(T_{k+1}) = p(T_{k+1}) = J([\gamma_{k}, T_{k+1}])$ for every $[\gamma_{k}, T_{k+1}] \in \hat{H}_{k}(X)$.
\end{proof}

\section{Relative, Non-Compact, and Local Coefficient Versions}\label{RelNonCptLocCSec}

We construct the relative versions of differential singular homology, dualising the cohomological ones described in \cite{BT}, \cite{FR3}, and \cite{FR2}. This allows to state Lefschetz duality for compact orientable manifolds with boundary. Moreover, we develop the non-compactly supported and the local coefficient versions, leading to Poincar\'e and Lefschetz dualities without compactness or orientability assumptions.

\SkipTocEntry \subsection{Relative Homology}

Given a smooth map $f \colon A \to X$, we recall that the corresponding mapping cone complex is defined as $C_{\bullet}(f) := C_{\bullet}(X) \oplus C_{\bullet-1}(A)$ with boundary
\begin{equation}\label{BoundaryConeChains}
	\partial_{k}(\alpha_{k}, \beta_{k-1}) := (\partial\alpha_{k} + f_{*}\beta_{k-1}, -\partial \beta_{k-1}).
\end{equation}
The same definition holds with real coefficients. Analogously, $\Omega^{\bullet}(f) := \Omega^{\bullet}(X) \oplus \Omega^{\bullet-1}(A)$ with differential
\begin{equation}\label{BoundaryConeForms}
	d^{k}(\omega^{k}, \eta^{k-1}) := (d\omega^{k}, f^{*}\omega^{k} - d\eta^{k-1}).
\end{equation}
We obtain the corresponding relative (co)homologies. Given a form $(\omega^{k}, \eta^{k-1}) \in \Omega^{k}(f)$ and a chain $(\alpha_{k}, \beta_{k-1}) \in C_{k}(f; \R)$, we set:
	\[\int_{(\alpha_{k}, \, \beta_{k-1})} (\omega^{k}, \eta^{k-1}) := \int_{\alpha_{k}} \omega^{k} + \int_{\beta_{k-1}} \eta^{k-1}
\]
Stokes' theorem holds as in the absolute setting. A relative closed form $(\omega, \eta)$ is integral when $\int_{(\gamma, \theta)} (\omega, \eta) \in \Z$ for every integral cycle $(\gamma, \theta)$. Also, $\Tau_{\bullet}(f) := \Tau_{\bullet}(X) \oplus \Tau_{\bullet-1}(A)$ with boundary analogous to \eqref{BoundaryConeChains}. This complex is naturally isomorphic to the dual of $\Omega^{\bullet}(f)$ with boundary $(\partial T_{k})(\omega^{k}, \eta^{k-1}) = T_{k}(d(\omega^{k}, \eta^{k-1}))$.

By replacing cycles and currents in definition \ref{DefDiffHom} with their relative version through the mapping cone, we obtain the \emph{relative differential homology groups} $\hat{H}_{\bullet}(f)$. The usual transformations $I$, $R$, $a$, $b$, and $J$ are constructed in the analogous way. This definition dualises the relative Cheeger-Simons characters presented in Definition 2.1 of \cite{BT}, and corresponds to the case $(p, p-1)$ in Definition 2.2 of \cite{FR3}.
\begin{Def2} We call \emph{covariant derivative} the natural morphism
\begin{eqnarray}
	\cov \colon \Tau_{k}(A) & \to & \hat{H}_{k}(f) \label{DefCov} \\
	T_{k} & \mapsto & a(0, T_{k}) \nonumber
\end{eqnarray}
\end{Def2}
The proofs of the following proposition and of most of the ones below are postponed to appendix \ref{AppExRelHom}.
\begin{Prop}\label{PropEx1} Each sequence of the following form is exact:

\vspace{-10pt}
\begin{small}
\begin{equation}\label{LongExact1}
\begin{tikzcd}
\cdots \arrow[r] & H_{\bullet+2}(A; \R/\Z) \arrow[r]
\arrow[d, phantom, ""{coordinate, name=Z}]
& H_{\bullet+2}(X; \R/\Z) \arrow[r] & H_{\bullet+2}(f; \R/\Z) \arrow[dll, rounded corners=8pt, curvarr=Z] & 
\\
& H_{\bullet+1}(A; \R/\Z) \arrow[r, "\iota \circ f_{*}"]\arrow[d, phantom, ""{coordinate, name=W}] & \hat{H}_{\bullet}(X) \arrow[r, "i_{X}"]
& \hat{H}_{\bullet}(f)\arrow[dll,rounded corners=8pt,curvarr=W, "\pi_{A}" description] & 
\\
& \hat{H}_{\bullet-1}(A) \arrow[r, "f_{*} \circ \, I"] & H_{\bullet-1}(X)\arrow[r] & H_{\bullet-1}(f) \arrow[r] & \cdots
\end{tikzcd}
\end{equation}
\end{small}

\noindent where
\begin{align}
	& \iota \circ f_{*}[\Gamma_{k+1}] = [-\partial f_{*}\tilde{\Gamma}_{k+1}, T_{f_{*}\tilde{\Gamma}_{k+1}}] & & i_{X}[\gamma_{k}, T_{k+1}] := \bigl[(\gamma_{k}, 0), (T_{k+1}, 0)\bigr] \nonumber \\
	& \pi_{A}\bigl[(\gamma_{k}, \theta_{k-1}), (T_{k+1}, T'_{k})\bigr] := [\theta_{k-1}, -T'_{k}] & & f^{*} \circ I[\theta_{k-1}, T'_{k}] = [f_{*}\theta_{k-1}].\label{MapPiA}
\end{align}
Moreover, the morphisms
	\[i_{X} \circ f_{*} \colon \hat{H}_{\bullet}(A) \to \hat{H}_{\bullet}(f) \hspace{50pt} f_{*} \circ \pi_{A} \colon \hat{H}_{\bullet}(f) \to \hat{H}_{\bullet-1}(X)
\]
satisfy the identities
\begin{equation}\label{FormulasCompRel}
	i_{X} \circ f_{*} = \cov \circ R \hspace{50pt} f_{*} \circ \pi_{A} = -a \circ p_{X} \circ R,
\end{equation}
where $p_{X} \colon \Tau_{k}(f) \to \Tau_{k}(X)$ is the projection $(T_{k}, T'_{k+1}) \mapsto T_{k}$.
\end{Prop}
When $f \colon A \hookrightarrow X$ is a closed embedding, the relative chain complex is defined as $C_{\bullet}(X, A) := C_{\bullet}(X)/f_{*}C_{\bullet}(A)$, the version with real coefficients being analogous. Moreover, the complex of relative forms is $\Omega^{\bullet}(X, A) := \{\omega \in \Omega^{\bullet}(X): f^{*}\omega = 0\}$. The corresponding (co)homology groups are canonically isomorphic to the ones obtained through the complex cone. A relative closed form $\omega \in \Omega^{\bullet}(X, A)$ is integral when $\int_{\gamma} \omega \in \Z$ for every integral cycle $[\gamma] \in Z_{k}(X, A)$. Also, $\Tau_{\bullet}(X, A) := \Tau_{\bullet}(X)/f_{*}\Tau_{\bullet}(A)$ is naturally isomorphic to the dual of $\Omega^{\bullet}(X, A)$.

By replacing cycles and currents in definition \ref{DefDiffHom} with the relative versions just remembered, we obtain the \emph{parallel differential homology groups} $\hat{H}^{\bullet}(X, A)$. Again, the usual transformations $I$, $R$, $a$, $b$, and $J$ are constructed in the analogous way. The elements of $\hat{H}^{\bullet}(X, A)$ are called \emph{parallel classes}. This definition corresponds to the case $(p, p)$ in Definition 2.2 of \cite{FR3}. We also have the natural morphism:
\begin{eqnarray}
	\hat{f} \colon \hspace{5pt} \hat{H}_{\bullet}(f) & \to & \hat{H}_{\bullet}(X, A) \label{ProjParallel} \\
	\bigl\lbrack(\gamma_{k}, \theta_{k-1}), (T_{k+1}, T'_{k})\bigr\rbrack & \mapsto & \bigl[[\gamma_{k}], [T_{k+1}] \bigr] \nonumber
\end{eqnarray}
\begin{Prop}\label{PropEx2} The morphism \eqref{ProjParallel} is surjective and $\Ker(\hat{f}) = \IIm(\cov)$. Furthermore, each sequence of the following form is exact:

\vspace{-10pt}
\begin{small}
\begin{equation}\label{LongExact2}
\begin{tikzcd}
\cdots \arrow[r] & H_{\bullet+2}(A; \R/\Z) \arrow[r]
\arrow[d, phantom, ""{coordinate, name=Z}]
& H_{\bullet+2}(X; \R/\Z) \arrow[r] & H_{\bullet+2}(X, A; \R/\Z) \arrow[dll, rounded corners=8pt, curvarr=Z, "\iota \circ \beta" description] & 
\\
& \hat{H}_{\bullet}(A) \arrow[r, "f_{*}"]\arrow[d, phantom, ""{coordinate, name=W}] & \hat{H}_{\bullet}(X) \arrow[r, "\pi_{X}"]
& \hat{H}_{\bullet}(X, A)\arrow[dll,rounded corners=8pt,curvarr=W, "\beta \circ I" description] & 
\\
& H_{\bullet-1}(A) \arrow[r] & H_{\bullet-1}(X)\arrow[r] & H_{\bullet-1}(X, A) \arrow[r] & \cdots
\end{tikzcd}
\end{equation}
\end{small}

\noindent Here the map $\iota \circ \beta$ is the composition of $\beta \colon H_{\bullet+2}(X, A; \R/\Z) \to H_{\bullet+1}(A; \R/\Z)$ with $\iota \colon H_{\bullet+1}(A; \R/\Z) \to \hat{H}_{\bullet}(A)$. In addition, $\pi_{X}[\gamma_{k}, T_{k+1}] := \bigl[[\gamma_{k}], [T_{k+1}]\bigr]$, and $\beta \circ I$ is the composition of $I \colon \hat{H}_{\bullet}(X, A) \to H_{\bullet}(X, A)$ with $\beta \colon H_{\bullet}(X, A) \to H_{\bullet-1}(A)$.
\end{Prop}

The morphism \eqref{ProjParallel}, together with $\iota \colon H_{\bullet-2}(A; \R/\Z) \to \hat{H}_{\bullet-1}(A)$ and $I \colon \hat{H}_{\bullet}(A) \to H_{\bullet}(A)$, induces a morphism of exact sequences from \eqref{LongExact1} to \eqref{LongExact2}. Moreover, the next proposition shows that the parallel theory satisfies excision as well as the topological one.

\begin{Prop} If $A$ closed in $X$ and $B$ open in $A$ are submanifolds (with or without boundary) satisfying $\bar{B} \subset \mathring{A}$, then the pairwise inclusion $i \colon (X \setminus B, A \setminus B) \to (X, A)$ induces the natural \emph{isomorphism}
	\[i_{*} \colon \hat{H}_{\bullet}(X \setminus B, A \setminus B) \overset{\!\simeq}\longrightarrow \hat{H}_{\bullet}(X, A).
\]
\end{Prop}
\begin{proof} We consider the following morphism of short exact sequences:
	\[\xymatrix{
	0 \ar[r] & \frac{\Tau_{\bullet+1}(X \setminus B, A \setminus B)}{\Tau^{\Int}_{\bullet+1}(X \setminus B, A \setminus B)} \ar[r]^(.43){a} \ar[d] & \hat{H}_{\bullet}(X \setminus B, A \setminus B) \ar[r]^{I} \ar[d] & H_{\bullet}(X \setminus B, A \setminus B) \ar[r] \ar[d] & 0 \\
	0 \ar[r] & \frac{\Tau_{\bullet+1}(X, A)}{\Tau^{\Int}_{\bullet+1}(X, A)} \ar[r]^(.48){a} & \hat{H}_{\bullet}(X, A) \ar[r]^{I} & H_{\bullet}(X , A) \ar[r] & 0
}\]
The right vertical arrow is the excision isomorphism in singular homology. About the left one, we observe that the pullback $\Omega^{\bullet+1}(X, A) \to \Omega^{\bullet+1}(X \setminus B, A \setminus B)$ is an isomorphism, since a form on $X \setminus B$ that vanishes in $A \setminus B$ can be extended uniquely to a form on $X$ that vanishes in $A$. It follows that the pushforward $\Tau_{\bullet+1}(X \setminus B, A \setminus B) \to \Tau_{\bullet+1}(X, A)$ is an isomorphism as well. It sends integral currents to integral currents, since its restriction to closed currents induces the excision isomorphism in real singular homology. Hence, the left vertical arrow in the previous diagram is an isomorphism. Then the five lemma leads to the result.
\end{proof}

\begin{Corollary} We have the following Mayer-Vietoris sequence:

\vspace{-10pt}
\begin{footnotesize}
\begin{equation}\label{MayerVietoris}
\begin{tikzcd}
\cdots \arrow[r] \arrow[d, phantom, ""{coordinate, name=Z}]
& H_{\bullet+2}(U; \R/\Z) \oplus H_{\bullet+2}(V; \R/\Z) \arrow[r] & H_{\bullet+2}(X; \R/\Z) \arrow[dll, rounded corners=8pt, curvarr=Z, "\iota \circ \beta" description] & 
\\
\hat{H}_{\bullet}(U \cap V) \arrow[r, "(j_{U})_{*} \oplus (j_{V})_{*}"]\arrow[d, phantom, ""{coordinate, name=W}] & \hat{H}_{\bullet}(U) \oplus \hat{H}_{\bullet}(V) \arrow[r, "(i_{U})_{*} - (i_{V})_{*}"]
& \hat{H}_{\bullet}(X)\arrow[dll,rounded corners=8pt,curvarr=W, "\beta \circ I" description] & 
\\
H_{\bullet-1}(U \cap V) \arrow[r] & H_{\bullet-1}(U) \oplus H_{\bullet-1}(V) \arrow[r] & \cdots
\end{tikzcd}
\end{equation}
\end{footnotesize}

\noindent where $U, V, U \cap V \subseteq X$ are closed submanifolds such that $X = \mathring{U} \cup \mathring{V}$,\footnote{This assumption becomes much more flexible by extending the theory to stratifolds.} and $i_{U} \colon U \hookrightarrow X$, $i_{V} \colon V \hookrightarrow X$, $j_{U} \colon U \cap V \hookrightarrow U$, and $j_{V} \colon U \cap V \hookrightarrow V$ are the corresponding inclusions.
\end{Corollary}

We can also construct a version of relative differential homology that fits in a long exact sequence completely made by differential extensions. The construction is similar to the one in section 3 of \cite{BT}. We first observe that the groups $\hat{H}_{\bullet}(X)$ can be equivalently defined as the homology groups of the following chain complex:
	\[\hat{C}_{\bullet}(X) := C_{\bullet}(X) \oplus \Tau_{\bullet+1}(X) \oplus \Tau_{\bullet}^{\Int}(X) \hspace{20pt} \partial(\gamma_{k}, T_{k+1}, R_{k}) := (\partial \gamma_{k}, R_{k} - T_{\gamma_{k}} - \partial T_{k+1}, 0)
\]
The natural isomorphism from definition \ref{DefDiffHom} to this model can be easily obtained by setting $R_{k}$ as the curvature of $[\gamma_{k}, T_{k+1}]$. Now, given a smooth map $f \colon A \to X$, we define $\hat{H}'_{\bullet}(f)$ as the homology of the mapping cone of $\hat{C}_{\bullet}$.

\begin{Prop}\label{PropEx3} We have

\vspace{-10pt}
\begin{small}
\begin{equation}\label{HPrimeIso}
	\hat{H}'_{\bullet}(f) \simeq \frac{\Ker\bigl(\hat{H}_{\bullet}(f) \xrightarrow{f_{*} \circ \,\pi_{A}} \hat{H}_{\bullet-1}(X)\bigr)}{\IIm\bigl(\hat{H}_{\bullet}(A) \xrightarrow{i_{X} \circ f_{*}} \hat{H}_{\bullet}(f)\bigr)}
\end{equation}
\end{small}

\noindent and the following sequence is exact:

\vspace{-10pt}
\begin{small}
\begin{equation}\label{LongExact3}
\begin{tikzcd}
\cdots \arrow[r] & \hat{H}_{\bullet+1}(A) \arrow[r]
\arrow[d, phantom, ""{coordinate, name=Z}]
& \hat{H}_{\bullet+1}(X) \arrow[r,] & \hat{H}'_{\bullet+1}(f) \arrow[dll, rounded corners=8pt, curvarr=Z] & 
\\
& \hat{H}_{\bullet}(A) \arrow[r, "f_{*}"]\arrow[d, phantom, ""{coordinate, name=W}] & \hat{H}_{\bullet}(X) \arrow[r, "i_{X}"]
& \hat{H}'_{\bullet}(f)\arrow[dll,rounded corners=8pt,curvarr=W, "\pi_{A}" description] & 
\\
& \hat{H}_{\bullet-1}(A) \arrow[r] & \hat{H}_{\bullet-1}(X)\arrow[r] & \hat{H}'_{\bullet-1}(f) \arrow[r] & \cdots
\end{tikzcd}
\end{equation}
\end{small}
\end{Prop}

For completeness, we define the fourth and last type of relative homology, corresponding to the case $(p, p-2)$ in Definition 2.2 of \cite{FR3}. We set $\hat{H}''_{\bullet}(f) := \Ker(\pi_{A} \circ R)$, where $\pi_{A} \circ R \colon \hat{H}_{\bullet}(f) \to \Tau_{\bullet-1}(A)$ selects the second component of the curvature.

\begin{Prop}\label{PropEx4} Each sequence of the following family is exact:

\vspace{-10pt}
\begin{small}
\begin{equation}\label{LongExact4}
\begin{tikzcd}
\cdots \arrow[r] & H_{\bullet+2}(A; \R/\Z) \arrow[r]
\arrow[d, phantom, ""{coordinate, name=Z}]
& H_{\bullet+2}(X; \R/\Z) \arrow[r] & H_{\bullet+2}(f; \R/\Z) \arrow[dll, rounded corners=8pt, curvarr=Z] & 
\\
& H_{\bullet+1}(A; \R/\Z) \arrow[r, "\iota \circ f_{*}"]\arrow[d, phantom, ""{coordinate, name=W}] & \hat{H}_{\bullet}(X) \arrow[r, "i_{X}"]
& \hat{H}''_{\bullet}(f)\arrow[dll,rounded corners=8pt,curvarr=W, "\iota^{-1} \circ \pi_{A}" description] & 
\\
& H_{\bullet}(A; \R/\Z) \arrow[r, "\beta \circ f_{*}"] & H_{\bullet-1}(X)\arrow[r] & H_{\bullet-1}(f) \arrow[r] & \cdots
\end{tikzcd}
\end{equation}
\end{small}

\end{Prop}

The inclusion $\hat{H}''_{\bullet}(f) \hookrightarrow \hat{H}_{\bullet}(f)$ and $\iota \colon H_{\bullet}(A; \R/\Z) \hookrightarrow \hat{H}_{\bullet}(A)$ induce a morphism of exact sequences from \eqref{LongExact4} to \eqref{LongExact1}.

\subsubsection{Relative Cohomology}

The mapping cone complex of a smooth map $f \colon A \to X$ in cohomology is $C^{\bullet}(f) := C^{\bullet}(X) \oplus C^{\bullet-1}(A)$ with coboundary analogous to \eqref{BoundaryConeForms}. The definition with real coefficients is analogous, and the de-Rham isomorphism holds as in the absolute setting. We define the \emph{relative differential cohomology groups} $\hat{H}^{\bullet}(f)$ as in \eqref{DefDiffCoh}, by considering relative cocycles and forms instead. Given a relative class with curvature $(\omega, \eta) \in \Omega^{\bullet}_{\Int}(f)$, we call $\eta$ \emph{covariant derivative}; a class with vanishing covariant derivative is called \emph{parallel}. When $f \colon A \hookrightarrow X$ is a closed embedding, we obtain the \emph{parallel differential cohomology groups} $\hat{H}^{\bullet}(X, A)$ by using the corresponding relative complexes in \eqref{DefDiffCoh}. Here $\hat{H}^{\bullet}(X, A)$ is a subgroup of $\hat{H}^{\bullet}(f)$, while it was a quotient in homology. We also have the cohomological versions of $\hat{H}'_{\bullet}(f)$ and $\hat{H}''_{\bullet}(f)$. In each case, exact sequences dual to the ones above hold.

\subsubsection{Cap Product and Lefschetz Duality}

Given a smooth map $f \colon A \to X$, we have the cap products $C_{l}(X) \times C^{s}(f) \to C_{l-s}(f)$ and $C_{l}(f) \times C^{s}(X) \to C_{l-s}(f)$ defined componentwise. Similarly, we have the wedge products $\Tau_{l}(X) \times \Omega^{s}(f) \to \Tau_{l-s}(f)$ and $\Tau_{l}(f) \times \Omega^{s}(X) \to \Tau_{l-s}(f)$. Hence, by adapting definition \eqref{DiffCapP} in the natural way, we obtain the following products:
\begin{equation}\label{RelDiffCapP}
\begin{split}
	\hat{\cap} \colon \hat{H}_{l}(X) \times \hat{H}^{s}(f) \to \hat{H}_{l-s}(f) \hspace{40pt} \hat{\cap} \colon \hat{H}_{l}(f) \times \hat{H}^{s}(X) \to \hat{H}_{l-s}(f)
\end{split}
\end{equation}
When $f \colon A \hookrightarrow X$ is a closed embedding, the compatibility between the cap product and $a$ implies that \eqref{RelDiffCapP} projects to
\begin{equation}\label{RelDiffCapP2}
\begin{split}
	\hat{\cap} \colon \hat{H}_{l}(X) \times \hat{H}^{s}(X, A) \to \hat{H}_{l-s}(X, A) \hspace{25pt} \hat{\cap} \colon \hat{H}_{l}(X, A) \times \hat{H}^{s}(X) \to \hat{H}_{l-s}(X, A).
\end{split}
\end{equation}
Moreover, in this case, we have the cap product $C_{l}(X, A) \times C^{s}(X, A) \to C_{l-s}(X)$ defined by $[\alpha] \cap \varphi := \alpha \cap \varphi$, and the analogous wedge product $\Tau_{l}(X, A) \times \Omega^{s}(X, A) \to \Tau_{l-s}(X)$. Hence, we obtain:
\begin{equation}\label{RelDiffCapP3}
	\hat{\cap} \colon \hat{H}_{l}(X, A) \times \hat{H}^{s}(X, A) \to \hat{H}_{l-s}(X)
\end{equation}
Let us consider a compact orientable manifold $X$ of dimension $n$ with boundary. We have the fundamental class $[X] \in H_{n}(X, \partial X)$. By calling $\beta$ the Bockstein map of the long exact sequence induced by $(X, \partial X)$, the identity $\beta[X] = [\partial X]$ holds. Because of the relative version of corollary \ref{LemmaH}-(i), we can identify $[X]$ with $I^{-1}[X] \in \hat{H}_{n}(X, \partial X)$. Formulas \eqref{RelDiffCapP2} and \eqref{RelDiffCapP3} with $A = \partial X$ lead to the \emph{differential Lefschetz dualities}
\begin{equation}\label{DiffLef}
	\hatL \colon \hat{H}^{k}(X) \to \hat{H}_{n-k}(X, \partial X) \hspace{50pt} \hatL \colon \hat{H}^{k}(X, \partial X) \to \hat{H}_{n-k}(X)
\end{equation}
both defined as $\xi \mapsto [X] \hatcap \xi$. The statement and the proof of proposition \ref{ProofPD} can be adapted to \eqref{DiffLef}, provided that we construct diagrams analogous to \eqref{DiagPD2} and \eqref{DiagPD1}. This is possible thanks to the embeddings $\Omega^{\bullet}(X) \hookrightarrow \Tau^{\bullet}(X, \partial X)$ and $\Omega^{\bullet}(X, \partial X) \hookrightarrow \Tau^{\bullet}(X)$, both defined by $\omega \mapsto T_{\omega}$ where $T_{\omega}(\eta) := \int_{X} \omega \wedge \eta$. The Stokes' theorem implies $T_{d\omega} = dT_{\omega}$, so that $H^{\bullet}_{\dR}(X) \simeq H\Tau^{\bullet}(X, \partial X)$ and $H^{\bullet}_{\dR}(X, \partial X) \simeq H\Tau^{\bullet}(X)$ naturally.\footnote{In the case of real chains, we have the maps $C_{\bullet}(X, \partial X; \R) \to \Tau_{\bullet}(X, \partial X)$ and $C_{\bullet}(X; \R) \to \Tau_{\bullet}(X)$ with no absolute-relative inversion. We obtain the Lefschetz dualities $H_{\bullet}(X, \partial X; \R) \simeq H\Tau_{\bullet}(X, \partial X) = H\Tau^{n-\bullet}(X, \partial X) \simeq H^{n-\bullet}_{\dR}(X)$ and $H_{\bullet}(X; \R) \simeq H\Tau_{\bullet}(X) = H\Tau^{n-\bullet}(X) \simeq H^{n-\bullet}_{\dR}(X, \partial X)$.} If we extend cohomology or restrict homology, we get diagrams similar to \eqref{DiagPD2B} and \eqref{DiagPD1B}, so that formulas \eqref{DiffLef} become isomorphisms. Lastly, the family of exact sequences \eqref{LongExact2} and its cohomological counterpart induce the following family of commutative diagrams with exact rows:
	\[\resizebox{\textwidth}{!}{\xymatrix{
	\cdots \ar[r] & H^{\bullet-2}(X; \R/\Z) \ar[r] \ar[d]^(.45){\textnormal{L}} & H^{\bullet-2}(\partial X; \R/\Z) \ar[r] \ar[d]^(.45){\iota \,\circ\, \PD} & \hat{H}^{\bullet}(X, \partial X) \ar[r] \ar[d]^(.45){\hatL} & \hat{H}^{\bullet}(X) \ar[r] \ar[d]^(.45){\hatL} & \hat{H}^{\bullet}(\partial X) \ar[r] \ar[d]^(.45){\PD \,\circ\, I} & H^{\bullet+1}(X, \partial X) \ar[r] \ar[d]^(.45){\textnormal{L}} & \cdots \\
	\cdots \ar[r] & H_{n-\bullet+2}(X, \partial X; \R/\Z) \ar[r] & \hat{H}_{n-\bullet}(\partial X) \ar[r] & \hat{H}_{n-\bullet}(X) \ar[r] & \hat{H}_{n-\bullet}(X, \partial X) \ar[r] & H_{n-\bullet-1}(\partial X) \ar[r] & H_{n-\bullet-1}(X) \ar[r] & \cdots
}}\]
\vspace{-5pt}

\SkipTocEntry \subsection{Non-Compactly Supported Homology}

We defined differential homology via singular chains, which have compact support by construction, and compactly-supported currents. We can construct the analogous version with no restrictions on the supports. This leads to Poincar\'e and Lefchetz dualities without the compactness assumption.

We briefly recall the definition of \emph{Borel-Moore homology} on smooth manifolds. In singular homology, the group of $k$-chains is the direct sum of a copy of $\Z$ for each singular simplex, that is, $C_{k}(X) := \bigoplus_{\sigma \colon \Delta^{k} \to X} \Z$. Here we define $C^{\BM}_{k}(X)$ as the subgroup of the direct product $\prod_{\sigma \colon \Delta^{k} \to X} \Z$ formed by \emph{locally finite} elements. This means that every $x \in X$ admits a neighbourhood $U$ intersecting the image of finitely many singular simplexes with non-zero coefficient. The boundary can be defined as usual, since only finitely many $k$-simplexes with non-zero coefficient can share the same $(k-1)$-simplex as a face. We then obtain the chain complex $(C^{\BM}_{\bullet}(X), \partial_{\bullet})$ and denote by $H^{\BM}_{\bullet}(X)$ the corresponding homology. The latter is covariant with respect to proper smooth functions and contravariant with respect to open embeddings.

The space of (non-compactly-supported) currents of degree $k$ on $X$ is the dual of $\Omega^{n-k}_{\cpt}(X)$. By analogy with chains, we denote this space by $\Tau^{k}_{\BM}(X)$ and use the homological notation $\Tau_{k}^{\BM}(X)$ as before. Functoriality coincides with the one of Borel-Moore homology. By using $C^{\BM}_{\bullet}(X)$ and $\Tau_{\bullet}^{\BM}(X)$ in definition \ref{DefDiffHom}, we obtain $\hat{H}^{\BM}_{\bullet}(X)$. The cap product, the relative versions of $\hat{H}^{\BM}_{\bullet}$, and the Poincar\'e and Lefschetz dualities are then constructed similarly.

\SkipTocEntry \subsection{Local Coefficient Homology}

In order to state Poincar\'e and Lefschetz dualities for not necessarily orientable manifolds, we define differential homology with local coefficients. We begin by briefly reviewing the theory of densities and its relation with currents.

\subsubsection{Densities and Currents} Given a smooth manifold $X$ (even with boundary), we consider the two-sheeted orientation cover $\tilde{X} \to X$. As a set, it is formed by the pairs $(x, o_{x})$ where $x \in X$ and $o_{x}$ is an orientation of the tangent fibre $T_{x}X$. The \emph{orientation line bundle} of $X$ is $\Oo := \tilde{X} \times_{\Z_{2}} \R$ and it is classified up to isomorphism by the first Stiefel-Whitney class $w_{1}(X) \in H^{1}(X; \Z_{2})$. Furthermore, $\Oo$ is naturally endowed with the metric $\langle [x, o_{x}, t], [x, o_{x}, u] \rangle := tu$, and there exist two unitary sections in every local chart. A \emph{density} on $X$ is a differential form with values in $\Oo$, that is, a section of $\Lambda^{\bullet} T^{*}X \otimes \Oo$. Concretely, we represent it as a sum of elementary densities of the form $\omega^{k} \otimes s$, where $\omega^{k}$ is a $k$-form and $s$ is a section of $\Oo$. One can multiply densities by forms (or vice-versa) by setting $\omega^{h} \wedge (\omega^{k} \otimes s) := (\omega^{h} \wedge \omega^{k}) \otimes s$. Moreover, the product of two densities is defined as the form $(\omega^{k} \otimes s) \wedge (\omega^{h} \otimes s') := \langle s, s' \rangle \cdot \omega^{k} \wedge \omega^{h}$.

We denote by $\Omega^{\bullet}(X; \Oo)$ the complex of densities on $X$ with exterior differential defined as follows. A section $s$ of $\Oo$ can be written in a local chart $U$ of $X$ as $s(x) = [x, o_{x}, f(x)]$, where $x \mapsto o_{x}$ is a fixed orientation of $U$ and $f$ is a smooth real-valued function. We say that $s$ is \emph{constant} if $f$ is a constant function. We set $d(\omega^{k} \otimes s_{0}) := d\omega^{k} \otimes s_{0}$ for any constant section $s_{0}$. The space of \emph{compactly-supported $k$-currents with local coefficients} on $X$, denoted by $\Tau^{k}(X; \Oo)$, is then the dual of $\Omega^{n-k}(X; \Oo)$. The corresponding exterior differential is defined as usual and, with homological notation, we obtain the complex $(\Tau_{\bullet}(X; \Oo), \partial_{\bullet})$. Densities and local-coefficient currents are functorial with respect to smooth covering maps.

Densities can be integrated similarly to forms by fixing an atlas, without the orientability assumption. We then have the embeddings

\noindent \begin{minipage}{0.5\textwidth}
    \centering
    \begin{eqnarray}\label{EmbJLoc}
		\ScT^{k}_{\Oo} \colon \Omega^{k}(X) & \hookrightarrow & \Tau^{k}(X; \Oo) \\
		\omega^{k} & \mapsto & T_{\omega^{k}} \nonumber
\end{eqnarray}
\end{minipage}
\hspace{-30pt}
\begin{minipage}{0.5\textwidth}
    \centering
    \begin{eqnarray*}
		\ScT^{k}_{\Oo'} \colon \Omega^{k}(X; \Oo) & \hookrightarrow & \Tau^{k}(X) \\
		\omega^{k} \otimes s & \mapsto & T_{\omega^{k} \otimes s}
\end{eqnarray*}
\end{minipage}
\vspace{10pt}

\noindent where $T_{\omega^{k}}(\eta^{n-k} \otimes s) := \int_{X} \omega^{k} \wedge \eta^{n-k} \otimes s$ and $T_{\omega^{k} \otimes s}(\eta^{n-k}) := \int_{X} \omega^{k} \wedge \eta^{n-k} \otimes s$. By projecting to cohomology, we obtain the isomorphisms $\bar{\ScT}^{k}_{\Oo} \colon H^{k}_{\dR}(X) \overset{\!\simeq}\longrightarrow H\Tau^{k}(X; \Oo)$ and $\bar{\ScT}^{k}_{\Oo'} \colon H^{k}_{\dR}(X; \Oo) \overset{\!\simeq}\longrightarrow H\Tau^{k}(X)$.

\subsubsection{Singular (Co)Homology} A \emph{real singular $k$-chain with local coefficients} in $X$ is as a sum of terms of the form $s\sigma$, where $\sigma \colon \Delta^{k} \to X$ is a singular simplex and $s \colon \Delta^{k} \to \Oo$ is a \emph{constant} section of $\Oo$ that lifts $\sigma$. The boundary of $s\sigma$ is defined by considering the usual boundary of $\sigma$ and restricting $s$ to each face. We obtain the chain complex $C_{\bullet}(X; \Oo)$ inducing the homology groups $H_{\bullet}(X; \Oo)$. Considering the sub-fibre-bundle $\Oo_{\Z} := \tilde{X} \times_{\Z_{2}} \Z$ of $\Oo$, we define the complex $C_{\bullet}(X; \Oo_{\Z})$ of \emph{integral} chains with local coefficients, inducing the homology groups $H_{\bullet}(X; \Oo_{\Z})$. Cohomology is defined as usual. We have the natural morphism
\begin{eqnarray*}
	\ScT_{k, \Oo} \colon C_{k}(X; \Oo) & \to & \Tau_{k}(X; \Oo) \\
	s\sigma & \mapsto & T_{s\sigma}
\end{eqnarray*}
where $T_{s\sigma}(\omega \otimes s') := \langle s, s' \rangle \cdot \int_{\sigma} \omega$, inducing $\bar{\ScT}_{k, \Oo} \colon H_{k}(X; \Oo) \overset{\!\simeq}\longrightarrow H\Tau_{k}(X; \Oo)$.\footnote{By composing with the isomorphisms induced by \eqref{EmbJLoc}, we obtain the Poincar\'e dualities $H_{\bullet}(X; \Oo) \simeq H\Tau_{\bullet}(X; \Oo) = H\Tau^{n-\bullet}(X; \Oo) \simeq H^{n-\bullet}_{\dR}(X)$ and $H_{\bullet}(X; \R) \simeq H\Tau_{\bullet}(X) = H\Tau^{n-\bullet}(X) \simeq H^{n-\bullet}_{\dR}(X; \Oo)$.} Lastly, the embedding $\Omega^{k}(X; \Oo) \hookrightarrow C^{k}(X; \Oo)$ is defined as $(\omega \otimes s)(s'\sigma) := \langle s, s' \rangle \cdot \int_{\sigma} \omega$.

\subsubsection{Differential (Co)Homology}

By using local coefficients in definitions \ref{DefDiffHom}, we obtain the \emph{differential singular $k$-homology group with local coefficients} $\hat{H}_{k}(X; \Oo)$. The same operation on \eqref{DefDiffCoh} leads to $\hat{H}^{k}(X; \Oo)$. The following cap products are defined as in \eqref{DiffCapP}:
	\[\hat{\cap} \colon \hat{H}_{l}(X; \Oo) \times \hat{H}^{s}(X) \to \hat{H}_{l-s}(X; \Oo) \hspace{25pt} \hat{\cap} \colon \hat{H}_{l}(X; \Oo) \times \hat{H}^{s}(X; \Oo) \to \hat{H}_{l-s}(X)
\]
Every $n$-manifold without boundary has a fundamental class $[X] \in H_{n}(X; \Oo_{\Z})$. By the straightforward adaptation of corollary \ref{LemmaH}-(i), we obtain the following versions of Poincar\'e duality:
	\[\hatPD \colon \hat{H}^{k}(X) \to \hat{H}_{n-k}(X; \Oo) \hspace{50pt} \hatPD \colon \hat{H}^{k}(X; \Oo) \to \hat{H}_{n-k}(X)
\]
both defined by $\xi \mapsto [X] \hatcap \xi$. These maps are injective and become isomorphisms by extending cohomology to currents or restricting homology to forms. The analogous construction in the relative framework leads to the Lefschetz dualities, and both the absolute and the relative versions can be adapted to the non-compactly-supported setting.

\section{Generalised Differential Homology}\label{GenHomSec}

We extend the construction of the differential extension to a generalised homology theory.

\SkipTocEntry \subsection{Preliminaries on differential cohomology}

We fix a multiplicative differential cohomology theory $\hat{h}^{\bullet}$, refining the topological theory $h^{\bullet}$. We set $\h^{\bullet} := h^{\bullet}(\pt)$ and $\h^{\bullet}_{\R} := \h^{\bullet} \otimes \R$, supposing $\chr(\h^{\bullet}) = 0$. Moreover, $\ch \colon h^{\bullet}(X) \to H^{\bullet}(X; \h_{\R})$ is the generalised Chern character, inducing the isomorphism $h^{\bullet}(X) \otimes \R \simeq H^{\bullet}(X; \h_{\R})$.

Given a real smooth vector bundle $\pi \colon E \to X$ of rank $n$, a \emph{differential Thom class} on $E$ is a vertically-compact class $\hat{u} \in \hat{h}^{n}_{\vcpt}(E)$ such that $u := I(\hat{u})$ is a Thom class. A \emph{differential orientation} of the bundle is a differential Thom class up to homotopy and stabilization (see \cite[sections 4.8--4.10]{Bunke} and \cite[section 3]{FR}). It is also possible to orient a neat proper submersion $\varphi \colon Y \to X$ between manifolds with (or without) boundary, essentially by orienting the normal bundle of $Y$ with respect to a neat embedding $\iota \colon Y \hookrightarrow X \times \R^{N}$ (for any $N$) such that $\pi_{X} \circ \iota = f$. In this case, we obtain the \emph{Gysin map} $\varphi_{!} \colon \hat{h}^{\bullet}(Y) \to \hat{h}^{\bullet-d}(X)$, where $d := \dim(Y) - \dim(X)$. Denoting by $\hat{u}$ the orientation of the normal bundle $\mathcal{N}$, we have
\begin{equation}\label{GysinR}
	R\bigl(\varphi_{!}(\hat{\alpha})\bigr) = \int_{Y/X} R(\alpha) \wedge \Td(\hat{u}),
\end{equation}
where $\Td(\hat{u}) := \int_{\mathcal{N}/Y} R(\hat{u})$ is the Todd class induced by $\hat{u}$. Coherently, by setting $u := I(\hat{u})$, the underlying topological map satisfies the identity
\begin{equation}\label{GysinCh}
	\ch\bigl(\varphi_{!}(\alpha)\bigr) = \int_{Y/X} \ch(\alpha) \cdot \Td(u),
\end{equation}
where $\Td(u) = \int_{\mathcal{N}/Y} \ch(u)$. Lastly, a differential orientation of a smooth manifold is a differential orientation of its stable normal bundle (equivalently, of its tangent bundle). If the manifold is compact without boundary, this is equivalent to orienting the proper submersion $p_{X} \colon X \to \pt$.

\SkipTocEntry \subsection{Differential Construction of Topological Homology}

Given a multiplicative cohomology theory $h^{\bullet}$ represented by a spectrum $\{E_{q}, \varepsilon_{q}\}_{q \in \Z}$, the dual homology theory $h_{\bullet}$ can be defined through the well-known formula $h_{k}(X) := \varinjlim \pi_{k+q}(X_{+} \wedge E_{q})$. In \cite{Jakob}, the author proposed a bordism-type model of $h_{\bullet}$ in which a homology class is represented by a quadruple $(M, u, \alpha, f)$, where $(M, u)$ is an $h^{\bullet}$-oriented compact manifold without boundary, $\alpha$ is a cohomology class on $M$ and $f \colon M \to X$ is a continuous function. The suitable equivalence relation between two representatives involves an operation called ``vector bundle modification'', that was equivalently replaced in \cite{FR} by the Gysin map. The same result can be obtained through a differential refinement of the representatives, as the following definition shows (see section 4 of \cite{FR}).
\begin{Def2}\label{DefHomH} On a smooth compact manifold $X$, we define:
\begin{itemize}
	\item the group of \emph{$k$-precycles}, denoted by $\hat{p}_{k}(X)$, as the free abelian group generated by the quadruples $(M, \hat{u}, \hat{\alpha}, f)$, where:
\begin{itemize}
	\item $(M, \hat{u})$ is a smooth compact manifold without boundary with $\hat{h}^{\bullet}$-orientation $\hat{u}$, whose connected components $\{M_{i}\}$ have dimension $k+q_{i}$ with $q_{i}$ arbitrary;
	\item $\hat{\alpha} \in \hat{h}^{\bullet}(M)$, such that $\hat{\alpha}\vert_{M_{i}} \in \hat{h}^{q_{i}}(M)$;
	\item $f \colon M \to X$ is a smooth map;
\end{itemize}
	\item the group of \emph{$k$-cycles}, denoted by $\hat{z}_{k}(X)$, as the quotient of $\hat{p}_{k}(X)$ by the free subgroup generated by elements of the form:
\begin{align*}
	& (M, \hat{u}, \hat{\alpha} + \hat{\beta}, f) - (M, \hat{u}, \hat{\alpha}, f) - (M, \hat{u}, \hat{\beta}, f) \\
	& (M_{1} \sqcup M_{2}, \hat{u}, \hat{\alpha}, f) - (M_{1}, \hat{u}\vert_{M_{1}}, \hat{\alpha}\vert_{M_{1}}, f\vert_{M_{1}}) - (M_{2}, \hat{u}\vert_{M_{2}}, \hat{\alpha}\vert_{M_{2}}, f\vert_{M_{2}}) \\
	& (M, \hat{u}, \varphi_{!}\hat{\alpha}, f) - (N, \hat{v}, \hat{\alpha}, f \circ \varphi)
\end{align*}
where $\varphi \colon N \to M$ is a submersion oriented via the 2x3 principle;
	\item the group of \emph{$k$-prechains}, denoted by $\hat{c}_{k}(X)$, as the free abelian group generated by the quadruples $(W, \hat{U}, \hat{A}, F)$ defined as in the first item, but allowing $W$ to have boundary; we set $\partial(W, \hat{U}, \hat{A}, F) := (\partial W, \hat{U}\vert_{\partial W}, \hat{A}\vert_{\partial W}, F\vert_{\partial W})$, the latter being a $(k-1)$-precycle;
	\item the group of \emph{$k$-boundaries}, denoted by $\hat{b}_{k}(X)$, as the subgroup of $\hat{z}_{k}(X)$ formed by the cycles that admit a representative of the form $\partial(W, \hat{U}, \hat{A}, F)$;\footnote{We observe that this is actually a subgroup, since we do not assume $W$ connected, so that $(W, \hat{U}, \hat{A}, F) - (W', \hat{U}', \hat{A}', F')$ is equivalent to $(W \sqcup W', \hat{U} \sqcup \hat{U}', \hat{A} \sqcup -\hat{A}', F \sqcup F')$.}
	\item the $k$-homology group $h_{k}(X) := \hat{z}_{k}(X) / \hat{b}_{k}(X)$. \hfill$\diamondsuit$
\end{itemize}
\end{Def2}
This model might seem redundant, since it is naturally isomorphic to the one with topological representatives through the transformation $[M, \hat{u}, \hat{\alpha}, f] \mapsto [M, I(\hat{u}), I(\hat{\alpha}), f]$. Nevertheless, differential representatives will be necessary as an intermediate step to refine homology. The situation is similar to the one of singular homology of a manifold described through \emph{smooth} chains: representatives carry differential information, but the corresponding homology groups are canonically isomorphic to the topological ones. In order to obtain a differential refinement of homology, we need currents as before.

We define the homological Chern character:
\begin{eqnarray}
	\ch \colon \; h_{\bullet}(X) & \to & H_{\bullet}(X; \h_{\R}) \label{HomCh} \\
	\lbrack M, \hat{u}, \hat{\alpha}, f \rbrack & \mapsto & \lbrack M, \ch(u), \ch(\alpha), f \rbrack \nonumber
\end{eqnarray}
where $\alpha := I(\hat{\alpha})$ and $u := I(\hat{u})$. This character induces the isomorphism
\begin{equation}\label{IsoHomR}
	h_{\bullet}(X) \otimes \R \simeq H_{\bullet}(X; \h_{\R}).
\end{equation}
In order to define the cap product, given a cohomology class $\beta$ and a homology class $[M, \hat{u}, \hat{\alpha}, f]$, we fix any differential refinement $\hat{\beta}$ and set
\begin{equation}\label{CapProductH}
	\beta \pmb{\cdot} [M, \hat{u}, \hat{\alpha}, f] := [M, \hat{u}, f^{*}\hat{\beta} \cdot \hat{\alpha}, f].
\end{equation}
Contrary to the case of singular homology, we put cohomology on the left (as in \cite{Jakob}), since this choice will lead to more natural sign conventions. From now on, we suppose that $h^{\bullet}$ is rationally-even.

\SkipTocEntry \subsection{Currents}

We define the following graded groups:
\begin{align}
	& \Tau^{k}(X; \h_{\R}) := \bigoplus_{h \in \Z} \Tau^{k-h}(X) \otimes \h^{h}_{\R} \nonumber \\
	& \Tau_{k}(X; \h_{\R}) := \Tau^{n-k}(X; \h_{\R}) = \bigoplus_{h \in \Z} \Tau^{n-k-h}(X) \otimes \h^{h}_{\R} = \bigoplus_{h \in \Z} \Tau_{k+h}(X) \otimes \h^{h}_{\R} \label{DefCurrentH}
\end{align}
Therefore, $T_{k} \in \Tau_{k}(X; \h_{\R})$ and $\omega^{k+h} \in \Omega^{k+h}(X)$ lead to $T_{k}(\omega^{k+h}) \in \h^{h}_{\R}$. Equivalently, $\omega^{k+h} := \omega^{k+h-l}_{0} \otimes x^{l} \in \Omega^{k+h}(X; \h_{\R})$ leads to $T_{k}(\omega^{k+h}) := T_{k}(\omega^{k+h-l}_{0}) \cdot x^{l} \in \h^{h}_{\R}$. Definition \eqref{DefCurrentH} is then equivalent to
\begin{equation}\label{DefCurrentH2}
	\Tau_{k}(X; \h_{\R}) := \Hom_{\h_{\R}}\bigl(\Omega^{k+\bullet}(X; \h_{\R}), \h^{\bullet}_{\R} \bigr),
\end{equation}
where $\Hom_{\h_{\R}}$ denotes morphisms of $\h_{\R}$-modules.

We define the exterior differential and the boundary of a current respectively as in \eqref{ExtDiffCurrents} and \eqref{DefBoundaryCurr}, so that \eqref{BoundaryDifferential} holds. Coherently with the convention on the cap product \eqref{CapProductH}, given a form $\eta \in \Omega^{\bullet}(X; \h_{\R})$ and a current $T \in \Tau_{\star}(X; \h_{\R})$, we define the current $\eta \wedge T \in \Tau_{\star-\bullet}(X; \h_{\R})$ as follows:
\begin{equation}\label{WedgeOtherSide}
	(\eta \wedge T)(\omega) := T(\omega \wedge \eta)
\end{equation}
We show in appendix \ref{ProofBdWedge} that
\begin{equation}\label{WedgeOtherSideBoundary}
	\partial(\eta^{h} \wedge T_{k}) = \eta^{h} \wedge \partial T_{k} + (-1)^{k+h} \, d\eta^{h} \wedge T_{k}
\end{equation}
If $X$ is oriented and $\omega^{k} \in \Omega^{k}_{\cpt}(X; \h_{\R})$, we define $\int_{X} \omega ^{k}$ as follows: we select the components $\sum \omega^{n}_{0} \otimes x^{k-n}$, where $n = \dim(X)$, and set $\int_{X} \omega^{k} := \sum (\int_{X} \omega^{n}_{0}) \cdot x^{k-n}$. In this way, a $k$-prechain $(W, \hat{U}, \hat{A}, F)$ induces the current $T_{(W, \hat{U}, \hat{A}, F)} \in T_{k}(X; \h_{\R})$ defined by
\begin{equation}\label{DefIndCurr}
	T_{(W, \hat{U}, \hat{A}, F)}(\omega) := \int_{W} F^{*}\omega \wedge R(\hat{A}) \wedge \Td(\hat{U}).
\end{equation}
Indeed, indicating explicitly the degrees, we integrate the $(q+k+h)$-form $F^{*}\omega^{k+h} \wedge R(\hat{A}^{q}) \wedge \Td(\hat{U})$ in the $(k+q)$-manifold $W$. The result belongs to $\h^{h}_{\R}$, coherently with \eqref{DefCurrentH2}. This leads to the natural morphism:
\begin{eqnarray}
	\ScT_{k} \colon \; \hat{c}_{k}(X) \otimes \R & \to & \Tau_{k}(X; \h_{\R}) \label{PsikH} \\
	(W, \hat{U}, \hat{A}, F) \otimes 1 & \mapsto & T_{(W, \hat{U}, \hat{A}, F)} \nonumber
\end{eqnarray}
The following lemma shows that the restriction of \eqref{PsikH} to precycles projects to homology.
\begin{Lemma}\label{LemmaProjHom} We have $\partial T_{(W, \hat{U}, \hat{A}, F)} = T_{\partial(W, \hat{U}, \hat{A}, F)}$. Moreover, the restriction of \eqref{PsikH} to precycles projects to cycles, hence it induces the isomorphism:
\begin{equation}\label{BarPsiIsoH}
	\bar{\ScT}_{k} \colon \; h_{k}(X) \otimes \R \overset{\!\simeq}\longrightarrow H\Tau_{k}(X; \h_{\R})
\end{equation}
\end{Lemma}
The proof can be found in appendix \ref{ProofLemmaProjHom}, and we explain the chosen conventions on signs and fibrewise integration in appendix \ref{ConvFibreInt}. For further reference, we observe that, given a prechain $(W, \hat{U}, \hat{A}, F)$ and a cohomology class $\hat{B}$ in $X$, we have
\begin{equation}\label{CurrentPrechainWedge}
	T_{(W, \hat{U}, F^{*}\hat{B} \cdot \hat{A}, F)} = R(\hat{B}) \wedge T_{(W, \hat{U}, \hat{A}, F)}
\end{equation}
as the reader can verify by direct computation from \eqref{WedgeOtherSide} and \eqref{DefIndCurr}.

\begin{Def}\label{DefIntFormH} A form $\omega \in \Omega^{\bullet}(X; \h_{\R})$ is \emph{integral} when it is closed and represents a class in the image of $\ch \colon h^{\bullet}(X) \to H^{\bullet}(X; \h_{\R})$.
\end{Def}

\begin{Lemma}\label{LemmaIntForms} An integral form has integral periods, that is, $T_{(M, \hat{u}, \hat{\alpha}, f)}(\omega) \in \IIm(\h^{\bullet} \to \h^{\bullet}_{\R})$ for every precycle $(M, \hat{u}, \hat{\alpha}, f)$, where $\h^{\bullet} \to \h^{\bullet}_{\R}$ coincides with the Chern character on the point.
\end{Lemma}

The proof is straightforward and can be found in appendix \ref{ProofLemmaIntForms}. Moreover, we define \emph{integral} currents as follows:
\begin{equation}\label{DefIntCurrH}
	\Tau_{k}^{\Int}(X; \h_{\R}) := \bigl\{ T_{(M, \hat{u}, \hat{\alpha}, f)} + \partial T'_{k+1}: (M, \hat{u}, \hat{\alpha}, f) \in \hat{p}_{k}(X), T'_{k+1} \in \Tau_{k+1}(X; \h_{\R}) \bigr\}
\end{equation}
Lemma \ref{LemmaIntForms} implies that an integral current satisfies $T_{k}(\omega^{k}) \in \IIm(\h^{\bullet} \to \h^{\bullet}_{\R})$ for every integral form $\omega^{k}$. Lastly, a $\hat{h}$-orientation $\hat{u}$ on an $n$-manifold $X$ without boundary induces the embedding
\begin{eqnarray*}
	\ScT^{k} \colon \Omega^{k}(X; \h_{\R}) & \hookrightarrow & \Tau^{k}(X; \h_{\R}) \\
	\eta^{k} & \mapsto & T_{\eta^{k}}
\end{eqnarray*}
where
\begin{equation}\label{DefTOmegaH}
	T_{\eta^{k}}(\omega^{n-k}) := \int_{X} \omega^{n-k} \wedge \eta^{k} \wedge \Td(\hat{u}).
\end{equation}
We have $dT_{\omega^{k}} = (-1)^{n-1} T_{d\omega^{k}}$, the proof being analogous to the one after formula \eqref{EmbJK}. By using the orientation $\hat{u}$ in a precycle $(M, \hat{u}, \hat{\alpha}, f)$, formula \eqref{DefIndCurr} can be stated as follows:
\begin{equation}\label{DefIndCurr2}
	T_{(M, \hat{u}, \hat{\alpha}, f)} := f_{*}T_{R(\hat{\alpha})}
\end{equation}

\SkipTocEntry \subsection{Differential Homology}

\begin{Def}\label{DefDiffHomGen} The \emph{differential homology group} $\hat{h}_{k}(X)$ is the quotient of
	\[\hat{z}_{k}(X) \times \Tau_{k+1}(X; \h_{\R})
\]
by the subgroup whose elements are of the form
	\[\bigl(\partial (W, \hat{U}, \hat{A}, F), -T_{(W, \hat{U}, \hat{A}, F)} - \partial T'_{k+2}\bigr)
\]
with $(W, \hat{U}, \hat{A}, F) \in \hat{c}_{k}(X)$ and $T'_{k+2} \in \Tau_{k+2}(X; \h_{\R})$.
\end{Def}

Given a smooth map $f \colon X \to Y$, the natural pushforward $f_{*} \colon \hat{h}_{k}(X) \to \hat{h}_{k}(Y)$ is defined by $f_{*}[M, \hat{u}, \hat{\alpha}, g, T_{k+1}] := [M, \hat{u}, \hat{\alpha}, f \circ g, f_{*}T_{k+1}]$. We define the following natural transformations:
\begin{align}
	& I \colon \hat{h}_{k}(X) \to h_{k}(X), & & [M, \hat{u}, \hat{\alpha}, f, T_{k+1}] \mapsto [M, \hat{u}, \hat{\alpha}, f] \label{DefIGen} \\
	& R \colon \hat{h}_{k}(X) \to \Tau_{k}(X; \h_{\R}), & & [M, \hat{u}, \hat{\alpha}, f, T_{k+1}] \mapsto T_{(M, \hat{u}, \hat{\alpha}, f)} + \partial T_{k+1} \label{DefRGen} \\
	& a \colon \Tau_{k+1}(X; \h_{\R}) \to \hat{h}_{k}(X), & & T_{k+1} \mapsto [0, T_{k+1}] \label{DefAGen}
\end{align}
We can easily check that they are well-defined, as we did about singular homology. The following lemma will be useful later on.
\begin{Lemma} We have
\begin{equation}\label{LemmaA}
	[M, \hat{u}, a(\theta), f, 0] = (-1)^{k+1}a(f_{*}T_{\theta})
\end{equation}
where $\theta \in \Omega^{\bullet}(M; \h_{\R})$ and $[M, \hat{u}, a(\theta), f, 0] \in \hat{h}_{k}(X)$. Similarly,
\begin{equation}\label{LemmaA2}
	[M, \hat{u} + a(\theta), \hat{\alpha}, f, 0] - [M, \hat{u}, \hat{\alpha}, f, 0] = (-1)^{m+1} \, a(f_{*}T_{R(\hat{\alpha}) \wedge (\int_{\mathcal{N}/M} \theta) \wedge \Td(\hat{u})^{-1}})
\end{equation}
where $\mathcal{N}$ is the normal bundle of $M$ in which $\hat{u}$ is a Thom class, $\theta \in \Omega^{\abs{\mathcal{N}}-1}_{\cpt}(\mathcal{N}; \h_{\R})$, and $m := \dim(M)$.
\end{Lemma}
\begin{proof} We set $\Ii := [0, 1]$ and $\pi_{\Ii} \colon \Ii \times M \to M$, $(t, x) \mapsto x$. With this notation:
	\[(M, \hat{u}, a(\theta), f) = \partial (\Ii \times M, \pi_{\Ii}^{*}\hat{u}, a(t \cdot \pi_{\Ii}^{*}\theta), f \circ \pi_{\Ii})
\]
Therefore, $[M, \hat{u}, a(\theta), f, 0] = [0, T_{(\Ii \times M, \pi_{\Ii}^{*}\hat{u}, a(t \cdot \pi_{\Ii}^{*}\theta), f \circ \pi_{\Ii})}]$, and we have:
\begin{align*}
	T_{(\Ii \times M, \pi_{\Ii}^{*}\hat{u}, a(t \cdot \pi_{\Ii}^{*}\theta), f \circ \pi_{\Ii})}(\omega) & \overset{\eqref{DefIndCurr}}= (-1)^{m} \int_{M \times \Ii} \pi_{\Ii}^{*}f^{*}\omega \wedge d(t \cdot \pi_{\Ii}^{*}\theta) \wedge \pi_{\Ii}^{*}\Td(\hat{u}) \\
	& \hspace{3pt} = \hspace{3pt} (-1)^{m+\abs{\theta}} \int_{M \times \Ii} \pi_{\Ii}^{*}\bigl( f^{*}\omega \wedge \theta \wedge \Td(\hat{u}) \bigr) \wedge dt \\
	& \hspace{3pt} = \hspace{3pt} (-1)^{k+1} \int_{M} f^{*}\omega \wedge \theta \wedge \Td(\hat{u}) \\
	& \overset{\eqref{DefTOmegaH}}= (-1)^{k+1} \, T_{\theta}(f^{*}\omega) = (-1)^{k+1} (f_{*}T_{\theta})(\omega)
\end{align*}
Similarly:
	\[(M, \hat{u} + a(\theta), \hat{\alpha}, f) - (M, \hat{u}, \hat{\alpha}, f) = \partial (\Ii \times M, \hat{U}, \pi_{\Ii}^{*}\hat{\alpha}, f \circ \pi_{\Ii})
\]
with $\hat{U} := \pi_{\Ii}^{*}\hat{u} + a(t \cdot \pi_{\Ii}^{*}\theta)$. Thus, the left-hand side of \eqref{LemmaA2} is $[0, T_{(\Ii \times M, \hat{U}, \pi_{\Ii}^{*}\hat{\alpha}, f \circ \pi_{\Ii})}]$. Since
\begin{align*}
	\Td(\hat{U}) & = \int_{\Ii \times \mathcal{N}/\Ii \times M} R(\hat{U}) = \int_{\Ii \times \mathcal{N}/\Ii \times M} \bigl( \pi_{\Ii}^{*}R(\hat{u}) + d(t \cdot \pi_{\Ii}^{*}\theta) \bigr) \\
	& = \pi_{\Ii}^{*}\Td(\hat{u}) + d \biggl( t \cdot \pi_{\Ii}^{*}\int_{\mathcal{N}/M} \theta \biggr),
\end{align*}
we have:
\begin{align*}
	T_{(\Ii \times M, \hat{U}, \pi_{\Ii}^{*}\hat{\alpha}, f \circ \pi_{\Ii})}(\omega) & \overset{\eqref{DefIndCurr}}= (-1)^{m+1} \int_{M \times \Ii} \pi_{\Ii}^{*} \biggl( f^{*}\omega \wedge R(\hat{\alpha}) \wedge \int_{\mathcal{N}/M} \theta \biggr) \wedge dt \\
	& \hspace{3pt} = \hspace{3pt} (-1)^{m+1} \int_{M} \biggl( f^{*}\omega \wedge R(\hat{\alpha}) \wedge \int_{\mathcal{N}/M} \theta \biggr) \\
	& \overset{\eqref{DefTOmegaH}}= (-1)^{m+1}\,  T_{R(\hat{\alpha}) \wedge (\int_{\mathcal{N}/X} \theta) \wedge \Td(\hat{u})^{-1}}(f^{*}\omega) \\
	& \hspace{3pt} = \hspace{3pt} (-1)^{m+1} (f_{*}T_{R(\hat{\alpha}) \wedge (\int_{\mathcal{N}/X} \theta) \wedge \Td(\hat{u})^{-1}})(\omega) \qedhere
\end{align*}
\end{proof}
We define the homological flat theory $\hat{h}_{\bullet}^{\fl}(X)$ as the dual of the cohomological one $\hat{h}^{\bullet}_{\fl}(X)$. The latter is a cohomology theory (see \cite[Section 5]{BS}), which is not multiplicative, but has a natural $h^{\bullet}$-module structure.\footnote{Given $\alpha \in h^{\bullet}(X)$ and $\hat{\beta} \in \hat{h}^{\bullet}_{\fl}(X)$, we set $\alpha \cdot \hat{\beta} := \hat{\alpha} \cdot \hat{\beta}$ for any differential refinement $\hat{\alpha}$ of $\alpha$. The choice of the refinement is immaterial, since $a(\omega) \cdot \hat{\beta} = a(\omega \wedge R(\hat{\beta})) = 0$.} Therefore, an element of $\hat{h}_{\bullet}^{\fl}(X)$ can be represented in the form $[M, u, \hat{\alpha}, f]$, where $u$ is a Thom class in $h^{\bullet}$ and $\hat{\alpha} \in \hat{h}^{\bullet}_{\fl}(X)$. When $\hat{h}^{\bullet}_{\fl}(\,\cdot\,) \simeq h^{\bullet-1}(\,\cdot\,; \R/\Z)$ (again, see \cite[Section 5]{BS}), we have $\hat{h}_{\bullet}^{\fl}(\,\cdot\,) \simeq h_{\bullet+1}(\,\cdot\,; \R/\Z)$ by construction. The following natural transformation will be proved to be injective:
\begin{eqnarray}
	\iota \colon \; \hat{h}_{k}^{\fl}(X) & \to & \hat{h}_{k}(X) \label{MorfIotaH} \\
	\lbrack M, u, \hat{\alpha}, f \rbrack & \mapsto & [M, \hat{u}, \hat{\alpha}, f, 0] \nonumber
\end{eqnarray}
Here $\hat{u}$ is any differential refinement of $u$. Let us show that $\iota$ is well-defined. The choice of $\hat{u}$ is immaterial because of formula \eqref{LemmaA2} with $R(\hat{\alpha}) = 0$. The independence of the precycle is straightforward, since additivity with respect to $M$ and $\hat{\alpha}$ is immediate and the Gysin map in the flat theory is a particular case of the one in the differential theory. Lastly, $[\partial (W, U, \hat{A}, F)] \mapsto [\partial (W, \hat{U}, \hat{A}, F), 0] = [0, T_{(W, \hat{U}, \hat{A}, F)}] = 0$ by formula \eqref{DefIndCurr} with $R(\hat{\alpha}) = 0$.

\begin{Prop}\label{PropExSeqH} The following square is well-defined and commutative:
	\[\xymatrix{
	\hat{h}_{k}(X) \ar@{->>}[rr]^{I} \ar@{->>}[d]_(.45){R} & & h_{k}(X) \ar[d]^(.475){\ch} \\
	\Tau_{k}^{\Int}(X; \h_{\R}) \ar[rr]^{\bar{\ScT}_{k}^{\,-1}} & & H_{k}(X; \h_{\R})
}\]
The isomorphism $\bar{\ScT}_{k}$ was defined in \eqref{BarPsiIsoH} and we are applying implicitly the projection from $\Tau_{k}^{\Int}(X; \h_{\R})$ to currential homology and the isomorphism \eqref{IsoHomR}. Moreover, we have
\begin{equation}\label{RAPartialh}
	R \circ a = \partial \hspace{50pt} I \circ \iota = B \hspace{50pt} \iota \circ p_{\fl} = (-1)^{k} \, a \circ \bar{\ScT}_{k}
\end{equation}
where $B$ and $p_{\fl}$ are part of the following exact sequence, induced by the corresponding cohomological one:
\begin{equation}\label{ExSeqFlat}
	\xymatrix{
	\cdots \ar[r] & h_{\bullet}(X) \ar[r]^(.4){- \otimes 1} & h_{\bullet}(X) \otimes \R \ar[r]^(.52){p_{\fl}} & \hat{h}_{\bullet+1}^{\fl}(X) \ar[r]^(.48){B} & h_{\bullet+1}(X) \ar[r] & \cdots
}
\end{equation}
Lastly, the following sequences are exact:
\begin{align}
	& \xymatrix{0 \ar[r] & \Tau_{k+1}^{\Int}(X; \h_{\R}) \ar[r] & \Tau_{k+1}(X; \h_{\R}) \ar[r]^(.58){a} & \hat{h}_{k}(X) \ar[r]^(.48){I} & h_{k}(X) \ar[r] & 0} \label{ExSeq1h} \\
	& \xymatrix{0 \ar[r] & \hat{h}_{k}^{\fl}(X) \ar[r]^{\iota} & \hat{h}_{k}(X) \ar[r]^(.4){R} & \Tau_{k}^{\Int}(X; \h_{\R}) \ar[r] & 0} \label{ExSeq2h}
\end{align}
\end{Prop}
\begin{proof} About the square, the upper-right path is
	\[[M, \hat{u}, \hat{\alpha}, f, T_{k+1}] \mapsto [M, \hat{u}, \hat{\alpha}, f] \mapsto \ch[M, \hat{u}, \hat{\alpha}, f],
\]
and the left-lower path is
	\[[M, \hat{u}, \hat{\alpha}, f, T_{k+1}] \mapsto T_{(M, \hat{u}, \hat{\alpha}, f)} + \partial T_{k+1} \mapsto \bar{\ScT}_{k}^{\,-1}[T_{(M, \hat{u}, \hat{\alpha}, f)}] = [M, \hat{u}, \hat{\alpha}, f] \otimes 1,
\]
so they coincide via \eqref{IsoHomR}. The morphism $I$ is surjective, since, given $[M, \hat{u}, \hat{\alpha}, f] \in h_{k}(X)$, we have $[M, \hat{u}, \hat{\alpha}, f] = I[M, \hat{u}, \hat{\alpha}, f, 0]$. Lastly, the surjectivity of $R$ immediately follows from definitions \eqref{DefRGen} and \eqref{DefIntCurrH}.

About \eqref{RAPartialh}, we have $R \circ a(T_{k+1}) = R[0, T_{k+1}] = \partial T_{k+1}$. Furthermore, we consider the following exact sequence in cohomology:
	\[\xymatrix{
	\cdots \ar[r] & h^{\bullet}(X) \ar[r]^(.27){- \otimes 1} & h^{\bullet}(X) \otimes \R \simeq H^{\bullet}(X; \h_{\R}) \ar[r]^(.67){a} & \hat{h}^{\bullet+1}_{\fl}(X) \ar[r]^(.48){I} & h^{\bullet+1}(X) \ar[r] & \cdots
}\]
By using the corresponding structure of $h^{\bullet}$-module, we represent $h_{\bullet}$, $h_{\bullet} \otimes \R$, and $\hat{h}_{\bullet}^{\fl}$ in the form $[M, u, \alpha, f]$, where $u$ is an $h^{\bullet}$-orientation and $\alpha$ belongs to the corresponding cohomology theory. In this way, by applying the morphisms of the previous sequence to $\alpha$, we obtain \eqref{ExSeqFlat} with the following morphisms: $[M, u, \alpha, f] \mapsto [M, u, \alpha \otimes 1, f] \simeq [M, u, \alpha, f] \otimes 1$; $p_{\fl}[M, u, [\omega], f] = [M, u, a(\omega), f]$ via \eqref{IsoHomR}; and $B[M, u, \hat{\alpha}, f] = [M, u, I(\hat{\alpha}), f]$. Then, by using the model with differential cycles, $B[M, u, \hat{\alpha}, f] = [M, \hat{u}, \hat{\alpha}, f] = I[M, \hat{u}, \hat{\alpha}, f, 0] = I \circ \iota[M, u, \hat{\alpha}, f]$. Also,
\begin{align*}
	\iota \circ p_{\fl}[M, u, [\omega], f] & \hspace{3pt} = \hspace{3pt} \iota[M, u, a(\omega), f] = [M, \hat{u}, a(\omega), f, 0] \\
	& \overset{\eqref{LemmaA}}= (-1)^{k} \, a(f_{*}T_{\omega}) = (-1)^{k} \, a \circ \bar{\ScT}_{k}[M, u, [\omega], f]
\end{align*}
the last equality following from the composition of $\bar{\ScT}_{k}$ with \eqref{IsoHomR}, since $[M, u, [\omega], f] = [M, u, \ch(\alpha), f] \simeq [M, u, \alpha, f] \otimes 1$ where $\ch(\alpha) = [\omega]$, so that
	\[\bar{\ScT}_{k}[M, u, [\omega], f] = \bar{\ScT}_{k}[M, \hat{u}, \hat{\alpha}, f] = [T_{(M, \hat{u}, \hat{\alpha}, f)}] \overset{\eqref{DefIndCurr2}}= f_{*}[T_{R(\hat{\alpha})}] = f_{*}[T_{\omega}]
\]
for any differential refinements $\hat{u}$ and $\hat{\alpha}$.

About \eqref{ExSeq1h}, we have $a(T_{k+1}) = [0, T_{k+1}] = 0$ if and only if there exist $(W, \hat{U}, \hat{A}, F)$ and $T'_{k+2}$ such that $(0, T_{k+1}) = (\partial(W, \hat{U}, \hat{A}, F), -T_{(W, \hat{U}, \hat{A}, F)} - \partial T'_{k+2})$. This means that there exist a \emph{precycle} $(W, \hat{U}, \hat{A}, F)$ and a current $T'_{k+2}$ such that $T_{k+1} = -T_{(W, \hat{U}, \hat{A}, F)} - \partial T'_{k+2}$, that is, $T_{k+1}$ is integral by definition \eqref{DefIntCurrH}. Moreover, $I \circ a(T_{k+1}) = I[0, T_{k+1}] = [0] = 0$. Conversely, if $I[M, \hat{u}, \hat{\alpha}, f, T_{k+1}] = [M, \hat{u}, \hat{\alpha}, f] = 0$, then $[M, \hat{u}, \hat{\alpha}, f] = [\partial (W, \hat{U}, \hat{A}, F)]$, thus $[M, \hat{u}, \hat{\alpha}, f, T_{k+1}] = [\partial (W, \hat{U}, \hat{A}, F), T_{k+1}] = [0, T_{(W, \hat{U}, \hat{A}, F)} + T_{k+1}] = a(T_{(W, \hat{U}, \hat{A}, F)} + T_{k+1})$. Lastly, we already proved the surjectivity of $I$.

About \eqref{ExSeq2h}, we already showed that $R$ is surjective. Definition \eqref{MorfIotaH} and formula \eqref{DefIndCurr2} immediately imply $R \circ \iota = 0$, hence it remains to prove that $\iota \colon \hat{h}_{k}^{\fl}(X) \to \Ker(R_{k})$ is an isomorphism. We have the following morphism of long exact sequences, the upper one being \eqref{ExSeqFlat} up to a sign in $p_{\fl}$: \pagebreak
\begin{equation}\label{ExactSeqFlDiag}
	\xymatrix{
	\cdots \ar[r] & h_{\bullet}(X) \ar[r]^(.4){- \otimes 1} \ar@{=}[d] & h_{\bullet}(X) \otimes \R \ar[r]^(.52){(-1)^{\bullet} \, p_{\fl}} \ar[d]_{\simeq}^{\bar{\ScT}_{\bullet}} & \hat{h}_{\bullet+1}^{\fl}(X) \ar[r]^(.48){B} \ar[d]^{\iota} & h_{\bullet+1}(X) \ar[r] \ar@{=}[d] & \cdots \\
	\cdots \ar[r] & h_{\bullet}(X) \ar[r]^(.39){\ch} & H\Tau_{\bullet}(X; \h_{\R}) \ar[r]^(.51){a} & \Ker(R_{\bullet+1}) \ar[r]^(.52){I} & h_{\bullet+1}(X) \ar[r] & \cdots
}\end{equation}
Commutativity follows from \eqref{RAPartialh}, hence the five lemma implies the result.\footnote{Alternatively, one can prove that $\Ker(R_{\bullet})$ is a homology theory and $\iota \colon \hat{h}_{\bullet}^{\fl}(X) \to \Ker(R_{\bullet})$ is a natural transformation inducing an isomorphism on the point.}
\end{proof}
Corollary \ref{CorExSeq3} can be easily adapted to the present framework, so that we obtain the following commutative hexagon with exact diagonals:
\begin{equation}\label{CommHexh}
	\resizebox{0.8\textwidth}{!}{
	\xymatrix{
	& \frac{\Tau_{\bullet+1}(X; \h_{\R})}{\Tau^{\Int}_{\bullet+1}(X; \h_{\R})} \ar[rr]^{\partial} \ar@{^(->}[dr]^{a} & & \Tau^{\Int}_{\bullet}(X; \h_{\R}) \ar[dr]^{\bar{\ScT}^{-1}} \\
	\frac{h_{\bullet+1}(X) \otimes \R}{h_{\bullet+1}(X)} \ar@{^(->}[ur]^{\bar{\ScT}} \ar@{^(->}[dr]^{p_{\fl}} & & \hat{h}_{\bullet}(X) \ar@{->>}[ur]^{R} \ar@{->>}[dr]^{I} & & h_{\bullet}(X) \otimes \R \\
	& \hat{h}_{\bullet}^{\fl}(X) \ar@{^(->}[ur]^{\iota} \ar[rr]^{B} & & h_{\bullet}(X) \ar[ur]^{- \otimes 1}
}}
\end{equation}

\SkipTocEntry \subsection{Differential Cap Product and Poincar\'e Duality}

We define the differential cap product
	\[\pmb{\cdot}\, \colon \hat{h}^{s}(X) \times \hat{h}_{l}(X) \to \hat{h}_{l-s}(X)
\]
as follows:
\begin{equation}\label{DiffCapPH}
	\hat{\beta} \pmb{\cdot} [M, \hat{u}, \hat{\alpha}, f, T_{l+1}] := [M, \hat{u}, f^{*}\hat{\beta} \cdot \hat{\alpha}, f, R(\hat{\beta}) \wedge T_{l+1}]
\end{equation}
Let us show that it is well-defined. Additivity with respect to $M$ and $\hat{\alpha}$ is immediate. Given a submersion $\varphi \colon M \to N$ and two precycles $(M, \hat{u}, \hat{\alpha}, f \circ \varphi)$ and $(N, \hat{v}, \varphi_{!}\hat{\alpha}, f)$, we have $(M, \hat{u}, \varphi^{*}f^{*}\hat{\beta} \cdot \hat{\alpha}, f \circ \varphi) \sim (N, \hat{v}, \varphi_{!}(\varphi^{*}f^{*}\hat{\beta} \cdot \hat{\alpha}), f) = (N, \hat{v}, f^{*}\beta \cdot \varphi_{!}\hat{\alpha}, f)$, thus \eqref{DiffCapPH} remains unchanged. Lastly:
\begin{align*}
	\hat{\beta} \pmb{\cdot} [\partial(W, \hat{U}, &\hat{A}, F), -T_{(W, \hat{U}, \hat{A}, F)} - \partial T'_{k+2}] \\
	& \overset{\eqref{DiffCapPH}}= \hspace{10pt} [\partial W, \hat{U}\vert_{\partial W}, F\vert_{\partial W}^{*}\hat{\beta} \cdot \hat{A}\vert_{\partial W}, F\vert_{\partial W}, - R(\hat{\beta}) \wedge (T_{(W, \hat{U}, \hat{A}, F)} + \partial T'_{k+2})] \\
	& \hspace{-8pt} \overset{\eqref{CurrentPrechainWedge}, \eqref{WedgeOtherSideBoundary}}= [\partial (W, \hat{U}, F^{*}\hat{\beta} \cdot \hat{A}, F), -T_{(W, \hat{U}, F^{*}\hat{\beta} \cdot \hat{A}, F)} - \partial(R(\hat{\beta}) \wedge T'_{k+2})] = 0
\end{align*}

\begin{Prop}\label{AxiomsDiffCapH} The differential cap product $\,\pmb{\cdot}\, \colon \hat{h}^{s}(X) \times \hat{h}_{l}(X) \to \hat{h}_{l-s}(X)$ satisfies the following axioms:
\begin{enumerate}
	\item it is bi-additive;
	\item it is compatible with the natural transformations $I$, $R$, $a$, and $\iota$ as follows:
	\begin{align*}
		& I(\hat{\beta} \pmb{\cdot} \hat{\lambda}) = I(\hat{\beta}) \pmb{\cdot} I(\hat{\lambda}) & & R(\hat{\beta} \pmb{\cdot} \hat{\lambda}) = R(\hat{\beta}) \wedge R(\hat{\lambda}) \\
		& \hat{\beta} \pmb{\cdot} a(T) = a(R(\hat{\beta}) \wedge T) & & a(\eta) \pmb{\cdot} \hat{\lambda} = (-1)^{s+l+1} \, a(\eta \wedge R(\hat{\lambda})) \\
		& \hat{\beta} \pmb{\cdot} \iota(\hat{\lambda}) = \iota(I(\hat{\beta}) \cdot \hat{\lambda}) & & \hat{\beta} \pmb{\cdot} \hat{\lambda} = \iota(\hat{\beta} \cdot I(\hat{\lambda})) \textnormal{ if } R(\hat{\beta}) = 0
	\end{align*}
	\item it is mixed-associative with cohomology, that is, $(\hat{\gamma} \cdot \hat{\beta}) \pmb{\cdot} \hat{\lambda} = \hat{\gamma} \pmb{\cdot} (\hat{\beta} \pmb{\cdot} \hat{\lambda})$;
	\item it is functorial, that is, given a smooth function $g \colon X \to Y$, we have
	\[g_{*}(g^{*}\hat{\beta} \pmb{\cdot} \hat{\lambda}) = \hat{\beta} \pmb{\cdot} (g_{*}\hat{\lambda}).
\]
\end{enumerate}
\end{Prop}
\begin{proof} One can immediately deduce bi-additivity from definition \eqref{DiffCapPH}, compatibility with $I$ from formula \eqref{CapProductH}, and compatibility with $R$ from formulas \eqref{CurrentPrechainWedge} and \eqref{WedgeOtherSideBoundary}. About $a$, we have
	\[\hat{\beta} \pmb{\cdot} a(T) = \hat{\beta} \pmb{\cdot} [0, T] \overset{\eqref{DiffCapPH}}= [0, R(\hat{\beta}) \wedge T] = a(R(\hat{\beta}) \wedge T),
\]
where the class $0$ can be represented by any precycle of the form $(M, \hat{u}, 0, f)$. Furthermore:
\begin{align*}
	a(\eta) \pmb{\cdot} [M, \hat{u}, \hat{\alpha}, f, T] & \overset{\eqref{DiffCapPH}}= [M, \hat{u}, a(f^{*}\eta \wedge R(\hat{\alpha})), f, d\eta \wedge T] \\
	& \overset{\eqref{LemmaA}}= [0, (-1)^{s+l+1} \, f_{*}T_{f^{*}\eta \wedge R(\hat{\alpha})} + d\eta \wedge T] \\
	& \hspace{3pt} = \hspace{3pt} [0, (-1)^{s+l+1} \, \eta \wedge f_{*}T_{R(\hat{\alpha})} + d\eta \wedge T] \\
	& \overset{\eqref{WedgeOtherSideBoundary}}= (-1)^{s+l+1} \, [0, \eta \wedge f_{*}T_{R(\hat{\alpha})} + \eta \wedge \partial T] \\
	& \overset{\eqref{LemmaA}}= (-1)^{s+l+1} \, a(\eta \wedge R[M, \hat{u}, \hat{\alpha}, f, T])
\end{align*}
The formulas with $\iota$ and the last two items are straightforward computations.
\end{proof}

Let $(X, \hat{u})$ be a compact $\hat{h}$-oriented $n$-manifold. We define the \emph{fundamental class} $[X] := [X, \hat{u}, 1, \id_{X}, 0]$ and state the \emph{differential Poincar\'e duality} as follows:
\begin{eqnarray}
	\hatPD \colon \; \hat{h}^{k}(X) & \to & \hat{h}_{n-k}(X) \label{DiffPDH} \\
	\hat{\beta} & \mapsto & \hat{\beta} \,\pmb{\cdot}\, [X] \nonumber
\end{eqnarray}
It follows from definition \eqref{DiffCapPH} that $\hatPD(\hat{\beta}) = [X, \hat{u}, \hat{\beta}, \id_{X}, 0]$.

\begin{Prop}\label{ProofPDH} The morphism \eqref{DiffPDH} is injective and its image is formed by the elements of $\hat{h}_{n-k}(X)$ whose curvature is a form.
\end{Prop}

The proof is similar to the one of proposition \ref{ProofPD}, by starting from the following diagram:
\begin{equation}\label{DiagPD2H}
	\xymatrix{
	0 \ar[r] & \hat{h}^{k}_{\fl}(X) \ar@{^(->}[r] \ar[d]^(.46){\PD}_(.46){\simeq} & \hat{h}^{k}(X) \ar[r]^(.43){R} \ar[d]^(.43){\hatPD} & \Omega^{k}_{\Int}(X; \h_{\R}) \ar[r] \ar@{^(->}[d]^(.45){T^{k}} & 0 \\
	0 \ar[r] & \hat{h}_{n-k}^{\fl}(X) \ar[r]^{\iota} & \hat{h}_{n-k}(X) \ar[r]^(.43){R} & \Tau_{n-k}^{\Int}(X; \h_{\R}) \ar[r] & 0
}\end{equation}
Equivalently, one can consider the following diagram:
\begin{equation}\label{DiagPD1H}
	\xymatrix{
	0 \ar[r] & \frac{\Omega^{k-1}(X; \h_{\R})}{\Omega^{k-1}_{\Int}(X; \h_{\R})} \ar[r]^{a} \ar@{^(->}[d]^(.45){(-1)^{k-1} \, T^{k-1}} & \hat{h}^{k}(X) \ar[r]^{I} \ar[d]^(.43){\hatPD} & h^{k}(X) \ar[r] \ar[d]^(.46){\PD}_(.46){\simeq} & 0 \\
	0 \ar[r] & \frac{\Tau_{n-k+1}(X; \h_{\R})}{\Tau_{n-k+1}^{\Int}(X; \h_{\R})} \ar[r]^{a} & \hat{h}_{n-k}(X) \ar[r]^{I} & h_{n-k}(X) \ar[r] & 0
}\end{equation}
As in the case of singular homology, we can turn $\hatPD$ into an isomorphism by restricting homology to classes whose curvature is a form or by extending cohomology to classes whose curvature is a current. The latter possibility is more interesting and can be realised by adapting definition \ref{DefRedHom} in the natural way.

Lastly, one can define the differential external and slant products similarly to the case of singular homology, that is, $[M, \hat{u}, \hat{\alpha}, f, T] \hattimes [N, \hat{v}, \hat{\beta}, g, U]$ is defined as
\begin{equation}
	(-1)^{\abs{M} \cdot \abs{\hat{\beta}}} \bigl[M \times N, \pi_{M}^{*}\hat{u} \cdot \pi_{N}^{*}\hat{v}, \pi_{M}^{*}\hat{\alpha} \cdot \pi_{N}^{*}\hat{\beta}, f \times g, T_{(M, \hat{u}, \hat{\alpha}, f)} \times U + T \times T_{(N, \hat{v}, \hat{\beta}, g)} + T \times \partial U \bigr]
\end{equation}
and $\hat{\alpha} \hatslant \hat{\lambda} := (\pi_{X})_{*}(\pi_{Y}^{*}\hat{\alpha} \pmb{\cdot} \hat{\lambda})$. We obtain the \emph{$S^{1}$-differentiation map}:
\begin{eqnarray}
	\partial_{S^{1}} \colon \hat{h}_{\bullet}(X) & \to & \hat{h}_{\bullet+1}(X \times S^{1}) \label{DiffS1h} \\
	\hat{\lambda} & \mapsto & \hat{\lambda} \hattimes [S^{1}] \nonumber
\end{eqnarray}
 where the orientation of $S^{1}$ is the trivial one on $TS^{1} \simeq S^{1} \times \R$. Proposition \ref{PropS1Diff} holds essentially with the same statement. We do not show the proof here, since it is a particular case of the analogous statement about the Gysin map in homology (see proposition \ref{PropGysinHom} and corollary \ref{CorGysinTrivial} below).

\section{Gysin Maps and Uniqueness}\label{GysinMapHomSec}

Let us fix two $\hat{h}$-oriented compact manifolds $(X, \hat{u})$ and $(Y, \hat{v})$. We set $m := \dim(X)$, $k := \dim(Y)$, and $d := k-m$. Given a neat submersion $\varphi \colon Y \to X$, oriented through the 2x3 principle, the following diagram commutes:
\begin{equation}\label{CommDiagGysinPD}
	\xymatrix{
	\hat{h}^{\bullet}(Y) \ar[rr]^{\varphi_{!}} \ar[d]_(.45){\hatPD} & & \hat{h}^{\bullet-d}(X) \ar[d]^(.45){\hatPD} \\
	\hat{h}_{k-\bullet}(Y) \ar[rr]^{\varphi_{*}} & & \hat{h}_{k-\bullet}(X)
}\end{equation}
Indeed:
	\[\varphi_{*} \circ \hatPD(\hat{\alpha}) = \varphi_{*}[Y, \hat{v}, \hat{\alpha}, \id_{Y}, 0] = [Y, \hat{v}, \hat{\alpha}, \varphi, 0] = [X, \hat{u}, \varphi_{!}\hat{\alpha}, \id_{X}, 0] = \hatPD(\varphi_{!}\hat{\alpha})
\]
Without the compactness hypothesis, the same result holds by assuming $\varphi$ proper and considering non-compactly supported homology, whose definition we will sketch below, deferring the details to a future paper.

We observe that the morphism $\hat{\alpha} \mapsto \hatPD^{-1} \circ \varphi_{*} \circ \hatPD(\hat{\alpha})$ is well-defined in currential cohomology even when $\varphi$ is not a submersion (we need currents because of the pushforward $\varphi_{*}$). Moreover, it is functorial and depends only on the orientations of $X$ and $Y$. Therefore, it is natural to inquire if there is a suitable definition of $\varphi_{!}$ that makes the previous diagram commute for every smooth proper map $\varphi$. Let us show that this is possible.

\SkipTocEntry \subsection{Currential Thom Morphism}

We fix a smooth vector bundle $\pi \colon E \to X$ of rank $n$ and denote by $i \colon X \to E$ the embedding of $X$ as the zero-section (which is proper). Also, we denote by $\bar{\Tau}^{\bullet}$ the space of currents with no hypotheses on the support. A \emph{strict differential Thom class} on $E$ is a currential class $\check{u} \in \check{h}^{n}_{\vcpt}(E)$ such that:\footnote{About currential cohomology, see definition \ref{DefRedHom} adapted to a generalised cohomology theory.}
\begin{itemize}
	\item $u := I(\check{u})$ is a Thom class;
	\item there exists a closed current $\Td(\check{u}) \in \bar{\Tau}^{0}_{\cl}(X; \h_{\R})$ representing $\Td(u)$ and satisfying the condition $R(\check{u}) = i_{*}\Td(\check{u})$.
\end{itemize}
The degrees are consistent, since $\Td(\check{u}) \in \bar{\Tau}^{0}(X; \h_{\R})$ implies $i_{*}\Td(\check{u}) \in \bar{\Tau}^{n}_{\vcpt}(E)$. We obtain the corresponding Thom morphism:
\begin{eqnarray}
	\Thom \colon \hat{h}^{\bullet}(X) & \hookrightarrow & \check{h}^{\bullet+n}_{\vcpt}(E) \label{ThomMorphFC} \\
	\hat{\alpha} & \mapsto & \pi^{*}\hat{\alpha} \cdot \check{u} \nonumber
\end{eqnarray}
Here we are using the natural left-module structure of currential cohomology over cohomology, that is, $\hat{\beta} \cdot [\hat{\alpha}, T] := [\hat{\beta} \cdot \hat{\alpha}, R(\hat{\beta}) \wedge T]$.

We now assume that the base-manifold $X$ is $\hat{h}$-oriented, so that we can embed forms in currents. By fixing a form that represents $\Td(u)$, denoted by $\Td(\hat{u}) \in \Omega^{0}(X; \h_{\R})$, the morphism \eqref{ThomMorphFC} can be extended to currential cohomology in the domain. In fact, we set $\Td(\check{u}) := \ScT^{0}\Td(\hat{u})$ and define:
\begin{eqnarray}
	\Thom \colon \; \check{h}^{\bullet}(X) & \hookrightarrow & \check{h}^{\bullet+n}_{\vcpt}(E) \label{ThomMorphCC} \\
	\lbrack \hat{\alpha}, T \rbrack & \mapsto & \pi^{*}\hat{\alpha} \cdot \check{u} + a\bigl(i_{*}(T \wedge \Td(\hat{u}))\bigr) \nonumber
\end{eqnarray}
The previous formula is motivated by observing that $[\hat{\alpha}, T] = [\hat{\alpha}, 0] + a(T)$ and, in principle, $a(\pi^{*}T) \cdot \check{u} = a(\pi^{*}T \wedge R(\check{u}))$. Nevertheless, the product of currents $\pi^{*}T \wedge R(\check{u})$ is not defined, nor the pull-back $\pi^{*}T$. Hence, we replace $\pi^{*}T$ by $i_{*}T$ and obtain $i_{*}T \wedge i_{*}\ScT^{0}\Td(\hat{u})$, which we \emph{define} to be $i_{*}(T \wedge \Td(\hat{u}))$.

\begin{Rmk*} Let us show that the morphism \eqref{ThomMorphCC} is well-defined. We first observe that, given a current $T \in \bar{\Tau}^{\star}(X)$ and a form $\eta \in \Omega^{\bullet}(X)$, the following equalities hold:
\begin{align}
	& i_{*}(\eta \wedge T) = \pi^{*}\eta \wedge i_{*}T \label{PushFWedge} \\
	& \eta \wedge \ScT^{0}\Td(\hat{u}) = \ScT^{\bullet}\eta \wedge \Td(\hat{u}) \label{WedgeTodd}
\end{align}
The proof can be found in appendix \ref{LemmaCurrBundleProof}. By definition of currential cohomology, we have $[a(\eta), 0] = [0, \ScT^{\bullet}\eta]$. Coherently:
\begin{flalign*}
	& & \Thom[a(\eta), 0] &\overset{\eqref{ThomMorphCC}}= \pi^{*}a(\eta) \cdot \check{u} = a\bigl(\pi^{*}\eta \wedge R(\check{u})\bigr) = a\bigl(\pi^{*}\eta \wedge i_{*}\ScT^{0}\Td(\hat{u})\bigr) \\
	& & & \overset{\eqref{PushFWedge}}= a\bigl(i_{*}(\eta \wedge \ScT^{0}\Td(\hat{u}))\bigr) \overset{\eqref{WedgeTodd}}= a\bigl(i_{*}(\ScT^{\bullet}\eta \wedge \Td(\hat{u}))\bigr) \overset{\eqref{ThomMorphCC}}= \Thom[0, \ScT^{\bullet}\eta] & & \diamondsuit
\end{flalign*}
\end{Rmk*}

Since $R(\check{u}) = i_{*}\ScT^{0}\Td(\hat{u})$, we have $R(\check{u})(\omega) = \int_{X} i^{*}\omega \wedge \Td(\hat{u}) \wedge \Td(\hat{v})$ where $\hat{v}$ is the orientation of $X$. In particular, in the setting of singular cohomology, we can choose $\Td(\hat{u}) = \Td(\hat{v}) = 1$ and obtain the Dirac delta on $X$, that is, $(i_{*}1)(\omega) = \int_{X} \omega$.

\begin{Lemma} By considering respectively \eqref{ThomMorphFC} and \eqref{ThomMorphCC}, we have
\begin{equation}\label{CurvatureThom}
	R(\Thom(\hat{\alpha})) = i_{*}(R(\hat{\alpha}) \wedge \Td(\check{u})) \hspace{30pt} R(\Thom(\check{\alpha})) = i_{*}(R(\check{\alpha}) \wedge \Td(\hat{u}))
\end{equation}
\end{Lemma}
\begin{proof} About the left-hand side, formula \eqref{ThomMorphFC} implies $R(\Thom(\hat{\alpha})) = \pi^{*}R(\hat{\alpha}) \wedge R(\check{u}) = \pi^{*}R(\hat{\alpha}) \wedge i_{*}\Td(\check{u})$, hence formula \eqref{PushFWedge} implies the result. On the right-hand side, we set $\check{\alpha} = [\hat{\alpha}, T]$. From the left-hand side and formulas \eqref{ThomMorphCC} and \eqref{WedgeTodd} we obtain $R(\Thom(\check{\alpha})) = i_{*}\bigl(R(\hat{\alpha}) \wedge \ScT^{0}\Td(\hat{u}) + \partial T \wedge \Td(\hat{u})\bigr) = i_{*}\bigl((\ScT^{\bullet}R(\hat{\alpha}) + \partial T) \wedge \Td(\hat{u})\bigr) = i_{*}(R[\hat{\alpha}, T] \wedge \Td(\hat{u}))$.
\end{proof}

\SkipTocEntry \subsection{Currential Gysin Map}

Given a smooth map between compact manifolds $\varphi \colon Y \to X$, we define a \emph{currential orientation} of $\varphi$ as usual: up to homotopy and stabilisation, we fix a neat embedding $\iota \colon Y \hookrightarrow X \times \R^{N}$ such that $\pi_{X} \circ \iota = \varphi$, a \emph{strict} Thom class $\check{u}$ of the normal bundle $\mathcal{N} := \iota^{*}\mathcal{N}_{\iota(Y)}(X \times \R^{N})$, a tubular neighbourhood $j \colon U \hookrightarrow X \times \R^{N}$ of $\iota(Y)$, and a diffeomorphism $\phi \colon \mathcal{N} \to U$. We obtain the Gysin map $\varphi_{!} \colon \hat{h}^{\bullet}(Y) \to \check{h}^{\bullet-d}(X)$ where $d := \dim(Y) - \dim(X)$, defined through \eqref{ThomMorphFC} as follows:
\begin{equation}\label{GysinMap}
	\varphi_{!}(\hat{\alpha}) := \int_{X \times \R^{N}/X} j_{*}\phi_{*}\Thom(\hat{\alpha})
\end{equation}
where the pushforward in currential compactly-supported cohomology is defined componentwise (for example, $j_{*}[\hat{\beta}, T] := [j_{*}\hat{\beta}, j_{*}T]$) and, by calling $\pi_{X} \colon X \times \R^{N} \to X$ the projection, we set $\int_{X \times \R^{N}/X} [\hat{\beta}, T] := \bigl[\int_{X \times \R^{N}/X} \hat{\beta}, (\pi_{X})_{*}T\bigr]$. If also $X$ is $\hat{h}$-oriented and $\Td(u)$ is represented by a form, then the Gysin map in currential cohomology $\varphi_{!} \colon \check{h}^{\bullet}(Y) \to \check{h}^{\bullet-d}(X)$ is defined through \eqref{ThomMorphCC} with the same formula \eqref{GysinMap}. This is the interesting case for us. Without the compactness hypothesis, we assume $\varphi$ proper and we have to choose the tubular neighbourhood in such a way that $j \circ \phi \colon \mathcal{N} \to X \times \R^{N}$ is vertically proper, so that the push-forward in verically-compactly-supported cohomology is well-defined.

\begin{Prop} The curvature map on currents
\begin{equation}\label{CurvMap}
	\varphi_{!}(T) := \varphi_{*}(T \wedge \Td(\hat{u}))
\end{equation}
satisfies the identities $R(\varphi_{!}(\check{\alpha})) = \varphi_{!}(R(\check{\alpha}))$ and $\varphi_{!}(a(T)) = a(\varphi_{!}(T))$.
\end{Prop}
\begin{proof} Since $\pi_{X} \circ j \circ \phi \circ i = \pi_{X} \circ j = \varphi$, we have:
\begin{equation}\label{CurvMap2}
	\varphi_{!}(T) = (\pi_{X})_{*} j_{*} \phi_{*} i_{*} \bigl(T \wedge \Td(\hat{u})\bigr)
\end{equation}
About curvature, definition \eqref{GysinMap} implies $R(\varphi_{!}(\check{\alpha})) = \int_{X \times \R^{N}/X} j_{*}\phi_{*}R(\Thom(\hat{\alpha}))$. Hence, the result follows from formula \eqref{CurvatureThom} by showing that $(\pi_{X})_{*} = \int_{X \times \R^{N}/X}$ on forms, that is:
\begin{equation}\label{PiXForms}
	(\pi_{X})_{*}\ScT^{\bullet}\eta = \ScT^{\bullet-N} \biggl(\int_{X \times \R^{N}/X} \eta\biggr)
\end{equation}
In fact, $((\pi_{X})_{*}\ScT^{\bullet}\eta)(\omega) = (\ScT^{\bullet}\eta)(\pi^{*}\omega) = \int_{X \times \R^{N}} \pi^{*}\omega \wedge \eta \wedge \pi^{*}\Td(\hat{v}) = \int_{X} \omega \wedge \bigl(\int_{X \times \R^{N}/X}\eta\bigr) \wedge \Td(\hat{v}) = \bigl(\ScT^{\bullet-N}\bigl(\int_{X \times \R^{N}/X} \eta\bigr)\bigr)(\omega)$. Lastly, the identity about $a$ immediately follows from formula \eqref{ThomMorphCC}.
\end{proof}

\begin{Rmk} Let us compare \eqref{CurvMap2} with the corresponding one on forms, that is:
	\[\varphi_{!}(\eta) = \int_{X \times \R^{N}/X} j_{*} \phi_{*} \bigl(\pi^{*}\eta \wedge R(\hat{u})\bigr)
\]
In \eqref{CurvMap}, the term $\eta \wedge R(\hat{u})$ is replaced by $i_{*}(T \wedge \Td(\hat{u}))$. This is coherent because of formula \eqref{PushFWedge}, since, if $T = \ScT^{\bullet}\eta$, then $i_{*}(\ScT^{\bullet}\eta \wedge \Td(\hat{u})) = \pi^{*}\eta \wedge i_{*} \Td(\hat{u}) = \pi^{*}\eta \wedge R(\check{u})$. Moreover, the integration $X \times \R^{N}/X$ is replaced by the pushforward $(\pi_{X})_{*}$ on currents, coherently with formula \eqref{PiXForms}.
\end{Rmk}

With forms, the identity analogous to \eqref{CurvMap} holds only by assuming that $\varphi$ is a submersion, since, under this hypothesis, we can choose $U$ and $\phi$ in such a way that $\phi(\mathcal{N}_{y}) \subset \{\iota(y)\} \times \R^{N}$ for every $y \in Y$. Using a strict Todd class on $\mathcal{N}$, this problem disappears, because the support of $R(\check{u})$ is completely contained in $Y$, thus the behaviour of the tubular neighbourhood is immaterial with respect to the curvature. Since topologically any two tubular neighbourhoods are equivalent, the Gysin map behaves naturally even if $\varphi$ is not a submersion.

Moreover, the fact that $\varphi_{!}(T)$ only depends on $\Td(\hat{u})$ implies that two representatives of an orientation of $\varphi$ are equivalent essentially if and only if the two corresponding Thom classes $\check{u}$ are homotopic. This allows to find a natural classification of the differential orientations refining a fixed topological one, which is essentially the same classification on the normal bundle. This suggests that currential cohomology is the natural setting to orient smooth maps in general, without supposing they are submersions (see problem 4.230 in \cite{Bunke}).

Lastly, a smooth proper map $\varphi \colon Y \to X$ between two oriented manifolds $(X, \hat{u})$ and $(Y, \hat{v})$ inherits an orientation as usual. In fact, we fix any $\iota \colon Y \hookrightarrow X \times \R^{N}$, so that $\iota^{*}T(X \times \R^{N}) \simeq TY \oplus \mathcal{N}$. Also, we orient $\mathcal{N}$ through the 2x3 rule and choose the representative $\check{w}$ such that $R(\check{w}) = i_{*}(\Td(\hat{u}) \wedge \Td(\hat{v})^{-1})$. By fixing any tubular neighbourhood $U$ and any diffeomorphism $\phi \colon \mathcal{N} \to U$, we obtain the orientation $[\iota, \hat{w}, \phi]$ of $\varphi$, so that $\varphi_{!}$ is well-defined.

\SkipTocEntry \subsection{Gysin Map and Differential Homology}

Precycles were defined without currents, and we cannot replace them with the currential version, since, in order to define $\partial(W, \hat{U}, \check{A}, F)$, we should restrict $\check{A}$ to $\partial W$, while we cannot restrict currents. Nevertheless, we give the following definition:
\begin{equation}\label{DefPrecCurr}
	[M, \hat{u}, [\hat{\alpha}, T], f, U] := [M, \hat{u}, \hat{\alpha}, f, (-1)^{k+1} f_{*}T + U]
\end{equation}
This formula is coherent with the identity $[a(\omega), 0] = [0, T_{\omega}]$, since
	\[[M, \hat{u}, [a(\omega), 0], f, 0] \overset{\eqref{DefPrecCurr}}= [M, \hat{u}, a(\omega), f, 0] \overset{\eqref{LemmaA}}= [0, (-1)^{k+1} f_{*}T_{\omega}] \overset{\eqref{DefPrecCurr}}= [M, \hat{u}, [0, T_{\omega}], f, 0].
\]
\begin{Prop} Using definition \eqref{DefPrecCurr}, we have
\begin{equation}\label{GysinAnyMap}
	[M, \hat{u}, \varphi_{!}\check{\alpha}, f, 0] = [N, \hat{v}, \check{\alpha}, f \circ \varphi, 0]
\end{equation}
for any smooth map $\varphi \colon N \to M$ oriented via the 2x3 principle and any currential class $\check{\alpha}$ on $N$.
\end{Prop}
\begin{proof} We call $\hat{h}'_{\bullet}(X)$ the quotient of $\hat{h}_{\bullet}(X)$ by the equivalence relation generated by \eqref{GysinAnyMap}. Let us prove that the projection $\pi \colon \hat{h}_{\bullet}(X) \to \hat{h}'_{\bullet}(X)$ is an isomorphism. In fact, the exact sequence \eqref{ExSeq1h} holds also replacing $\hat{h}_{k}(X)$ with $\hat{h}'_{k}(X)$, by defining $I[M, \hat{u}, [\hat{\alpha}, T], f, U] := [M, \hat{u}, \hat{\alpha}, f]$ and $a(T) := [0, T]$. Therefore, we have the following commutative diagram of exact sequences:
	\[\xymatrix{
	0 \ar[r] & \frac{\Tau_{k+1}(X; \h_{\R})}{\Tau_{k+1}^{\Int}(X; \h_{\R})} \ar[r]^(.52){a} \ar@{=}[d] & \hat{h}_{k}(X) \ar[r]^{I} \ar[d]^{\pi} & h_{k}(X) \ar[r] \ar@{=}[d] & 0 \\
	0 \ar[r] & \frac{\Tau_{k+1}(X; \h_{\R})}{\Tau_{k+1}^{\Int}(X; \h_{\R})} \ar[r]^(.52){a} & \hat{h}'_{k}(X) \ar[r]^{I} & h_{k}(X) \ar[r] & 0
}\]
It follows from the five lemma that $\pi$ is an isomorphism.
\end{proof}

We get diagram \eqref{CommDiagGysinPD} with currential cohomology, the proof of commutativity being the same. In this way, we get an explicit expression also for the inverse of Poincar\'e duality on an oriented compact manifold $(X, \hat{u})$:
	\[\hatPD(\check{\alpha}) = [X, \hat{u}, \check{\alpha}, \id_{X}, 0] \qquad\quad \hatPD^{-1}[M, \hat{v}, \check{\beta}, f, T] = f_{!}\check{\beta} + a(T)
\]

\SkipTocEntry \subsection{Homological Thom Morphism and Differentiation}

We now come back to orientation of a vector bundle in ordinary differential cohomology (without currents) and introduce an operation in homology which dualises fibrewise integration. Let $\pi \colon E \to X$ be a real vector bundle of rank $n$ with differential Thom class $\hat{u} \in \hat{h}^{n}_{\vcpt}(E)$. The Thom morphism
\begin{eqnarray}
	\Thom \colon \hat{h}^{\bullet}(X) & \hookrightarrow & \hat{h}^{\bullet+n}_{\vcpt}(E) \label{ThomCoh} \\
	\hat{\alpha} & \mapsto & \pi^{*}\hat{\alpha} \cdot \hat{u}, \nonumber
\end{eqnarray}
induces the integration map
\begin{eqnarray}
	\int_{E/X} \colon \; \hat{h}^{\bullet+n}_{\vcpt}(E) & \twoheadrightarrow & \hat{h}^{\bullet}(X) \label{IntMap} \\
	\Thom(\hat{\alpha}) + a(\omega) & \mapsto & \hat{\alpha} + a\Bigl( \Td(\hat{u})^{-1} \wedge \int_{E/X} \omega \Bigr). \nonumber
\end{eqnarray}
The reader can verify that this map is well-defined---that is, it does not depend on the decomposition $\Thom(\hat{\alpha}) + a(\omega)$---and two differential Thom classes are homotopic if and only if they induce the same integration map. Moreover, it immediately follows from the definition that the integration map is left-inverse to the Thom morphism. Dually, we define the \emph{homological Thom morphism}
\begin{eqnarray}
	\Thom' \colon \hat{h}_{\bullet+n}^{\vncpt}(E) & \twoheadrightarrow & \hat{h}_{\bullet}(X) \label{ThomHom} \\
	\hat{\lambda} & \mapsto & (-1)^{n(m-\bullet)} \, \pi_{*}(\hat{u} \pmb{\cdot} \hat{\lambda}), \nonumber
\end{eqnarray}
where $m := \dim(X)$ and $\hat{h}_{\bullet}^{\vncpt}(E)$ denotes the vertically non-compactly supported homology. We recall that the vertically compactly supported cohomology $\hat{h}^{\bullet}_{\vcpt}(E)$ is the colimit of the parallel groups $\hat{h}^{\bullet}(E, E \setminus K)$ over the directed set of vertically compact subsets $K \subset E$. Similarly, $\hat{h}_{\bullet}^{\vncpt}(E)$ is the limit of the parallel groups $\hat{h}_{\bullet}(E, E \setminus K)$; we constructed the parallel groups in the case of singular homology, and they can be defined similarly for a generalised cohomology theory. The group $\hat{h}_{\bullet}^{\vncpt}(E)$ can be represented through a model similar to the one we used for ordinary homology: we consider pre-cycles of the form $(M, \hat{u}, \hat{\alpha}, f)$, where $M$ is not necessarily compact, and $f \colon M \to X$ is vertically proper (that is, the inverse image of a vertically compact subset is compact). By considering proper maps (not only vertically), we obtain the non-compactly supported homology.

The Thom morphism induces the \emph{differentiation map}
\begin{eqnarray}\label{DiffMap}
	\partial_{E/X} \colon \hat{h}_{\bullet}(X) & \hookrightarrow & \hat{h}_{\bullet+n}^{\vncpt}(E) \\
	\hat{\lambda} & \mapsto & \partial_{E/X}\hat{\lambda} \nonumber	
\end{eqnarray}
where $\partial_{E/X}\hat{\lambda}$ is the unique class satisfying the two following conditions:
\begin{align}
	& \Thom'(\partial_{E/X}\hat{\lambda}) = \hat{\lambda} \label{Cond1Diff} \\
	& R(\partial_{E/X}\hat{\lambda}) = (-1)^{n(m-\bullet)} R(\hat{\lambda}) \circ \biggl(\Td(\hat{u})^{-1} \wedge \int_{E/X}\biggr) \label{Cond2Diff}
\end{align}
Formulas \eqref{Cond1Diff} and \eqref{Cond2Diff} are compatible with definition \eqref{ThomHom}, since, for every $\hat{\lambda} \in \hat{h}_{k}(X)$ and $\omega \in \Omega^{k+n}_{\vcpt}(E)$, we have:
\begin{align*}
	\bigl(R \circ \Thom'(\partial_{E/X}\hat{\lambda})\bigr)(\omega) & \overset{\eqref{ThomHom}}= (-1)^{n(m-k)} \pi_{*}\bigl(R(\hat{u}) \wedge R(\partial_{E/X}\hat{\lambda})\bigr)(\omega) \\
	& \hspace{5pt} = \hspace{5pt} (-1)^{n(m-k)} R(\partial_{E/X}\hat{\lambda})\bigl(\pi^{*}\omega \wedge R(\hat{u})\bigr) \\
	& \overset{\eqref{Cond2Diff}}= R(\hat{\lambda})\Bigl(\Td(\hat{u})^{-1} \wedge \int_{E/X} \bigl( \pi^{*}\omega \wedge R(\hat{u}) \bigr) \Bigr) \\
	& \hspace{5pt} = \hspace{5pt} \bigl(R(\hat{\lambda})\bigr)(\omega)
\end{align*}
Again, it immediately follows from the definition that the differentiation map is right-inverse to the homological Thom morphism.

When $X$ is oriented by $\hat{v}$, we obtain the fundamental class of the total space $[E, \hat{v} \times \hat{u}', 1, \id_{E}, 0]$, inducing Poincar\'e duality. The orientation $\hat{v} \times \hat{u}'$ can be constructed as follows. By fixing any connection on the vector bundle $E$, we get a splitting $TE \simeq \pi^{*}TX \oplus \pi^{*}E$. Since we are supposing $\hat{v}$ defined on a normal bundle of $X$, we call $\hat{v}'$ the complementary orientation on the tangent bundle. We obtain $\hat{w}' := \pi^{*}\hat{v}' \cdot \pi^{*}\hat{u}$ on $TE$, inducing $\hat{w}$ on a normal bundle of $E$. It follows that $\Td(\hat{w}) = \pi^{*}\Td(\hat{v}) \wedge \pi^{*}\Td(\hat{u})^{-1}$, hence we denote $\hat{w}$ by $\hat{v} \times \hat{u}'$. The proof of the following theorem is in appendix \ref{ProofThmDiff}.

\begin{Theorem}\label{ThmDiff} For every $\hat{\lambda} \in \hat{h}_{\bullet+N}^{\vncpt}(E)$ and $\hat{\alpha} \in \hat{h}^{\star}(X)$:
\begin{equation}\label{FundThThom}
	\textstyle \Thom(\hat{\alpha}) \pmb{\cdot} \hat{\lambda} = (-1)^{n(m-\bullet)}i_{*}(\hat{\alpha} \pmb{\cdot} \Thom'(\hat{\lambda}))
\end{equation}
where $m := \dim(X)$. Also, for every $\hat{\lambda} \in \hat{h}_{\bullet}(X)$ and $\hat{\alpha} \in \hat{h}^{\star+N}_{\vcpt}(E)$:
\begin{equation}\label{FundThVectBundle}
	\textstyle \bigl(\int_{E/X} \hat{\alpha}\bigr) \pmb{\cdot} \hat{\lambda} = (-1)^{n(m-\bullet)}\pi_{*}(\hat{\alpha} \pmb{\cdot} \partial_{E/X} \hat{\lambda})
\end{equation}
Moreover, by assuming $X$ compact and oriented, the following diagrams commute:
\begin{equation}\label{CommDiagThom}
	\resizebox{0.85\textwidth}{!}{\xymatrix{
	\hat{h}_{\bullet}^{\vncpt}(E) \ar[rr]^{\Thom'} & & \hat{h}_{\bullet-n}(X) \\
	\hat{h}^{m+n-\bullet}(E) \ar[rr]^{i^{*}} \ar[u]^{\hatPD} & & \hat{h}^{m+n-\bullet}(X) \ar[u]_{\hatPD}
	} \qquad\quad \xymatrix{
	\hat{h}^{\bullet}(X) \ar[rr]^{\Thom} \ar[d]_{\hatPD} & & \hat{h}^{\bullet+n}_{\vcpt}(E) \ar[d]^{\hatPD} \\
	\hat{h}_{m-\bullet}(X) \ar[rr]^{i_{*}} & & \hat{h}_{m-\bullet}^{\vncpt}(E)
}}\end{equation}
Also, the following diagrams commute:
\begin{equation}\label{CommDiagDiff}
	\resizebox{0.85\textwidth}{!}{\xymatrix{
	\hat{h}_{\bullet}(X) \ar[rr]^{\partial_{E/X}} & & \hat{h}_{\bullet+n}^{\vncpt}(E) \\
	\hat{h}^{m-\bullet}(X) \ar[rr]^{\pi^{*}} \ar[u]^{\hatPD} & & \hat{h}^{m-\bullet}(E) \ar[u]_{\hatPD}
	} \qquad\quad \xymatrix{
	\hat{h}^{\bullet}_{\vcpt}(E) \ar[rr]^{\int_{E/X}} \ar[d]_{\hatPD} & & \hat{h}^{\bullet-n}(X) \ar[d]^{\hatPD} \\
	\hat{h}_{m+n-\bullet}(E) \ar[rr]^{\pi_{*}} & & \hat{h}^{m+n-\bullet}(X)
}}\end{equation}
\end{Theorem}
In the previous statement and in theorem \ref{PropGysinHom} below, the compactness hypotheses can be removed by considering compactly-supported cohomology.

\SkipTocEntry \subsection{Homological Gysin Map}

Given a proper oriented submersion $\varphi \colon Y \to X$ and using the same notation of formula \eqref{GysinMap}, we define the Gysin map in homology $\varphi^{!} \colon \hat{h}_{\bullet}(X) \to \hat{h}_{\bullet+d}(Y)$ as follows:
\begin{equation}\label{GysinMapHom}
	\varphi^{!} := \Thom' \circ \phi^{*} \circ j^{*} \circ \partial_{X \times \R^{N}/X}
\end{equation}
This is dual to the cohomological definition, that is, $\varphi_{!} := \int_{X \times \R^{N}/X} \circ j_{*} \circ \phi_{*} \circ \Thom$.
\begin{Prop} The curvature map on currents
\begin{equation}\label{CurvMapHom}
	\varphi^{!}(T) := (-1)^{d(m-\bullet)} \, T \circ \biggl(\int_{Y/X} \Td(\hat{u}) \wedge \text{-} \biggr)
\end{equation}
satisfies the identities $R(\varphi^{!}(\hat{\lambda})) = \varphi^{!}(R(\hat{\lambda}))$ and $\varphi^{!}(a(T)) = a(\varphi^{!}(T))$.
\end{Prop}
\begin{proof} About curvature:
\begin{align*}
	R(\varphi^{!}(\hat{\lambda})) & \overset{\eqref{GysinMapHom}, \eqref{ThomHom}} = (-1)^{\abs{\mathcal{N}}(m-\bullet)} (\pi_{\mathcal{N}})_{*}\bigl( R(\hat{u}) \wedge R\bigl(\phi^{*} \circ j^{*} \circ \partial_{X \times \R^{N}/X}(\hat{\lambda})\bigr) \bigr) \\
	& \hspace{11.5pt} \overset{\eqref{Cond2Diff}}= (-1)^{\abs{\mathcal{N}}(m-\bullet) + N(m-\bullet)} (\pi_{\mathcal{N}})_{*}\biggl( R(\hat{u}) \wedge \phi^{*} \circ j^{*} \biggl( R(\hat{\lambda}) \circ \int_{X \times \R^{N}/X}\biggr) \biggr) \\
	& \hspace{15pt} = (-1)^{d(m-\bullet)} (\pi_{\mathcal{N}})_{*}\biggl( R(\hat{u}) \wedge \biggl( R(\hat{\lambda}) \circ \int_{X \times \R^{N}/X} \circ j_{*} \circ \phi_{*}\biggr)\biggr) \\
	& \hspace{15pt} = (-1)^{d(m-\bullet)} (\pi_{\mathcal{N}})_{*}\biggl( R(\hat{u}) \wedge \biggl( R(\hat{\lambda}) \circ \int_{Y/X} \int_{\mathcal{N}/Y} \biggr)\biggr) \\
	& \hspace{15pt} = (-1)^{d(m-\bullet)} R(\hat{\lambda}) \circ \biggl(\int_{Y/X} \int_{\mathcal{N}/Y} (\pi_{\mathcal{N}})^{*}(\,\text{-}\,) \wedge R(\hat{u}) \biggr) \\
	& \hspace{15pt} = (-1)^{d(m-\bullet)} R(\hat{\lambda}) \circ \biggl(\int_{Y/X} \Td(\hat{u}) \wedge \text{-} \biggr) \overset{\eqref{CurvMapHom}}= \varphi^{!}(R(\hat{\lambda}))
\end{align*}
The proof about $a$ is similar by starting from the identity
	\[\partial_{E/X} a(T) = (-1)^{n(m-\bullet)} \, a\biggl(T \circ \biggl(\Td(\hat{u})^{-1} \wedge \int_{E/X}\biggr)\biggr)
\]
that can be proven from \eqref{Cond1Diff} and \eqref{Cond2Diff} applied to $\hat{\lambda} = a(T)$.
\end{proof}

The proof of the following theorem is in appendix \ref{ProofPropGysinHom}. 

\begin{Theorem}\label{PropGysinHom} For every $\hat{\lambda} \in \hat{h}_{\bullet}(X)$ and $\hat{\alpha} \in \hat{h}^{\star}_{\vcpt}(Y)$:
\begin{equation}\label{FundThGysin}
	\textstyle (\varphi_{!} \hat{\alpha}) \pmb{\cdot} \hat{\lambda} = (-1)^{d(m-\bullet)}\varphi_{*}(\hat{\alpha} \pmb{\cdot} \varphi^{!}\hat{\lambda})
\end{equation}
where $d := \dim(Y) - \dim(X)$ and $m := \dim(X)$. Moreover, by assuming $X$ and $Y$ compact and oriented, the following diagrams commute:
\begin{equation}\label{CommDiagGysin}
	\resizebox{0.85\textwidth}{!}{\xymatrix{
	\hat{h}_{\bullet}(X) \ar[rr]^{\varphi^{!}} & & \hat{h}_{\bullet+d}(Y) \\
	\hat{h}^{m-\bullet}(X) \ar[rr]^{\varphi^{*}} \ar[u]^(.55){\hatPD} & & \hat{h}^{m-\bullet}(Y) \ar[u]_(.55){\hatPD}
	} \qquad\quad \xymatrix{
	\hat{h}^{\bullet}(Y) \ar[rr]^{\varphi_{!}} \ar[d]_{\hatPD} & & \hat{h}^{\bullet-d}(X) \ar[d]^{\hatPD} \\
	\hat{h}_{m+d-\bullet}(Y) \ar[rr]^{\varphi_{*}} & & \hat{h}_{m+d-\bullet}(X)
}}\end{equation}
\end{Theorem}

\begin{Corollary} Given a compact oriented manifold $(X, \hat{u})$, the corresponding map to the point $p \colon X \to \pt$ satisfies the identity
\begin{equation}\label{FundGysin}
	[X] = p^{!}\bigl(\hatPD(1)\bigr).
\end{equation}
\end{Corollary}
\begin{proof} From \eqref{CommDiagGysin}, we obtain $p^{!}\bigl(\hatPD(1)\bigr) = \hatPD(p^{*}1) = \hatPD(1) = [X, \hat{u}, 1, \id_{X}, 0] = [X]$.
\end{proof}

\begin{Corollary}\label{CorGysinTrivial} Given a trivial fibre bundle $\pi \colon X \times F \to X$, where $F$ is a compact oriented manifold, the following identity holds:
\begin{equation}\label{GysinTrivial}
	\pi^{!}(\hat{\lambda}) = \hat{\lambda} \times [F]
\end{equation}
In particular, if $F = S^{1}$, then we obtain $\partial_{S^{1}} \hat{\lambda} = \pi^{!}(\hat{\lambda})$.
\end{Corollary}
\begin{proof} From \eqref{CommDiagGysin}, we deduce:
\begin{align*}
	\pi^{!}\bigl(\hatPD(\hat{\alpha})\bigr) &= \hatPD(\pi^{*}_{X}\hat{\alpha}) = [X \times F, \pi_{X}^{*}\hat{u} \cdot \pi_{F}^{*}\hat{v}, \pi^{*}_{X}\hat{\alpha}, \id_{X \times F}, 0] \\
	&= [X, \hat{u}, \pi^{*}_{X}\hat{\alpha}, \id_{X}, 0] \hat{\times} [F, \hat{v}, 1, \id_{F}, 0] = \hatPD(\hat{\alpha}) \hat{\times} [F]. \qedhere
\end{align*}
\end{proof}

In the homological setting, we have not considered currential orientations. The reason is the following one. In chomology, they lead to a natural curvature map without assuming that the underlying function is a submersion, since we obtain the currential pushfoward twisted by the Todd class. In homology, in order to obtain a  natural pullback, we need to restrict curvature to forms. This is possible and leads to the natural curvature map $\varphi^{!}(\eta) = (-1)^{n(m-\bullet)} \Td(\hat{u}) \wedge \varphi^{*}\eta$, which is a particular case of \eqref{CurvMapHom} with $T = \ScT^{\bullet}\eta$. However, since restricted homology is less interesting than extended cohomology, we do not show the details here.

\SkipTocEntry \subsection{Uniqueness}

Let us prove the essential uniqueness of differential homology. We fix a multiplicative cohomology theory $h^{\bullet}$, its dual homology theory $h_{\bullet}$, and a multiplicative differential refinement $\hat{h}^{\bullet}$ with its natural transformations. A dual differential refinement is a functor $\hat{h}_{\bullet}$ endowed with its natural transformations and the cap product---that is, $(\hat{h}_{\bullet}, I, R, a, \iota, \,\pmb{\cdot}\,)$---in such a way that propositions \ref{PropExSeqH} and \ref{AxiomsDiffCapH} hold as axioms. We denote by $\hat{h}_{\bullet}$ the model we used up to now and we fix any other refinement $(\hat{h}'_{\bullet}, I', R', a', \iota', \,\pmb{\cdot}'\,)$. We have a natural isomorphism $\Phi \colon \hat{h}_{\bullet} \overset{\!\simeq}\longrightarrow \hat{h}'_{\bullet}$ defined as follows. We first observe that $\hat{h}_{\bullet}(\pt)$ is canonically determined by the topological theory because $h^{\bullet}$ is rationally even. In fact, sequences \eqref{ExSeq1h} and \eqref{ExSeq2h} imply that $I \colon \hat{h}_{2k}(\pt) \to h_{2k}(\pt)$ and $\iota \colon \hat{h}_{2k+1}^{\fl}(\pt) \to \hat{h}_{2k+1}(\pt)$ are isomorphisms. Thus, we have differential Poincar\'e duality on the point as a consequence of the topological one. Each oriented manifold has a fundamental class in $\hat{h}'_{\bullet}$ because of formula \eqref{FundGysin}, inducing Poincar\'e Duality. Therefore, we set:
\begin{eqnarray*}
	\Phi(X) \colon \; \hat{h}_{\bullet}(X) & \to & \hat{h}'_{\bullet}(X) \\
	\lbrack M, \hat{u}, \hat{\alpha}, f, T \rbrack & \mapsto & f_{*}\hatPD'_{(M, \hat{u})}(\hat{\alpha}) + a(T)
\end{eqnarray*}
The following diagram commutes:
	\[\xymatrix{
	0 \ar[r] & \frac{\Tau_{k+1}(X; \h_{\R})}{\Tau_{k+1}^{\Int}(X; \h_{\R})} \ar[r]^(.55){a} \ar@{=}[d] & \hat{h}_{k}(X) \ar[r]^{I} \ar[d]^(.43){\Phi(X)} & h_{k}(X) \ar[r] \ar@{=}[d] & 0 \\
	0 \ar[r] & \frac{\Tau_{k+1}(X; \h_{\R})}{\Tau_{k+1}^{\Int}(X; \h_{\R})} \ar[r]^(.55){a'} & \hat{h}'_{k}(X) \ar[r]^{I} & h_{k}(X) \ar[r] & 0
}\]
It follows from the five lemma that $\Phi$ is an isomorphism. As in the case of singular homology, the transformation $\iota$ is completely determined by the other ones, since diagram \eqref{DiagPD2H} implies that $\iota(\hat{\alpha}_{\fl}) = \hatPD \circ \PD^{-1}(\hat{\alpha}_{\fl})$.

\section{Further Perspectives}\label{FurthPer}

We briefly outline some issues that we will develop in future papers. Some of them have already been mentioned in the previous sections.

\SkipTocEntry \subsection*{Possible Applications in Mathematical Physics}

Let us consider a charge source, a field strength with potential, and the corresponding minimal coupling. This is a ty\-pical situation in electromagnetism. Also, in string theory, the source can be a D-brane world-volume, and the minimal coupling with the Ramond-Ramond potentials is called Wess-Zumino action. In such a situation, the source is usually described mathematically as a fixed cycle in a suitable homology theory. The underlying homology class is the corresponding charge, but the representative itself is meaningful, since it indicates the trajectory in space-time (or the position in space, depending on the context). The field potential is a differential cohomology class, and the minimal coupling is the corresponding holonomy map on the world-volume. The chosen (co)homology theory depends on the context, the main examples being singular (co)homology and K-theory. The disadvantage of this picture consists in the fact that a fixed cycle carries an excess of information with respect to the physical meaning: considering for example singular homology, a representative of the fundamental class of a submanifold essentially requires to fix a triangulation, which is physically meaningless. This is the reason why differential homology may be useful, since it provides a description of the charge source without fixing any representative. Indeed, we saw that the natural transformation $b$ identifies the set of embedded oriented submanifolds of $X$ with a subset of $\hat{H}_{\bullet}(X)$, and the corresponding holonomy of a Cheeger-Simons character is well-defined. By extending the transformation $b$ to a gene\-ralised homology theory (in particular, to K-theory), one can describe any charge source as a differential homology class, which naturally couples to the cohomology class represen\-ting the field. This coupling is realised through the notion of \emph{generalised Cheeger-Simons characters}, that was defined in \cite{FR} on cycles and can be projected to homology.

\SkipTocEntry \subsection*{Homotopy \emph{vs} Homology}

In the topological framework, the well-know Hurewicz map $w_{k} \colon \pi_{k}(X, x_{0}) \to H_{k}(X)$ relates homotopy and homology. One can define \emph{thin} homotopy groups of a smooth manifold, which we denote by $\pi_{k}^{\textnormal{th}}(X, x_{0})$, by requiring that the differential of a homotopy between two representatives does not reach the maximal rank in any point. In this way, the holonomy of a differential character projects to a thin homotopy class. Since a thin homotopy is in particular a homotopy, we have the natural projection $P_{k} \colon \pi_{k}^{\textnormal{th}}(X, x_{0}) \to \pi_{k}(X, x_{0})$. One can define the \emph{differential Hurewicz Map} $\hat{w}_{k} \colon \pi_{k}^{\textnormal{th}}(X, x_{0}) \to \hat{H}_{k}(X)$ as follows: we fix a cycle $u_{k} \in Z_{k}(S^{k})$ representing a generator in homology, and we set $\hat{w}_{k}[f] := [f_{*}u_{k}, 0]$ for any $f \colon S^{k} \to X$. Since two different choices of $u_{k}$ differ by the boundary of a thin chain, the class $[f_{*}u_{k}, 0]$ is well-defined, and it is straightforward to verify that the following diagram commutes:
	\[\xymatrix{
	\pi_{k}^{\textnormal{th}}(X, x_{0}) \ar[rr]^{\hat{w}_{k}} \ar[d]_{P_{k}} & & \hat{H}_{k}(X) \ar[d]^{I_{k}} \\
	\pi_{k}(X, x_{0}) \ar[rr] \ar[rr]^{w_{k}} & & H_{k}(X) \\
}\]
We think it would be interesting to analyse the differential Hurewicz map and its possible applications in detail.

\SkipTocEntry \subsection*{Relative, Non-Compact, and Twisted Differential Homology}

In the case of singular homology, we constructed in detail the relative, non-compact, and local coefficient versions, together with the corresponding dualities. The same can be done about a generalised theory. Actually, we already sketched how to defined the (vertically-)non-compact version through (vertically) proper maps, but we have to work out the details. The relative version can be performed by following the same line of the singular case, adapted to the model we are using, and the corresponding axiomatic framework can be worked out as in \cite{FR2}. About local coefficients, the generalisation is a deeper issue, since we have to extend the present work to twisted (co)homology theories.

\SkipTocEntry \subsection*{Hopkins-Singer Model}

The most general model of differential cohomology was introduced by Hopkins and Singer in \cite{HS} and elaborated in detail by Upmeier in \cite{Upmeier}. It consists in a refinement of the description of cohomology through spectra. In particular, fixing a spectrum based on the spaces $\{E_{k}\}_{k \in \Z}$ with fundamental cocycles $\theta^{k} \in H^{k}(E_{k}; \h_{\R})$, a class $\hat{\alpha} \in \hat{h}^{k}(X)$ is represented by a triple $(f, h, \omega)$, where $f \colon X \to E_{k}$ represents $I(\hat{\alpha})$, the form $\omega$ is the curvature, and $h$ is a $\h_{\R}$-valued $(k-1)$-cochain that trivialises $\omega - f^{*}\theta^{k}$. The equivalence relation is a suitable notion of homotopy that fixes the curvature. We believe that a similar construction can be realised in homology. In fact, starting from $h_{k}(X) = \varinjlim \pi_{k+q} \bigl(X_{+} \wedge E_{q}\bigr)$, one can consider pairs of the form $(f, T)$, where $f \colon S^{k+q} \to X_{+} \wedge E_{q}$ represents $I(\hat{\lambda})$, and $T$ is a current. Then, up to a suitable notion of homotopy compatible with the direct limit and fixing the curvature, we set $I[f, T] := [f]$, $R[f, T] := T_{(f_{*}u_{k+q})/\theta^{q}} + \partial T$, and $a(T) := [0, T]$ where $0$ is the constant function to the marked point.

\SkipTocEntry \subsection*{Other Topics on Singular Homology}

We showed an explicit geometric description of the Poincar\'e dual of a low-degree cohomology class, and we believe it is possible to generalise such a description to any degree through the language of abelian $k$-gerbes with connection.

Furthermore, differential singular homology can be extended to differential spaces (in particular, stratifolds) as in \cite{BB}. In this way, we deal with a much wider class of spaces, and the hypotheses on excision and the Mayer-Vietoris sequence become weaker. Moreover, the cap product can be described from a geometric point of view, also deducing in a more natural way its uniqueness from the axioms.

Lastly, the Deligne cohomology is a well-known model of differential singular cohomology, consisting in the cohomology of a suitable complex of sheaves. We believe it is possible to dualise this construction by defining \emph{Deligne homology} as the homology of a suitable complex of cosheaves.

\appendix

\section{Computations about the Cap Product}\label{AppCapProduct}

We show the computations about the differential cap product not included in the main text.

\SkipTocEntry \subsection{Chain Maps}\label{SecChMaps}

The two maps $C_{\bullet}(X; \R) \otimes_{\R} \Omega_{\star}(X) \to \Tau_{\bullet+\star}(X)$ defined respectively by $\gamma_{k} \otimes \omega_{h} \mapsto T_{\gamma_{k}} \wedge \omega_{h}$ and $\gamma_{k} \otimes \omega_{h} \mapsto T_{\gamma_{k} \cap \omega_{h}}$ are chain maps. In fact, with respect to the former:
\begin{align*}
	\partial(\gamma_{k} \otimes \omega_{h}) & \overset{\eqref{BoundaryTensor}}= (-1)^{h}(\partial \gamma_{k} \otimes \omega_{h} - \gamma_{k} \otimes d\omega_{h}) \\
	& \, \mapsto (-1)^{h}(\partial T_{\gamma_{k}} \wedge \omega_{h} - T_{\gamma_{k}} \wedge d\omega_{h}) \overset{\eqref{BoudaryWedge}}= \partial(T_{\gamma_{k}} \wedge \omega_{h})
\end{align*}
Similarly, with respect to the latter:
\begin{align*}
	\partial(\gamma_{k} \otimes \omega_{h}) & \overset{\eqref{BoundaryTensor}}= (-1)^{h}(\partial \gamma_{k} \otimes \omega_{h} - \gamma_{k} \otimes d\omega_{h}) \\
	& \,\mapsto (-1)^{h}(T_{\partial \gamma_{k} \cap \omega_{h} - \gamma_{k} \cap d\omega_{h}}) \overset{\eqref{BoundaryCap}}= \partial(T_{\gamma_{k} \cap \omega_{h}})
\end{align*}

\SkipTocEntry \subsection{Well-Definedness of the Cap Product}\label{AppProdWD}

If one of the two classes $[\gamma_{l}, T_{l+1}]$ and $[\mu^{s}, c^{s-1}, \omega^{s}]$ vanishes, then their cap product vanishes as well. Indeed, in the former case:
\begin{align*}
	[\partial\Gamma_{l+1}, & - T_{\Gamma_{l+1}} + \partial T_{\Gamma_{l+2}}] \hatcap [\mu^{s}, c^{s-1}, \omega^{s}] \\
	& \overset{\eqref{DiffCapP}}= \bigl[\partial\Gamma_{l+1} \cap \mu^{s}, (-1)^{s}\bigl((-T_{\Gamma_{l+1}} + \partial T_{l+2}) \wedge \omega^{s} + T_{\partial\Gamma_{l+1} \cap c^{s-1}}\bigr) + B(\partial\Gamma_{l+1} \otimes \omega^{s})\bigr] \\
	& \hspace{-8pt} \overset{\eqref{BoundaryCap}, \eqref{BoundaryTensor}}= \bigl[(-1)^{s}\partial(\Gamma_{l+1} \cap \mu^{s}), (-1)^{s}\bigl(({\color{red} -T_{\Gamma_{l+1}}} + \partial T_{l+2}) {\color{red} \wedge \omega^{s}} \\
	& \phantom{XXXXXXXXXXXXX}\, + T_{(-1)^{s-1}\partial(\Gamma_{l+1} \cap c^{s-1}) + \Gamma_{l+1} \cap \delta c^{s-1}} + {\color{red} B \partial(\Gamma_{l+1} \otimes \omega^{s})}\bigr) \bigr] \\
	& \overset{\eqref{HomotopyCapWedge}}= \bigl[(-1)^{s}\partial(\Gamma_{l+1} \cap \mu^{s}), (-1)^{s}\bigl({\color{red} - T_{\Gamma_{l+1} \cap \omega^{s}}} + \partial T_{l+2} \wedge \omega^{s} \\
	& \phantom{XXXXXXXXXXXXX}\, + T_{(-1)^{s-1}\partial(\Gamma_{l+1} \cap c^{s-1}) + \Gamma_{l+1} \cap \delta c^{s-1}} - {\color{red} \partial B(\Gamma_{l+1} \otimes \omega^{s})}\bigr) \bigr] \\
	& \hspace{-5pt} \overset{\eqref{BoudaryWedge}, \eqref{BoudaryCurrents}}= \bigl[(-1)^{s}\partial(\Gamma_{l+1} \cap \mu^{s}), (-1)^{s}(- T_{\Gamma_{l+1} \cap \omega^{s}} + T_{\Gamma_{l+1} \cap \delta c^{s-1}}) + \partial(\,\cdots) \bigr] \\
	& \overset{\eqref{FormC}}= \bigl[(-1)^{s}\partial(\Gamma_{l+1} \cap \mu^{s}), (-1)^{s}(-T_{\Gamma_{l+1} \cap \mu^{s}}) + \partial(\,\cdots)\bigr] = 0
\end{align*}
In the latter case:
\begin{align*}
	[\gamma_{l}, &T_{l+1}] \hatcap [\delta \nu^{s-1}, -\nu^{s-1} - \delta d^{s-2}, 0] \\
	& \overset{\eqref{DiffCapP}}= \bigl[\gamma_{l} \cap \delta \nu^{s-1}, (-1)^{s}T_{\gamma_{l} \cap (-\nu^{s-1}-\delta d^{s-2})}\bigr] \\
	& \overset{\eqref{BoundaryCap}}= \bigl[(-1)^{s}\partial(\gamma_{l} \cap \nu^{s-1}), (-1)^{s}(-T_{\gamma_{l} \cap \nu^{s-1}}) + \partial(\,\cdots)\bigr] = 0
\end{align*}

\SkipTocEntry \subsection{Proof of Proposition \ref{AxiomsDiffCap}}\label{AppProof}

(1) It is obvious from definition \eqref{DiffCapP}, since the expression is bi-additive on representatives.

\vspace{3pt} (2) Again, the compatibility with $I$ immediately follows from definition \eqref{DiffCapP}. About $R$:
\begin{align*}
	R([\gamma_{l}, T_{l+1}&] \hatcap [\mu^{s}, c^{s-1}, \omega^{s}]) \\
	& \hspace{-11pt} \overset{\eqref{DiffCapP}, \eqref{DefR}}= T_{\gamma_{l} \cap \mu^{s}} + (-1)^{s}\partial (T_{l+1} \wedge \omega^{s} + T_{\gamma_{l} \cap c^{s-1}}) + \partial B(\gamma_{l} \otimes \omega^{s}) \\
	& \hspace{-9pt} \overset{\eqref{BoudaryWedge}, \eqref{BoundaryCap}}= T_{\gamma_{l} \cap \mu^{s}} + \partial T_{l+1} \wedge \omega^{s} + T_{\gamma_{l} \cap \delta c^{s-1}} + \partial B(\gamma_{l} \otimes \omega^{s}) \\
	& \hspace{-11pt} \overset{\eqref{HomotopyCapWedge}, \eqref{FormC}}= {\color{blue} T_{\gamma_{l} \cap \mu^{s}}} + \partial T_{l+1} \wedge \omega^{s} + {\color{blue} T_{\gamma_{l} \cap (\omega^{s} - \mu^{s})}} + T_{\gamma_{l}} \wedge \omega^{s} - {\color{blue} T_{\gamma_{l} \cap \omega^{s}}} \\
	& \hspace{1pt} = (\partial T_{l+1} + T_{\gamma_{l}}) \wedge \omega^{s} \overset{\eqref{DefR}}= R([\gamma_{l}, T_{l+1}]) \wedge R([\mu^{s}, c^{s-1}, \omega^{s}])
\end{align*}
About $a$:
	\[a(T) \hatcap [\mu^{s}, c^{s-1}, \omega^{s}] \overset{\eqref{DefA}, \eqref{DiffCapP}}= \bigl[0, (-1)^{s}T_{l+1} \wedge \omega^{s}\bigr] \overset{\eqref{DefA}}= (-1)^{s}a(T_{l+1} \wedge \omega^{s})
\]
and
\begin{align*}
	[\gamma_{l}, &T_{l+1}] \hatcap a(\omega^{s-1}) = [\gamma_{l}, T_{l+1}] \hatcap [0, \omega^{s-1}, d\omega^{s-1}] \\
	& \hspace{-3pt} \overset{\eqref{DiffCapP}}= \bigl[0, (-1)^{s} ({\color{red} T_{l+1} \wedge d\omega^{s-1}} + T_{\gamma_{l} \cap \omega^{s-1}}) + B(\gamma_{l} \otimes d\omega^{s-1})\bigr] \\
	& \hspace{-1pt} \overset{\eqref{BoudaryWedge}}= \bigl[0, {\color{red} \partial(\,\cdots) + (-1)^{s}\partial T_{l+1} \wedge \omega^{s-1}} + (-1)^{s}T_{\gamma_{l} \cap \omega^{s-1}} + (-1)^{s}B(\gamma_{l} \otimes \partial\omega_{1-s})\bigr] \\
	& \hspace{-3pt} \overset{\eqref{BoundaryTensor}}= (-1)^{s} \bigl[0, \partial T_{l+1} \wedge \omega^{s-1} + T_{\gamma_{l} \cap \omega^{s-1}} + B\partial(\gamma_{l} \otimes \omega_{1-s})\bigr] \\
	& \hspace{-3pt} \overset{\eqref{HomotopyCapWedge}}= (-1)^{s} \bigl[0, \partial T_{l+1} \wedge \omega^{s-1} + {\color{green} T_{\gamma_{l} \cap \omega^{s-1}}} + T_{\gamma_{l}} \wedge \omega^{s-1} - {\color{green} T_{\gamma_{l} \cap \omega^{s-1}}} + \partial(\,\cdots) \bigr] \\
	& = (-1)^{s} \bigl[0, \partial T_{l+1} \wedge \omega^{s-1} + T_{\gamma_{l}} \wedge \omega^{s-1}\bigr] = (-1)^{s}a\bigl(R[\gamma_{l}, T_{l+1}] \wedge \omega^{s-1}\bigr)
\end{align*}
About $\iota$:
\begin{align*}
	\iota[\Gamma_{l+1}&] \hatcap [\mu^{s}, c^{s-1}, \omega^{s}] \overset{\eqref{MorfIota}}= \bigl[-\partial \tilde{\Gamma}_{l+1}, T_{\tilde{\Gamma}_{l+1}} \bigr] \hatcap [\mu^{s}, c^{s-1}, \omega^{s}] \\
	& \overset{\eqref{DiffCapP}}=  \bigl[-{\color{red} \partial \tilde{\Gamma}_{l+1} \cap \mu^{s}}, (-1)^{s}(T_{\tilde{\Gamma}_{l+1}} \wedge \omega^{s} - T_{{\color{magenta} \partial\tilde{\Gamma}_{l+1} \cap c^{s-1}}}) - {\color{blue} B(\partial\tilde{\Gamma}_{l+1} \otimes \omega^{s})}\bigr] \\
	& \hspace{-9pt} \overset{\eqref{BoundaryCap}, \eqref{BoundaryTensor}}= \bigl[-{\color{red}(-1)^{s} \partial (\tilde{\Gamma}_{l+1} \cap \mu^{s})}, (-1)^{s}T_{\tilde{\Gamma}_{l+1}} \wedge \omega^{s} - T_{{\color{magenta} \partial(\tilde{\Gamma}_{l+1} \cap c^{s-1}) + (-1)^{s}\tilde{\Gamma}_{l+1} \cap (\omega^{s}-\mu^{s})}} \\
	& \hspace{267pt} - {\color{blue} (-1)^{s}B\partial(\tilde{\Gamma}_{l+1} \otimes \omega^{s})} \bigr] \\
	& \overset{\eqref{HomotopyCapWedge}}=  \bigl[-(-1)^{s} \partial (\tilde{\Gamma}_{l+1} \cap \mu^{s}), {\color{green} (-1)^{s}T_{\tilde{\Gamma}_{l+1}} \wedge \omega^{s}} - T_{{\color{orange} \partial(\tilde{\Gamma}_{l+1} \cap c^{s-1})} + {\color{cyan} (-1)^{s}\tilde{\Gamma}_{l+1} \cap } ({\color{cyan} \omega^{s}}-\mu^{s})} \\
	& \hspace{160pt} - {\color{green} (-1)^{s} T_{\tilde{\Gamma}_{l+1}} \wedge \omega^{s}} + {\color{cyan} (-1)^{s} T_{\tilde{\Gamma}_{l+1} \cap \omega^{s}}} + {\color{orange} \partial(\,\cdots)} \bigr] \\
	& \hspace{3pt} = (-1)^{s} \bigl[-\partial (\tilde{\Gamma}_{l+1} \cap \mu^{s}), T_{\tilde{\Gamma}_{l+1} \cap \mu^{s}}] \\
	& \overset{\eqref{MorfIota}}= (-1)^{s} \iota[\Gamma_{l+1} \cap \mu^{s}] = (-1)^{s} \iota\bigl([\Gamma_{l+1}] \cap I[\mu^{s}, c^{s-1}, \omega^{s}]\bigr)
\end{align*}
and
\begin{align*}
	[\gamma_{l}, T_{l+1}] \hatcap \iota[\varphi^{s-1}] &= [\gamma_{l}, T_{l+1}] \hatcap [-\delta \tilde{\varphi}^{s-1}, \tilde{\varphi}^{s-1}, 0] \\
	& \hspace{-3pt} \overset{\eqref{DiffCapP}}= \bigl[-\gamma_{l} \cap \delta \tilde{\varphi}^{s-1}, (-1)^{s} T_{\gamma_{l} \cap \tilde{\varphi}^{s-1}}\bigr] \\
	& \hspace{-3pt} \overset{\eqref{BoundaryCap}}= (-1)^{s} \bigl[-\partial(\gamma_{l} \cap \tilde{\varphi}^{s-1}), T_{\gamma_{l} \cap \tilde{\varphi}^{s-1}}\bigr] \\
	& = (-1)^{s} \iota[\gamma_{l} \cap \varphi^{s-1}] = (-1)^{s}\iota\bigl(I[\gamma_{l}, T_{l+1}] \cap [\varphi^{s-1}]\bigr)
\end{align*}

\vspace{3pt} (3) The two chain maps $\Omega^{\bullet}(X) \otimes_{\R} \Omega^{\star}(X) \to C^{\bullet+\star}(X; \R)$ defined respectively by $\omega^{s} \otimes \eta^{t} \mapsto \omega^{s} \wedge \eta^{t}$ and $\omega^{s} \otimes \eta^{t} \mapsto \omega^{s} \cup \eta^{t}$ are chain homotopic in an essentially unique way; see \cite[formula (3.8)]{HS}. We thus fix any $B' \colon \Omega^{\bullet}(X) \otimes_{\R} \Omega^{\star}(X) \to C^{\bullet+\star-1}(X; \R)$ such that
\begin{equation}\label{DefBLinha}
	\omega \wedge \eta - \omega \cup \eta = (B'\delta + \delta B')(\omega \otimes \eta),
\end{equation}
where
\begin{equation}\label{DefDeltaTensor}
	\delta(\omega^{s} \otimes \eta^{t}) = d\omega^{s} \otimes \eta^{t} + (-1)^{s} \omega^{s} \otimes d\eta^{t}.
\end{equation}
The product in differential cohomology is defined as follows:
\begin{equation}\label{ProdDiffCoh}
	[\mu^{s}, c^{s-1}, \omega^{s}] \hatcup [\nu^{t}, d^{t-1}, \eta^{t}] := [\mu^{s} \cup \nu^{t}, (-1)^{s}\mu^{s} \cup d^{t-1} + c^{s-1} \cup \eta^{t} + B'(\omega^{s} \otimes \eta^{t}), \omega^{s} \wedge \eta^{t}]
\end{equation}
Therefore,
\begin{align*}
	[\gamma_{l}, &T_{l+1}] \hatcap \bigl([\mu^{s}, c^{s-1}, \omega^{s}] \hatcup [\nu^{t}, d^{t-1}, \eta^{t}]\bigr) \overset{\eqref{ProdDiffCoh}, \eqref{DiffCapP}}= \bigl[{\color{blue} \gamma_{l} \cap (\mu^{s} \cup \nu^{t})}, \\
	& \hspace{10pt} (-1)^{s+t}\bigl({\color{red} T_{l+1} \wedge (\omega^{s} \wedge \eta^{t})} + T_{{\color{green} \gamma_{l} \cap } ((-1)^{s} {\color{green} \mu^{s} \cup d^{t-1}} + c^{s-1} \cup \eta^{t} + B'(\omega^{s} \otimes \eta^{t}))}\bigr) + B(\gamma_{l} \otimes (\omega^{s} \wedge \omega^{t}))\bigr]
\end{align*}
and
\begin{align*}
	\bigl([\gamma_{l}, T_{l+1}] \hatcap [&\mu^{s}, c^{s-1}, \omega^{s}]\bigr) \hatcap [\nu^{t}, d^{t-1}, \eta^{t}]\bigr) \\
	& \hspace{-5pt} \overset{\eqref{DiffCapP}}= \bigl[\gamma_{l} \cap \mu^{s}, (-1)^{s}(T_{l+1} \wedge \omega^{s} + T_{\gamma_{l} \cap c^{s-1}}) + B(\gamma_{l} \otimes \omega^{s})\bigr] \hatcap [\nu^{t}, d^{t-1}, \eta^{t}] \\
	& \hspace{-5pt} \overset{\eqref{DiffCapP}}= \bigl[{\color{blue} (\gamma_{l} \cap \mu^{s}) \cap \nu^{t}}, (-1)^{t}\bigl(\bigl((-1)^{s}({\color{red} T_{l+1} \wedge \omega^{s}} + T_{\gamma_{l} \cap c^{s-1}}) \\
	& \hspace{95pt} + B(\gamma_{l} \otimes \omega^{s})\bigr) {\color{red} \wedge \eta^{t} } + {\color{green} T_{(\gamma_{l} \cap \mu^{s}) \cap d^{t-1}}}\bigr) + B((\gamma_{l} \cap \mu^{s}) \otimes \eta^{t})\bigr].
\end{align*}
We need to compare the following two expressions:
\begin{align}
	& {\color{magenta} (-1)^{s+t} } T_{{\color{magenta} \gamma_{l} \cap } ({\color{magenta} c^{s-1} \cup \eta^{t} } + B'(\omega^{s} \otimes \eta^{t}))} + B(\gamma_{l} \otimes (\omega^{s} \wedge \omega^{t})) \label{FirstTermToC} \\
	& {\color{teal} (-1)^{t} } \bigl( {\color{teal} (-1)^{s} T_{\gamma_{l} \cap c^{s-1}} } + B(\gamma_{l} \otimes \omega^{s})\bigr) \wedge \eta^{t} + B((\gamma_{l} \cap \mu^{s}) \otimes \eta^{t}) \label{SecondTermToC}
\end{align}
The first terms on the left are respectively $\color{magenta} (-1)^{s+t} T_{\gamma_{l} \cap (c^{s-1} \cup \eta^{t})}$ and $\color{teal} (-1)^{s+t} T_{\gamma_{l} \cap c^{s-1}} \wedge \eta^{t}$. The difference between them can be computed as follows:
\begin{align*}
	(-1)^{s+t} ({\color{teal} T_{\gamma_{l} \cap c^{s-1}} } & {\color{teal} \wedge \eta^{t} } - {\color{magenta} T_{(\gamma_{l} \cap c^{s-1}) \cap \eta^{t}} }) \\
	& \overset{\eqref{HomotopyCapWedge}}= (-1)^{s+t} B\partial(\gamma_{l} \cap c^{s-1} \otimes \eta^{t}) + \partial(\,\cdots) \\
	& \overset{\eqref{BoundaryTensor}}= (-1)^{s}B(\partial(\gamma_{l} \cap c^{s-1}) \otimes \eta^{t}) + \partial(\,\cdots) \\
	& \overset{\eqref{BoundaryCap}}= B(\gamma_{l} \cap \delta c^{s-1} \otimes \eta^{t}) + \partial(\,\cdots) \\
	& \overset{\eqref{FormC}}= B(\gamma_{l} \cap (\omega^{s} - \mu^{s}) \otimes \eta^{t}) + \partial(\,\cdots)
\end{align*}
Therefore, by subtracting the term $\color{magenta} (-1)^{s+t} T_{\gamma_{l} \cap (c^{s-1} \cup \eta^{t})}$ from \eqref{FirstTermToC} and \eqref{SecondTermToC}, we obtain respectively (up to boundaries):
\begin{align}
	& (-1)^{s+t}T_{\gamma_{l} \cap B'(\omega^{s} \otimes \eta^{t})} + B(\gamma_{l} \otimes (\omega^{s} \wedge \omega^{t})) \label{FirstTermToC2} \\
	& B(\gamma_{l} \cap \omega^{s} \otimes \eta^{t}) + (-1)^{t}B(\gamma_{l} \otimes \omega^{s}) \wedge \eta^{t} \label{SecondTermToC2}
\end{align}
In order to prove that these expressions coincide up to an exact current, we show that they are homotopies between the chain maps $C_{\bullet}(X; \R) \otimes \Omega_{\star}(X) \otimes \Omega_{\#}(X) \to \Tau_{\bullet+\star+\#}(X)$ defined by $\gamma \otimes \omega \otimes \eta \mapsto T_{\gamma} \wedge \omega \wedge \eta$ and $\gamma \otimes \omega \otimes \eta \mapsto T_{(\gamma \cap \omega) \cap \eta}$. The essential uniqueness then implies the thesis.

We first compute the boundary of the complex $C_{\bullet}(X; \R) \otimes \Omega_{\star}(X) \otimes \Omega_{\#}(X)$:
\begin{align}
	\partial(\gamma_{l} \otimes \omega^{s} \otimes \eta^{t}) & \overset{\eqref{BoundaryTensor}}= (-1)^{s+t} \partial \gamma_{l} \otimes \omega^{s} \otimes \eta^{t} - (-1)^{s+t} \gamma_{l} \otimes \delta(\omega^{s} \otimes \eta^{t}) \nonumber \\
	& \hspace{-2pt} \overset{\eqref{DefDeltaTensor}}= (-1)^{s+t} \partial \gamma_{l} \otimes \omega^{s} \otimes \eta^{t} - (-1)^{s+t} \gamma_{l} \otimes (d\omega^{s} \otimes \eta^{t} + (-1)^{s}\omega^{s} \otimes d\eta^{t}) \nonumber \\
	& \hspace{3pt}= (-1)^{s+t} (\partial \gamma_{l} \otimes \omega^{s} \otimes \eta^{t} - \gamma_{l} \otimes d \omega^{s} \otimes \eta^{t}) - (-1)^{t} \gamma_{l} \otimes \omega^{s} \otimes d\eta^{t} \label{TripleBoundary}
\end{align}

\emph{First case.} We have:
\begin{align}
	T_{\gamma_{l}} \wedge \omega^{s} \wedge \eta^{t} - T_{(\gamma_{l} \cap \omega^{s}) \cap \eta^{t}} & = (T_{\gamma_{l}} \wedge \omega^{s} \wedge \eta^{t} - {\color{cyan} T_{\gamma_{l} \cap (\omega^{s} \wedge \eta^{t})} }) + T_{{\color{cyan} \gamma_{l} \cap (\omega^{s} \wedge \eta^{t})} - \gamma_{l} \cap (\omega^{s} \cup \eta^{t})} \nonumber \\
	& \hspace{-13pt} \overset{\eqref{HomotopyCapWedge}, \eqref{DefBLinha}}= (B\partial + \partial B)(\gamma_{l} \otimes (\omega^{s} \wedge \eta^{t})) + T_{\gamma_{l} \cap (B'\delta + \delta B')(\omega^{s} \otimes \eta^{t})} \label{FirstBLL}
\end{align}
By calling $B''(\gamma_{l} \otimes \omega^{s} \otimes \eta^{t})$ the expression \eqref{FirstTermToC2}, we obtain:\footnote{In the first equality below, the sign $(-1)^{s+t}$ in \eqref{FirstTermToC2} remains unchanged when applied to $\partial_{\gamma_{l}} \otimes \omega^{s} \otimes \eta^{t}$, but becomes $(-1)^{(s+1)+t}$ and $(-1)^{s+(t+1)}$ when applied respectively to $\gamma_{l} \otimes d\omega^{s} \otimes \eta^{t}$ and $\gamma_{l} \otimes \omega^{s} \otimes d\eta^{t}$.}
\begin{align*}
	B'' & \partial(\gamma_{l} \otimes \omega^{s} \otimes \eta^{t}) \\
	& \overset{\eqref{TripleBoundary}}= T_{\partial \gamma_{l} \cap B'(\omega^{s} \otimes \eta^{t}) + \gamma_{l} \cap B'(d \omega^{s} \otimes \eta^{t}) + (-1)^{s} \gamma_{l} \cap B'(\omega^{s} \otimes d\eta^{t})} \\
	& \hspace{22pt} + B\bigl((-1)^{s+t} \bigl(\partial \gamma_{l} \otimes (\omega^{s} \wedge \eta^{t}) - \gamma_{l} \otimes (d \omega^{s} \wedge \eta^{t})\bigr) - (-1)^{t} \gamma_{l} \otimes (\omega^{s} \wedge d\eta^{t})\bigr) \\
	& \hspace{-8pt} \overset{\eqref{DefDeltaTensor}, \eqref{BoundaryTensor}}= T_{\partial \gamma_{l} \cap B'(\omega^{s} \otimes \eta^{t}) + \gamma_{l} \cap B' \delta(\omega^{s} \otimes \eta^{t})} + B\partial(\gamma_{l} \otimes (\omega^{s} \wedge \eta^{t}))
\end{align*}
Moreover:
\begin{align*}
	\partial B''(\gamma_{l} \otimes \omega^{s} \otimes \eta^{t}) & \overset{\eqref{FirstTermToC2}}= (-1)^{s+t} T_{\partial(\gamma_{l} \cap B'(\omega^{s} \otimes \eta^{t}))} + \partial B(\gamma_{l} \otimes (\omega^{s} \wedge \eta^{t})) \\
	& \hspace{2pt} \overset{\eqref{BoundaryCap}}= - T_{\partial\gamma_{l} \cap B'(\omega^{s} \otimes \eta^{t}) - \gamma_{l} \cap \delta B'(\omega^{s} \otimes \eta^{t})} + \partial B(\gamma_{l} \otimes (\omega^{s} \wedge \eta^{t})) 
\end{align*}
Thus, $(B'' \partial + \partial B'')(\gamma_{l} \otimes \omega^{s} \otimes \eta^{t})$ coincides with \eqref{FirstBLL}.

\vspace{3pt} \emph{Second case.} We have:
\begin{align}
	T_{\gamma_{l}} \wedge \omega^{s} \wedge \eta^{t} - T_{(\gamma_{l} \cap \omega^{s}) \cap \eta^{t}} &= (T_{\gamma_{l}} \wedge \omega^{s} \wedge \eta^{t} - {\color{cyan} T_{\gamma_{l} \cap \omega^{s}} \wedge \eta^{t} }) + ({\color{cyan} T_{\gamma_{l} \cap \omega^{s}} \wedge \eta^{t} } - T_{(\gamma_{l} \cap \omega^{s}) \cap \eta^{t}}) \nonumber \\
	&= (B\partial + \partial B)(\gamma_{l} \otimes \omega^{s}) \wedge \eta^{t} + (B\partial + \partial B)(\gamma_{l} \cap \omega^{s} \otimes \eta^{t}) \label{SecondBLL}
\end{align}
By calling $B'''(\gamma_{l} \otimes \omega^{s} \otimes \eta^{t})$ the expression \eqref{SecondTermToC2}, we obtain:
\begin{align*}
	B''' & \partial(\gamma_{l} \otimes \omega^{s} \otimes \eta^{t}) \\
	& \overset{\eqref{TripleBoundary}}= B\bigl((-1)^{s+t} (\partial \gamma_{l} \cap \omega^{s} \otimes \eta^{t} - \gamma_{l} \cap d \omega^{s} \otimes \eta^{t}) - (-1)^{t} \gamma_{l} \cap \omega^{s} \otimes d\eta^{t} \bigr) \\
	& \hspace{35pt} + (-1)^{s} \bigl( B(\partial \gamma_{l} \otimes \omega^{s}) \wedge \eta^{t} - B(\gamma_{l} \otimes d \omega^{s}) \wedge \eta^{t} \bigr) + B(\gamma_{l} \otimes \omega^{s}) \wedge d\eta^{t} \\
	& \hspace{-7pt} \overset{\eqref{BoundaryCap}, \eqref{BoundaryTensor}}= B\bigl((-1)^{t} \partial (\gamma_{l} \cap \omega^{s}) \otimes \eta^{t} - (-1)^{t} \gamma_{l} \cap \omega^{s} \otimes d\eta^{t} \bigr) \\
	& \hspace{35pt} + B\partial (\gamma_{l} \otimes \omega^{s}) \wedge \eta^{t} + B(\gamma_{l} \otimes \omega^{s}) \wedge d\eta^{t} \\
	& \hspace{2pt} \overset{\eqref{BoundaryTensor}}= B\partial(\gamma_{l} \cap \omega^{s} \otimes \eta^{t}) + B\partial(\gamma_{l} \otimes \omega^{s}) \wedge \eta^{t} + B(\gamma_{l} \otimes \omega^{s}) \wedge d\eta^{t}
\end{align*}
Moreover:
\begin{align*}
	\partial B'''(&\gamma_{l} \otimes \omega^{s} \otimes \eta^{t}) \overset{\eqref{SecondTermToC2}, \eqref{BoudaryWedge}}= \partial B(\gamma_{l} \cap \omega^{s} \otimes \eta^{t}) + \partial B(\gamma_{l} \otimes \omega^{s}) \wedge \eta^{t} - B(\gamma_{l} \otimes \omega^{s}) \wedge d\eta^{t}
\end{align*}
Thus, $(B''' \partial + \partial B''')(\gamma_{l} \otimes \omega^{s} \otimes \eta^{t})$ coincides with \eqref{SecondBLL}.

\vspace{3pt} (4) We have:
\begin{align*}
	f_{*}\bigl([\gamma_{l}, & T_{l+1}] \hatcap f^{*}[\mu^{s}, c^{s-1}, \omega^{s}]\bigr) \\
	& \hspace{-3pt} \overset{\eqref{DiffCapP}}= f_{*}\bigl[\gamma_{l} \cap f^{*}\mu^{s}, (-1)^{s}(T_{l+1} \wedge f^{*}\omega^{s} + T_{\gamma_{l} \cap f^{*}c^{s-1}}) + B(\gamma_{l} \otimes f^{*}\omega^{s})\bigr] \\
	& = \bigl[f_{*}(\gamma_{l} \cap f^{*}\mu^{s}), (-1)^{s}f_{*}(T_{l+1} \wedge f^{*}\omega^{s} + T_{\gamma_{l} \cap f^{*}c^{s-1}}) + f_{*}B(\gamma_{l} \otimes f^{*}\omega^{s})\bigr]
\end{align*}
and
\begin{align*}
	f_{*}[\gamma_{l}, & T_{l+1}] \hatcap [\mu^{s}, c^{s-1}, \omega^{s}] \overset{\eqref{DiffCapP}}= \bigl[f_{*}\gamma_{l} \cap \mu^{s}, (-1)^{s}(f_{*}T_{l+1} \wedge \omega^{s} + T_{f_{*}\gamma_{l} \cap c^{s-1}}) + B(f_{*}\gamma_{l} \otimes \omega^{s})\bigr]
\end{align*}
The only terms we have to compare are $f_{*}B(\gamma_{l} \otimes f^{*}\omega^{s})$ and $B(f_{*}\gamma_{l} \otimes \omega^{s})$. If we consider the chain maps $C_{\bullet}(X; \R) \otimes \Omega_{\star}(Y) \to \Tau_{\bullet+\star}(Y)$ defined by $\gamma_{l} \otimes \omega^{s} \mapsto T_{f_{*}\gamma_{l}} \wedge \omega^{k}$ and $\gamma_{l} \otimes \omega^{s} \mapsto T_{f_{*}\gamma_{l} \cap \omega^{k}}$, then one can readily verify that two suitable choices for a chain homotopy between them are $\gamma_{l} \otimes \omega^{k} \mapsto f_{*}B(\gamma_{l} \otimes f^{*}\omega^{s})$ and $\gamma_{l} \otimes \omega^{k} \mapsto B(f_{*}\gamma_{l} \otimes \omega^{s})$. The essential uniqueness implies the thesis.

\SkipTocEntry \subsection{Proof of Proposition \ref{PropS1Diff}}\label{AppProofS1}

By identifying the 1-simplex $\Delta^{1} \subset \R^{2}$ with the unit interval $\Ii := [0, 1]$, we call $u_{1} \in Z_{1}(S^{1})$ the singular simplex $u_{1} \colon \Ii \to S^{1}$, $t \mapsto e^{2\pi i t}$. It immediately follows that $[u_{1}, 0] = [S^{1}]$, since $[u_{1}] = [S^{1}]$ topologically, and $T_{u_{1}} = \int_{S^{1}}$. Hence:
\begin{equation}\label{IntS1Formula}
	\textstyle \partial_{S^{1}}[\gamma_{k}, T_{k+1}] \overset{\eqref{DiffS1}}= [u_{1}, 0] \hattimes [\gamma_{k}, T_{k+1}] \overset{\eqref{DiffExtPr}}= \bigl[u_{1} \times \gamma_{k}, -\int_{S^{1}} \times T_{k+1} \bigr]
\end{equation}
We observe that $\int_{S^{1}} \times T = T \circ \int_{S^{1}} = \partial_{S^{1}}T$ for every $T \in \Tau_{\bullet}(X)$. In fact, given $\omega_{1} \in \Omega^{1}(S^{1})$ and $\omega_{2} \in \Omega^{\bullet-1}(X)$, and calling $\pi \colon S^{1} \times X \to X$ and $p \colon S^{1} \times X \to S^{1}$ the projections, the following identity holds:
\begin{align*}
	\textstyle \bigl(T \circ \int_{S^{1}}\bigr)(\omega_{1} \times \omega_{2}) & \textstyle = T\bigl(\int_{S^{1}}(p^{*}\omega_{1} \wedge \pi^{*}\omega_{2})\bigr) = T\bigl(\bigl(\int_{S^{1}} p^{*}\omega_{1}\bigr) \wedge \omega_{2})\bigr) \\
	& \textstyle = \bigl(\int_{S^{1}}\omega_{1}\bigr) \cdot T(\omega_{2}) = \bigl(\int_{S^{1}} \times T\bigr)(\omega_{1} \times \omega_{2})
\end{align*}
Therefore, item (1) follows from proposition \ref{PropExtSlant}, considering that $I[S^{1}] = [S^{1}]$ and $R[S^{1}] = \int_{S^{1}}$. About (2):
\begin{align*}
	(t \times \id_{X})_{*} \circ \partial_{S^{1}}[\gamma_{k}, T_{k+1}] & \textstyle \overset{\eqref{IntS1Formula}}= (t \times \id_{X})_{*}\bigl[u_{1} \times \gamma_{k}, -\int_{S^{1}} \times T_{k+1}\bigr] \\
	&  \textstyle \hspace{3pt} = \hspace{3pt} \bigl[t_{*}u_{1} \times \gamma_{k}, -\bigl(t_{*}\int_{S^{1}}\bigr) \times T_{k+1}\bigr] = -\partial_{S^{1}} [\gamma_{k}, T_{k+1}]
\end{align*}
In order to prove (3):
	\[\textstyle \pi_{*} \circ \partial_{S^{1}}[\gamma_{k}, T_{k}] \overset{\eqref{IntS1Formula}}= \pi_{*}\bigl[u_{1} \times \gamma_{k}, -\int_{S^{1}} \times T_{k+1}\bigr] = \bigl[\pi_{*}(u_{1} \times \gamma_{k}), -\pi_{*}\bigl(\int_{S^{1}} \times T_{k+1}\bigr)\bigr] = 0
\]
In fact:
	\[\textstyle \bigl(\int_{S^{1} \times T_{k+1}}\bigr)(\pi^{*}\omega^{k+1}) = \bigl(\int_{S^{1}} \times T_{k+1}\bigr)(1 \times \omega^{k+1}) = \int_{S^{1}} 1 \cdot T_{k+1}(\omega^{k+1}) = 0 \cdot T_{k+1}(\omega^{k+1}) = 0
\]
Moreover, for any singular $k$-simplex $\sigma$, by calling $P_{k+1} \in C_{k+1}(\Ii \times \Delta^{k})$ the standard triangulation inducing the prism map, the cross product $u_{1} \times \sigma$ coincides with the push-forward $(u_{1} \times \sigma)_{*}(P_{k+1})$ through the function $u_{1} \times \sigma \colon \Ii \times \Delta^{k} \to S^{1} \times X$, and $\pi_{*}(u_{1} \times \gamma_{k}) = \pi_{*}(u_{1} \times \gamma_{k})_{*}(P_{k+1})$ is the boundary of an integral combination of degenerate simplices, which is thin.

With respect to item (4), we first prove the formula analogous to \eqref{IntDer} about currents, that is:
\begin{equation}\label{IntDerCurr}
	\textstyle T \wedge \int_{S^{1}} \eta = \pi_{*}(\partial_{S^{1}}T \wedge \eta)
\end{equation}
In fact:
\begin{align*}
	\textstyle \bigl(T \wedge \int_{S^{1}} \eta\bigr)(\omega) & \textstyle = T\bigl(\bigl(\int_{S^{1}} \eta\bigr) \wedge \omega\bigr) = T\bigl(\int_{S^{1}}(\eta \wedge \pi^{*}\omega)\bigr) = (\partial_{S^{1}}T)(\eta \wedge \pi^{*}\omega) \\
	& \textstyle = (\partial_{S^{1}}T \wedge \eta)(\pi^{*}\omega) = \bigl(\pi_{*}(\partial_{S^{1}}T \wedge \eta)\bigr)(\omega)
\end{align*}
Moreover, given $\gamma \in C_{k}(X)$ and $\varphi \in C^{h+1}(S^{1} \times X)$, we set $\partial_{S^{1}}\gamma := u_{1} \times \gamma$ and $\int_{S^{1}} \varphi := \varphi \circ \partial_{S^{1}} \in C^{h}(X)$. The same definitions hold with coefficients in $\R/\Z$. Then, $\int_{S^{1}} [\mu^{s+1}, c^{s}, \omega^{s+1}] := \bigl[\int_{S^{1}} \mu^{s+1}, -\int_{S^{1}} c^{s}, \int_{S^{1}} \omega^{s+1}\bigr]$. We have $\gamma \cap \int_{S^{1}} \varphi = \pi_{*}(\partial_{S^{1}}\gamma \cap \varphi)$ up to the boundary of a thin chain, the proof being similar to the one of \eqref{IntDerCurr} by applying the definition of cap product. Analogously, one can prove that $B(\gamma \otimes \int_{S^{1}} \omega) = B(\partial_{S^{1}}\gamma \otimes \omega)$ from definition \eqref{HomotopyCapWedge}. We obtain:
\begin{align*}
	[\gamma_{k}, T_{k+1}] & \textstyle \hatcap \int_{S^{1}} [\mu^{s+1}, c^{s}, \omega^{s+1}] \\
	& \textstyle \overset{\eqref{DiffCapP}}= \bigl[\gamma_{l} \cap \int_{S^{1}}\mu^{s+1}, (-1)^{s}\bigl(T_{l+1} \wedge \int_{S^{1}}\omega^{s+1} - T_{\gamma_{l} \cap \int_{S^{1}} c^{s}}\bigr) + B\bigl(\gamma_{l} \otimes \int_{S^{1}}\omega^{s+1}\bigr)\bigr] \\
	& \textstyle \hspace{-2pt} \overset{\eqref{IntDerCurr}}= \pi_{*}\bigl[\partial_{S^{1}}\gamma_{l} \cap \mu^{s+1}, (-1)^{s+1}(-\partial_{S^{1}} T_{l+1} \wedge \omega^{s+1} + T_{\partial_{S^{1}}\gamma_{l} \cap c^{s}}) + B(\partial_{S^{1}}\gamma_{l} \otimes \omega^{s+1})\bigr] \\
	& \textstyle \overset{\eqref{DiffCapP}}= \pi_{*}\bigl([\partial_{S^{1}}\gamma_{l}, -\partial_{S^{1}}T_{l+1}] \hatcap [\mu^{s+1}, c^{s}, \omega^{s+1}]\bigr) \\
	& \hspace{-2pt} \overset{\eqref{IntS1Formula}}= \pi_{*}\bigl(\partial_{S^{1}}[\gamma_{l}, T_{l+1}] \hatcap [\mu^{s+1}, c^{s}, \omega^{s+1}]\bigr)
\end{align*}
About the first diagram in (5):
	\[\hatPD(\pi^{*}\xi) = [S^{1} \times X] \hatcap \pi^{*}\xi = ([S^{1}] \hattimes [X]) \hatcap (1 \hattimes \xi) = [S^{1}] \hattimes ([X] \hatcap \xi) = \partial_{S^{1}}\hatPD(\xi)
\]
About the second:
	\[\textstyle \hatPD\bigl(\int_{S^{1}}\xi\bigr) = [X] \hatcap \int_{S^{1}} \xi \overset{\eqref{IntDer}}= \pi_{*}(\partial_{S^{1}}[X] \hatcap \xi) = \pi_{*}([S^{1} \times X] \hatcap \xi) = \pi_{*}\hatPD(\xi) \qedhere
\]

\section{Exactness in Relative Homology}\label{AppExRelHom}

We show the proofs of the propositions about relative homology not included in the main text.

\SkipTocEntry \subsection{Proof of Proposition \ref{PropEx1}}

The first part of \eqref{LongExact1} is the sequence in homology with coefficients in $\R/\Z$, and the last part is the one in integral homology. Exactness in $H_{\bullet+1}(A; \R/\Z)$ follows from homology in $\R/\Z$, since $\iota$ is injective. Analogously, exactness in $H_{\bullet-1}(X)$ follows from integral homology, since $I$ is surjective.

In $\hat{H}_{\bullet}(X)$, one can easily show that $i_{X} \circ (\iota \circ f_{*}) = \iota \circ j_{X} \circ f_{*}$ with $j_{X} \colon H_{\bullet+1}(X; \R/\Z) \to H_{\bullet+1}(f; \R/\Z)$, and $j_{X} \circ f_{*} = 0$ because it is a part of the sequence with coefficients in $\R/\Z$. In turn, let us suppose $i_{X}(\lambda) = 0$. By definition
\begin{equation}\label{FormulaRRel}
	R\bigl[(\gamma_{k}, \theta_{k-1}), (T_{k+1}, T'_{k}) \bigr] = (T_{\gamma_{k}} + \partial T_{k+1} + f_{*}T'_{k}, T_{\theta_{k-1}} - \partial T'_{k}).
\end{equation}
By setting $\lambda = [\gamma_{k}, T_{k+1}]$, we deduce $R \circ i_{X}(\lambda) = R[(\gamma_{k}, 0), (T_{k+1}, 0)] = (R(\lambda), 0)$, hence $R(\lambda) = 0$. This means that $\lambda$ is a flat class, thus it lifts to $H_{k+1}(A; \R/\Z)$ along $f_{*}$ in the topological sequence.

In $\hat{H}_{\bullet-1}(f)$, we have $\pi_{A} \circ i_{X}[\gamma_{k}, T_{k+1}] = \pi_{A}[(\gamma_{k}, 0), (T_{k+1}, 0)] = 0$. Now we set $\lambda := [(\gamma_{k}, \theta_{k-1}), (T_{k+1}, T'_{k})]$ and suppose $\pi_{A}(\lambda) = 0$, that is, $[\theta_{k-1}, -T'_{k}] = 0$ by formula \eqref{MapPiA}. This means that $(\theta_{k-1}, -T'_{k}) = (\partial \Theta_{k}, -T_{\Theta_{k}} - \partial T''_{k+1})$. Therefore:
\begin{align*}
	\lambda &= \bigl[ (\gamma_{k}, \partial \Theta_{k}), (T_{k+1}, T_{\Theta_{k}} + \partial T''_{k+1}) \bigr] \\
	&= \bigl[ (\gamma_{k} + f_{*}\Theta_{k}, 0) + \partial (0, -\Theta_{k}), (T_{k+1} + f_{*}T''_{k+1}, T_{\Theta_{k}}) - \partial(0, T''_{k+1}) \bigr] \\
	&= \bigl[ (\gamma_{k} + f_{*}\Theta_{k}, 0), T_{(0, -\Theta_{k})} + (T_{k+1} + f_{*}T''_{k+1}, T_{\Theta_{k}}) \bigr] \\
	&= \bigl[ (\gamma_{k} + f_{*}\Theta_{k}, 0), (T_{k+1} + f_{*}T''_{k+1}, 0) \bigr] = i_{X}[\gamma_{k} + f_{*}\Theta_{k}, T_{k+1} + f_{*}T''_{k+1}]
\end{align*}
In $\hat{H}_{\bullet-1}(A)$, one can easily show that $f_{*} \circ I \circ \pi_{A} = f_{*} \circ p_{A} \circ I$ with $p_{A} \colon H_{\bullet}(f) \to H_{\bullet-1}(A)$, and $f_{*} \circ p_{A} = 0$ because of exactness in integral cohomology. In turn, let us suppose $f_{*} \circ I[\theta_{k-1}, T_{k}] = 0$. It follows that $f_{*}\theta_{k-1} = \partial \Theta_{k}$. We obtain the cycle $(-\Theta_{k}, \theta_{k-1}) \in Z_{k}(f)$, so that $\pi_{A}\bigl[(-\Theta_{k}, \theta_{k-1}), (0, -T_{k})\bigr] = [\theta_{k-1}, T_{k}]$ by formula \eqref{MapPiA}.

About \eqref{FormulasCompRel}:
\begin{align*}
	i_{X} \circ f_{*}[\theta_{k}, T_{k+1}] &= i_{X}[f_{*}\theta_{k}, f_{*}T_{k+1}] = \bigl[ (f_{*}\theta_{k}, 0), (f_{*}T_{k+1}, 0) \bigr] \\
	&= \bigl[ \partial(0, \theta_{k}), \partial(0, T_{k+1}) + (0, \partial T_{k+1}) \bigr] = \bigl[ (0, 0), T_{(0, \theta_{k})} + (0, \partial T_{k+1}) \bigr] \\
	&= \bigl[ (0, 0), (0, T_{\theta_{k}} + \partial T_{k+1}) \bigr] = a\bigl(0, R[\theta_{k}, T_{k+1}]\bigr) \overset{\eqref{DefCov}}= \cov \circ R[\theta_{k}, T_{k+1}]
\end{align*}
Furthermore:
\begin{align*}
	f_{*} \circ \pi_{A} \bigl[&(\gamma_{k}, \theta_{k-1}), (T_{k+1}, T'_{k}) \bigr] = f_{*}[\theta_{k-1}, -T'_{k}\bigr] \\
	&\overset{(!)}= [-\partial \gamma_{k}, -f_{*}T'_{k}] = [0, -T_{\gamma_{k}} - f_{*}T'_{k}] \overset{\eqref{FormulaRRel}}= -a \circ \pi_{X} \circ R\bigl(\bigl[(\gamma_{k}, \theta_{k-1}), (T_{k+1}, T'_{k}) \bigr]\bigr)
\end{align*}
In the equality $(!)$, we applied the fact that $(\gamma_{k}, \theta_{k-1})$ is a cycle in the mapping cone, hence $\partial\gamma_{k} + f_{*}\theta_{k-1} = 0$.

\SkipTocEntry \subsection{Proof of Proposition \ref{PropEx2}}

We prove that \eqref{ProjParallel} is surjective. Given a parallel class $\lambda := \bigl[[\gamma_{k}], [T_{k+1}] \bigr]$, we have $\partial[\gamma_{k}] = 0$ by definition, that is, $\partial\gamma_{k} = f_{*}\theta_{k-1}$ for some $\theta_{k-1}$. Since $f_{*}$ is injective on chains, this implies $\partial\theta_{k-1} = 0$, thus $[(\gamma_{k}, -\theta_{k-1}), (T_{k+1}, 0)]$ is a lift of $\lambda$ along $\hat{f}$. We now prove that $\Ker(\hat{f}) = \IIm(\cov)$. From definition \eqref{DefCov}, we deduce
	\[\hat{f} \circ \cov(T'_{k}) = \hat{f}\bigl[(0, 0), (0, T'_{k})\bigr] = \bigl[[0], [0]\bigr] = 0,
\]
therefore $\IIm(\cov) \subseteq \Ker(\hat{f})$. Conversely, let us suppose
	\[\hat{f}\bigl[(\gamma_{k}, \theta_{k-1}), (T_{k+1}, T'_{k})\bigr] \overset{\eqref{ProjParallel}}= \bigl[[\gamma_{k}], [T_{k+1}] \bigr] = 0.
\]
This means $([\gamma_{k}], [T_{k+1}]) = ([\partial\Gamma_{k+1}], [-T_{\Gamma_{k+1}} - \partial T''_{k+2}])$ for some $\Gamma_{k+1}$ and $T''_{k+2}$, that is, there exist $\phi_{k}$ and $U_{k+1}$ in $A$ satisfying:
\begin{equation}\label{CondKerI}
	\gamma_{k} = \partial\Gamma_{k+1} + f_{*}\phi_{k} \hspace{50pt} T_{k+1} = -(T_{\Gamma_{k+1}} + \partial T''_{k+2} + f_{*}U_{k+1})
\end{equation}
We obtain $\partial\gamma_{k} = f_{*}\partial\phi_{k}$ from the first condition. Since $[\gamma_{k}, \theta_{k-1}]$ is a cycle, we have $\partial \gamma_{k} + f_{*}\theta_{k-1} = 0$, thus the injectivity of $f_{*}$ implies $\theta_{k-1} = -\partial\phi_{k}$. It follows that $(\gamma_{k}, \theta_{k-1}) = (\partial\Gamma_{k+1} + f_{*}\phi_{k}, -\partial\phi_{k}) = \partial(\Gamma_{k+1}, \phi_{k})$. Hence, the second condition in \eqref{CondKerI} implies
\begin{align*}
	\bigl[(\gamma_{k}, \theta_{k-1}), (T_{k+1}, T'_{k})\bigr] &= \bigl[\partial(\Gamma_{k+1}, \phi_{k}), -(T_{\Gamma_{k+1}} + \partial T''_{k+2} + f_{*}U_{k+1}), T'_{k} \bigr] \\
	&= \bigl[(0, 0), T_{(\Gamma_{k+1}, \phi_{k})} - (T_{\Gamma_{k+1}} + \partial T''_{k+2} + f_{*}U_{k+1}), T'_{k} \bigr] \\
	&= \bigl[(0, 0), (- \partial T''_{k+2} - f_{*}U_{k+1}, T_{\phi_{k}} + T'_{k}) \bigr] \\
	&= \bigl[(0, 0), -\partial(T''_{k+2}, U_{k+1}) + (0, T_{\phi_{k}} + T'_{k} - \partial U_{k+1}) \bigr] \\
	&= \cov(T_{\phi_{k}} + T'_{k} - \partial U_{k+1}),
\end{align*}
therefore $\Ker(\hat{f}) \subseteq \IIm(\cov)$.

We now prove the exactness of \eqref{LongExact2}. The first part is the sequence in homology with coefficients in $\R/\Z$, and the last part is the sequence in integral cohomology. Exactness in $H_{\bullet+2}(X, A; \R/\Z)$ follows from homology in $\R/\Z$, since $\iota$ is injective. Analogously, exactness in $H_{\bullet-1}(A)$ follows from integral cohomology, since $I$ is surjective.

In $\hat{H}_{\bullet}(A)$, we have $f_{*} \circ (\iota \circ \beta) = \iota \circ f_{*} \circ \beta$, and $f_{*} \circ \beta = 0$ because it is a part of the sequence with coefficients in $\R/\Z$. In turn, let us suppose $f_{*}(\lambda) = 0$. This implies $f_{*}(R(\lambda)) = 0$. Since $f$ is a closed embedding, $f^{*} \colon \Omega^{\bullet}(X) \to \Omega^{\bullet}(A)$ is surjective, therefore $f_{*} \colon \Tau_{\bullet}(A) \to \Tau_{\bullet}(X)$ is injective. Hence, $R(\lambda) = 0$, so that $\lambda$ lifts to $H_{\bullet+2}(A; \R/\Z)$ along the Bockstein map in the topological sequence.

In $\hat{H}_{\bullet}(X)$, we have $\pi_{X} \circ f_{*}[\theta_{k}, T'_{k+1}] = \pi_{X}[f_{*}\theta_{k}, f_{*}T'_{k+1}] = \bigl[[f_{*}\theta_{k}], [f_{*}T'_{k+1}]\bigr] = [0, 0]$. Conversely, if $\pi_{X}[\gamma_{k}, T_{k+1}] = \bigl[[\gamma_{k}], [T_{k+1}]\bigr] = 0$, then there exist $\Theta_{k+1}$ and $T'_{k+2}$ satisfying $([\gamma_{k}], [T_{k+1}]) = ([\partial\Theta_{k+1}], [-T_{\Theta_{k+1}} - \partial T'_{k+2}])$. The class $[\gamma_{k} - \partial\Theta_{k+1}, T_{k+1} + T_{\Theta_{k+1}} + \partial T'_{k+2}]$ in $\hat{H}_{k}(A)$ thence lifts $[\gamma_{k}, T_{k+1}]$ along $f_{*}$.

In $\hat{H}_{\bullet}(X, A)$, we have $(\beta \circ I) \circ \pi_{X} = \beta \circ \pi_{X} \circ I$, and $\beta \circ \pi_{X} = 0$ because it is a part of the sequence in integral cohomology. In turn, let us suppose $\beta \circ I\bigl[[\gamma_{k}], [T_{k+1}]\bigr] = 0$, that is, $[\partial\gamma_{k}] = 0$ in $H_{k-1}(A)$. It follows that there exists a chain $\theta_{k} \in C_{k}(A)$ satisfying $\partial\gamma_{k} = \partial\theta_{k}$. The class $[\gamma_{k} - f_{*}\theta_{k}, T_{k+1}]$ thence lifts $\bigl[[\gamma_{k}], [T_{k+1}]\bigr]$ along $\pi_{X}$.

\SkipTocEntry \subsection{Proof of Proposition \ref{PropEx3}}

Exactness is a consequence of the construction through the mapping cone, hence we only have to prove formula \eqref{HPrimeIso}. We show that the following natural morphism is well-defined and surjective:
\begin{eqnarray}
	p \colon \hspace{5pt} \Ker(f_{*} \circ \pi_{A}) & \to & \hat{H}'_{k}(f) \label{DefMorP} \\
	\bigl\lbrack u_{k} = ((\gamma_{k}, \theta_{k-1}), (T_{k+1}, T'_{k})) \bigr\rbrack & \mapsto & \bigl\lbrack v_{k} = ((\gamma_{k}, T_{k+1}, R_{k}), (\theta_{k-1}, -T'_{k}, S_{k-1})) \bigr\rbrack \nonumber
\end{eqnarray}
where $R[u_{k}] = (R_{k}, S_{k-1})$. Moreover, we prove that its kernel is $\IIm(i_{X} \circ f_{*})$, so that it induces the isomorphism \eqref{HPrimeIso}.

We first check that $v_{k}$ is a cycle in the mapping cone. We have $f_{*} \circ \pi_{A}[u_{k}] = -a(R_{k})$ by formula \eqref{FormulasCompRel}. Then $R_{k}$ is an integral current since $[u_{k}] \in \Ker(f_{*} \circ \pi_{A})$, so it fits in the triple $(\gamma_{k}, T_{k+1}, R_{k})$. Formula \eqref{FormulaRRel} or a direct computation shows that $S_{k-1}$ is also integral (independently of $[u_{k}]$ belonging to the kernel), thus it fits in the second triple. The boundary of $v_{k}$ is formed by two triples. The first one is $\partial(\gamma_{k}, T_{k+1}, R_{k}) + f_{*}(\theta_{k-1}, -T'_{k}, S_{k-1})$, which coincides with
\begin{equation}\label{Comp1Bd}
	 (\partial \gamma_{k} + f_{*}\theta_{k-1}, \; R_{k} - T_{\gamma_{k}} - \partial T_{k+1} - f_{*}T'_{k}, \; f_{*}S_{k-1}).
\end{equation}
The first entry vanishes since $(\gamma_{k}, \theta_{k-1})$ in $u_{k}$ is a cycle, and the second one because of formula \eqref{FormulaRRel}. About the third one, both $(R_{k}, S_{k-1})$ and $R_{k}$ are closed, since they are integral, hence it easily follows that $f_{*}S_{k-1} = 0$. The second component of $\partial v_{k}$ is $-\partial (\theta_{k-1}, -T'_{k}, S_{k-1})$, coinciding with
\begin{equation}\label{Comp2Bd}
	 -(\partial \theta_{k-1}, \; S_{k-1} - T_{\theta_{k-1}} + \partial T'_{k}, 0).
\end{equation}
Again, the first entry vanishes since $(\gamma_{k}, \theta_{k-1})$ is a cycle, and the second one because of formula \eqref{FormulaRRel}.

We now verify that $v_{k}$ is a boundary if and only if $[u_{k}] \in \IIm(i_{X} \circ f_{*})$, which holds in particular if $[u_{k}] = 0$. This completes the proof of well-definedness and confirms the claim about the kernel. With an irrelevant shift on the degree, $v_{k-1}$ is a boundary if and only if there exists a chain $w_{k} = ((\gamma_{k}, T_{k+1}, R_{k}), (\theta_{k-1}, -T'_{k}, S_{k-1}))$ such that $v_{k-1} = \partial w_{k} = (\eqref{Comp1Bd}, \eqref{Comp2Bd})$. By definition \eqref{DefMorP}, this means that
	\[u_{k-1} = \bigl( (\partial \gamma_{k} + f_{*}\theta_{k-1}, \; -\partial \theta_{k-1}), \; (R_{k} - T_{\gamma_{k}} - \partial T_{k+1} - f_{*}T'_{k}, \; S_{k-1} - T_{\theta_{k-1}} + \partial T'_{k}) \bigr)
\]
We decompose $u_{k-1} = u'_{k-1} + u''_{k-1}$ as follows:
\begin{align*}
	& u'_{k-1} = \bigl( (\partial \gamma_{k} + f_{*}\theta_{k-1}, \; -\partial \theta_{k-1}), \; (- T_{\gamma_{k}} - \partial T_{k+1} - f_{*}T'_{k}, \; - T_{\theta_{k-1}} + \partial T'_{k}) \bigr) \\
	& u''_{k-1} = \bigl( (0, 0), (R_{k}, S_{k-1}) \bigr)
\end{align*}
We have
	\[u'_{k-1} = \bigl( \partial(\gamma_{k}, \theta_{k-1}), (-T_{(\gamma_{k}, \theta_{k-1})} - \partial(T_{k+1}, -T'_{k})) \bigr),
\]
which is exactly the condition $[u'_{k}] = 0$. Since $(R_{k}, 0)$ is integral, the only meaningful term in $u''_{k-1}$ is $((0, 0), (0, S_{k-1})) = \cov(S_{k-1})$. Here $S_{k-1}$ is any integral form, therefore $u''_{k-1} = \cov \circ R(\mu)$ where $\mu \in \hat{H}_{k-1}(A)$. Formula \eqref{FormulasCompRel} then implies that the kernel of $p$ is the image of $i_{X} \circ f_{*}$.

It only remains to prove surjectivity. Since $v_{k}$ is a cycle in \eqref{DefMorP}, formulas \eqref{Comp1Bd} and \eqref{Comp2Bd} imply respectively $R_{k} = T_{\gamma_{k}} + \partial T_{k+1} + f_{*}T'_{k}$ and $S_{k-1} = T_{\theta_{k-1}} - \partial T'_{k}$, so that $R_{k}$ and $S_{k-1}$ are completely determined by the entries of $u_{k}$. The other conditions in the mentioned formulas are equivalent to $(\gamma_{k}, \theta_{k-1})$ in $u_{k}$ being a cycle, which is part of the definition of $u_{k}$. Thus, any class $[v_{k}]$ admits a lift $[u_{k}]$ along $p$.

\SkipTocEntry \subsection{Proof of Proposition \ref{PropEx4}}

We leave the proof to the reader.

\section{Computations on Currents in Generalised Homology}\label{AppCurrents}

We show the proofs of the propositions about currents in generalised homology not included in the main text.

\SkipTocEntry \subsection{Boundary of Wedge Product}\label{ProofBdWedge}

We have:
\begin{align*}
	\bigl(\partial(\eta^{h} \wedge T_{k})\bigr)(\omega^{l}) &= (\eta^{h} \wedge T_{k})(d\omega^{l}) = T_{k}(d\omega^{l} \wedge \eta^{h}) \\
	& = T_{k}\bigl( d(\omega^{l} \wedge \eta^{h}) - (-1)^{l} \, \omega^{l} \wedge d\eta^{h} \bigr) \\
	& = (\partial T_{k})(\omega^{l} \wedge \eta^{h}) - (-1)^{l} \, T_{k}(\omega^{l} \wedge d\eta^{h}) \\
	& = \bigl(\eta^{h} \wedge \partial T_{k} - (-1)^{l} \, d\eta^{h} \wedge T_{k}\bigr)(\omega^{l}).
\end{align*}
Since $h^{\bullet}$ rationally even, the identity $(-1)^{l} = (-1)^{k+h+1}$ holds, proving the thesis.

\SkipTocEntry \subsection{Proof of Lemma \ref{LemmaProjHom}}\label{ProofLemmaProjHom}

We have:
\begin{align*}
	\bigl(\partial T_{(W, \hat{U}, \hat{A}, F)}\bigr)(\omega) &\overset{\eqref{BoundaryDifferential}}= T_{(W, \hat{U}, \hat{A}, F)}(d\omega) = \int_{W} F^{*}d\omega \wedge R(\hat{A}) \wedge \Td(\hat{U}) \\
	& \hspace{2pt} = \int_{W} d \bigl( F^{*}\omega \wedge R(\hat{A}) \wedge \Td(\hat{U}) \bigr) \\
	& \hspace{2pt} = \int_{\partial W} F^{*}\omega \wedge R(\hat{A}) \wedge \Td(\hat{U}) = T_{\partial(W, \hat{U}, \hat{A}, F)}(\omega)
\end{align*}
The integral $\int_{W} F^{*}\omega \wedge R(\hat{A}) \wedge \Td(\hat{U})$ is trivially additive on $W$ and $\hat{A}$. Let us fix a submersion $\varphi \colon M \to N$ and two precycles $(M, \hat{u}, \hat{\alpha}, f \circ \varphi)$ and $(N, \hat{v}, \varphi_{!}\hat{\alpha}, f)$. We denote by $\hat{n}$ the orientation of the normal bundle $\mathcal{N}$ on $M$, that satisfies the identity
\begin{equation}\label{TdOrientH}
	\Td(\hat{u}) = \varphi^{*}\Td(\hat{v}) \wedge \Td(\hat{n})
\end{equation}
Then:
\begin{align*}
	T_{(M, \hat{u}, \hat{\alpha}, f \circ \varphi)}(\omega) & \overset{\eqref{DefIndCurr}}= \int_{M} (f \circ \varphi)^{*}\omega \wedge R(\hat{\alpha}) \wedge \Td(\hat{u}) \\
	& \hspace{2pt} = \int_{N} \int_{M/N} \varphi^{*}f^{*}\omega \wedge R(\hat{\alpha}) \wedge \Td(\hat{u}) \\
	& \hspace{-3pt} \overset{\eqref{TdOrientH}}= \int_{N} \int_{M/N} \varphi^{*}\bigl( f^{*}\omega \wedge \Td(\hat{v}) \bigr) \wedge R(\hat{\alpha}) \wedge \Td(\hat{n}) \\
	& \hspace{2pt} = \int_{N} f^{*}\omega \wedge \Td(\hat{v}) \wedge \biggl( \int_{M/N} R(\hat{\alpha}) \wedge \Td(\hat{n}) \biggr) \\
	& \hspace{-1pt} \overset{\eqref{GysinR}}= \int_{N} f^{*}\omega \wedge R(\varphi_!\hat{\alpha}) \wedge \Td(\hat{v}) \overset{\eqref{DefIndCurr}}= T_{(N, \hat{v}, \varphi_{!}\hat{\alpha}, f)}(\omega)
\end{align*}
This shows that \eqref{BarPsiIsoH} is a well-defined morphism. It is an isomorphism since composition with \eqref{IsoHomR} leads to the analogous isomorphism in singular cohomology with coefficients in $\h_{\R}$.

\SkipTocEntry \subsection{Conventions on Fibre-Wise Integration}\label{ConvFibreInt}

Dealing with a generalised cohomology theory, we integrate a form $\omega \in \Omega^{k}(Y)$ on a fibre bundle $\pi \colon Y \to X$ of rank $n$ following the convention
	\[\biggl(\int_{Y/X} \omega\biggr)_{x}(v_{1}, \ldots, v_{k-n}) := \int_{Y_{x}}\omega(\tilde{v}_{1}, \ldots, \tilde{v}_{k-n}, \,\cdot\,, \ldots, \,\cdot\,)
\]
where $df(\tilde{v}_{i}) = v_{i}$. It follows that $\int_{Y/X} d\omega = d \int_{Y/X}\omega$ and $\int_{Y/X} (\pi^{*}\omega \wedge \eta) = \omega \wedge \int_{Y/X} \eta$. In particular, if $Y = X \times F$, then $\int_{X \times F/X} \pi_{X}^{*}\eta \wedge \pi_{F}^{*}\vol_{F} = \eta$.

In the first part about singular homology, we used the opposite convention
	\[\biggl(\int_{Y/X} \omega\biggr)_{x}(v_{1}, \ldots, v_{k-n}) := \int_{Y_{x}}\omega(\,\cdot\,, \ldots, \,\cdot\,, \tilde{v}_{1}, \ldots, \tilde{v}_{k-n})
\]
which differs from the previous one by the sign $(-1)^{n(\abs{\omega}-1)}$. In this case, $\int_{Y/X} d\omega = (-1)^{n} d\int_{Y/X}\omega$ and $\int_{Y/X} (\eta \wedge \pi^{*}\omega = \bigl(\int_{Y/X} \eta\bigr) \wedge \omega$. In particular, if $Y = F \times X$, then $\int_{F \times X/X} \pi_{F}^{*}\vol_{F} \wedge \pi_{X}^{*}\eta = \eta$.

\SkipTocEntry \subsection{Proof of Lemma \ref{LemmaIntForms}}\label{ProofLemmaIntForms}

We have:
\begin{align*}
	\bigl[T_{(M, \hat{u}, \hat{\alpha}, f)}(\omega)\bigr] &\overset{\eqref{DefIndCurr}}= \biggl[ \int_{M} f^{*}\omega \wedge R(\hat{\alpha}) \wedge \Td(\hat{u}) \biggr] = \int_{M} f^{*}\ch(\xi) \cdot \ch(\alpha) \cdot \Td(u) \\
	& \hspace{3pt} = \int_{M} \ch(f^{*}\xi \cdot \alpha) \cdot \Td(u) \overset{\eqref{GysinCh}}= \ch\bigl((p_{M})_{!}(f^{*}\xi \cdot \alpha)\bigr) \in \IIm(\ch)
\end{align*}

\SkipTocEntry \subsection{Proof of Formulas \eqref{PushFWedge} and \eqref{WedgeTodd}}\label{LemmaCurrBundleProof}

We have:
\begin{align*}
	\bigl(i_{*}(\eta \wedge T)\bigr)(\omega) &= (\eta \wedge T)(i^{*}\omega) = T(i^{*}\omega \wedge \eta) = T\bigl(i^{*}(\omega \wedge \pi^{*}\eta)\bigr) \\
	&= (i_{*}T)(\omega \wedge \pi^{*}\eta) = (\pi^{*}\eta \wedge i_{*}T)(\omega)
\end{align*}
and, by calling $\hat{v}$ the orientation of $X$:
\begin{align*}
	\bigl(\eta \wedge \ScT^{0}\Td(\hat{u})\bigr)(\omega) &= T_{\Td(\hat{u})}(\omega \wedge \eta) = \int_{X} \omega \wedge \eta \wedge \Td(\hat{u}) \wedge \Td(\hat{v}) \\
	& = \int_{X} \Td(\hat{u}) \wedge \omega \wedge \eta \wedge \Td(\hat{v}) = T_{\eta}(\Td(\hat{u}) \wedge \omega) \\
	& = \bigl(T_{\eta} \wedge \Td(\hat{u})\bigr)(\omega) = \bigl(\ScT^{\bullet}\eta \wedge \Td(\hat{u})\bigr)(\omega)
\end{align*}

\section{Proofs on Differentiation}\label{AppDiff}

We show the proofs of the propositions about differentiation and homological Gysin map not included in the main text.

\SkipTocEntry \subsection{Proof of Theorem \ref{ThmDiff}}\label{ProofThmDiff}

By choosing a strict representative on $E$, the class $\hat{u} \pmb{\cdot} \hat{\lambda}$ belongs to the image of $i_{*}$, thus also $\Thom(\hat{\alpha}) \pmb{\cdot} \hat{\lambda} = \pi^{*}\hat{\alpha} \pmb{\cdot} (\hat{u} \pmb{\cdot} \hat{\lambda})$. Hence, $(\pi^{*}\hat{\alpha} \cdot \hat{u}) \pmb{\cdot} \hat{\lambda} = i_{*}\hat{\nu}$ for some $\hat{\nu}$. By applying $\pi_{*}$ on both sides, we obtain $\hat{\nu} = \pi_{*}(\pi^{*}\hat{\alpha} \pmb{\cdot} (\hat{u} \pmb{\cdot} \hat{\lambda})) = \hat{\alpha} \pmb{\cdot} \pi_{*}(\hat{u} \pmb{\cdot} \hat{\lambda}) = (-1)^{n(m-\bullet)} \hat{\alpha} \pmb{\cdot} \Thom'(\hat{\lambda})$, proving \eqref{FundThThom}. About \eqref{FundThVectBundle}, we first observe that
\begin{equation}\label{IntWedgeCurr}
	\textstyle \bigl(\int_{E/X} \eta\bigr) \wedge T = \pi_{*}\bigl(\eta \wedge \bigl(T \circ \int_{E/X}\bigr)\bigr)
\end{equation}
as the reader can prove by direct computation. We set $\lambda = [M, \hat{v}, \hat{\beta}, f, T] \in \hat{h}_{l}(X)$. Moreover, we call $\bar{E} := f^{*}E$ the pull-back of $E$, and $\bar{\pi} \colon \bar{E} \to M$ and $\bar{f} \colon \bar{E} \to E$ the corresponding projection and map over $f$. Then:
\begin{align*}
	\textstyle \bigl(\int_{E/X} \hat{\alpha}&\bigr) \pmb{\cdot} [M, \hat{v}, \hat{\beta}, f, T] \\
	& \textstyle \hspace{-25pt} \overset{\eqref{DiffCapPH}}= \hspace{1.5pt} \bigl[ M, \hat{v}, \bigl(f^{*} \int_{E/X} \hat{\alpha}\bigr) \cdot \hat{\beta}, f, R\bigl(\int_{E/X} \hat{\alpha}\bigr) \wedge T \bigr] \\
	& \textstyle \hspace{-22pt} = \hspace{4.5pt} \bigl[ M, \hat{v}, \bigl(\int_{\bar{E}/M} \bar{f}^{*}\hat{\alpha}\bigr) \cdot \hat{\beta}, f, \Td(\hat{u})^{-1} \wedge \bigl(\int_{E/X} R(\hat{\alpha})\bigr) \wedge T \bigr] \\
	& \textstyle \hspace{-22pt} = \hspace{4.5pt} \bigl[ M, \hat{v}, (-1)^{(\abs{\hat{\alpha}}-n)\abs{\hat{\beta}}} \bar{\pi}_{!}(\bar{\pi}^{*}\hat{\beta} \cdot \bar{f}^{*}\hat{\alpha}), f, \bigl(\int_{E/X} R(\hat{\alpha})\bigr) \wedge \bigl(\Td(\hat{u})^{-1} \wedge T\bigr) \bigr] \\
	& \textstyle \hspace{-27pt} \overset{\eqref{IntWedgeCurr}}= \hspace{-1pt} \bigl[ \bar{E}, \hat{v} \times \hat{u}', (-1)^{(\abs{\hat{\alpha}}-n)\abs{\hat{\beta}}} \bar{\pi}^{*}\hat{\beta} \cdot \bar{f}^{*}\hat{\alpha}, f \circ \bar{\pi}, \pi_{*}\bigl(R(\hat{\alpha}) \wedge \bigl(\bigl(\Td(\hat{u})^{-1} \wedge T\bigr) \circ \int_{E/X}\bigr) \bigr) \bigr] \\
	& \textstyle \hspace{-22pt} = \hspace{4pt}\bigl[ \bar{E}, (-1)^{n\abs{M}}\hat{u}' \times \hat{v}, (-1)^{n\abs{\hat{\beta}}} \bar{f}^{*}\hat{\alpha} \cdot \bar{\pi}^{*}\hat{\beta}, \pi \circ \bar{f}, \pi_{*}\bigl(R(\hat{\alpha}) \wedge T \circ \bigl(\Td(\hat{u})^{-1} \wedge \int_{E/X}\bigr) \bigr) \bigr] \\
	& \textstyle \hspace{-25pt} \overset{\eqref{DiffCapPH}}= \hspace{2pt} \pi_{*} \bigl( \hat{\alpha} \pmb{\cdot} \bigl[ \bar{E}, \hat{u}' \times \hat{v}, (-1)^{nl} \bar{\pi}^{*}\hat{\beta}, \bar{f}, T \circ \bigl(\Td(\hat{u})^{-1} \wedge \int_{E/X}\bigr) \bigr] \bigr)
\end{align*}
Let us show that the class multiplied by $\hat{\alpha}$ in the last formula, that we denote by $\hat{\mu}$, is $(-1)^{n(m-l)} \partial_{E/X} \hat{\lambda}$. Because of \eqref{Cond1Diff} and \eqref{Cond2Diff}, this means respectively that $\pi_{*}(\hat{u} \pmb{\cdot} \hat{\mu}) = \hat{\lambda}$ and $R(\hat{\mu}) = R(\hat{\lambda}) \circ \bigl(\Td(\hat{u})^{-1} \wedge \int_{E/X}\bigr)$. Indeed:
\begin{align*}
	\pi_{*}(\hat{u} \pmb{\cdot} \hat{\mu}) &= \textstyle \pi_{*}\bigl[ \bar{E}, \hat{u}' \times \hat{v}, (-1)^{nl} \hat{u} \pmb{\cdot} \bar{\pi}^{*}\hat{\beta}, \bar{f}, R(\hat{u}) \wedge T \circ \bigl(\Td(\hat{u})^{-1} \wedge \int_{E/X}\bigr) \bigr] \\
	&= \textstyle \bigl[ \bar{E}, \hat{v} \times \hat{u}', \bar{\pi}^{*}\hat{\beta} \pmb{\cdot} \hat{u}, \pi \circ \bar{f}, \pi_{*}\bigl(R(\hat{u}) \wedge T \circ \bigl(\Td(\hat{u})^{-1} \wedge \int_{E/X}\bigr) \bigr) \bigr] \\
	&= \textstyle [\bar{E}, \hat{v} \times \hat{u}', \bar{\pi}^{*}\hat{\beta} \pmb{\cdot} \hat{u}, f \circ \bar{\pi}, T] = [M, \hat{v}, \bar{\pi}_{!}(\bar{\pi}^{*}\hat{\beta} \pmb{\cdot} \hat{u}), f, T] \\
	&= [M, \hat{v}, \hat{\beta}, f, T] = \hat{\lambda}
\end{align*}
and $R(\hat{\mu}) = T_{(\bar{E}, \hat{u}' \times \hat{v}, (-1)^{nl} \bar{\pi}^{*}\hat{\beta}, \bar{f})} + \partial\bigl(T \circ \bigl(\Td(\hat{u})^{-1} \wedge \int_{E/X}\bigr) \bigr)$. We have
\begin{align*}
	T_{(\bar{E}, \hat{u}' \times \hat{v}, (-1)^{nl} \bar{\pi}^{*}\hat{\beta}, \bar{f})}(\omega) & = (-1)^{n\abs{\hat{\beta}}} T_{(\bar{E}, \hat{v} \times \hat{u}', \bar{\pi}^{*}\hat{\beta}, \bar{f})}(\omega) \\
	& = (-1)^{n\abs{\hat{\beta}}} \int_{\bar{E}} \bar{f}^{*}\omega \wedge \bar{\pi}^{*}\bigl(R(\hat{\beta}) \wedge \Td(\hat{v}) \wedge f^{*}\Td(\hat{u})^{-1}\bigr) \\
	& = (-1)^{(n-\abs{\omega})\abs{\hat{\beta}}} \int_{M} R(\hat{\beta}) \wedge \Td(\hat{v}) \wedge f^{*}\Td(\hat{u})^{-1} \wedge \int_{\bar{E}/M} \bar{f}^{*}\omega \\
	& = \int_{M} f^{*}\biggl( \Td(\hat{u})^{-1} \wedge \int_{E/M} \omega \biggr) \wedge R(\hat{\beta}) \wedge \Td(\hat{v}) \\
	& = \biggl( T_{(M, \hat{v}, \hat{\beta}, f)} \circ \biggl(\Td(\hat{u})^{-1} \wedge \int_{E/X}\biggr) \biggr)(\omega)
\end{align*}
and $\partial\bigl(T \circ \bigl(\Td(\hat{u})^{-1} \wedge \int_{E/X}\bigr) \bigr) = (\partial T) \circ \bigl(\Td(\hat{u})^{-1} \wedge \int_{E/X}\bigr)$, hence the result follows.

About \eqref{CommDiagThom}, for every $\hat{\alpha} \in \hat{h}^{m+n-k}(E)$ we have:
\begin{align*}
	& \Thom' \circ \hatPD(\hat{\alpha}) = \Thom'[E, \hat{v} \times \hat{u}', \hat{\alpha}, \id_{E}, 0] = (-1)^{n\abs{\hat{\alpha}}} \pi_{*}(\hat{u} \pmb{\cdot} [E, \hat{v} \times \hat{u}', \hat{\alpha}, \id_{E}, 0]) \\
	& \hspace{63pt} = \pi_{*}([E, \hat{v} \times \hat{u}', (-1)^{n\abs{\hat{\alpha}}} \hat{u} \cdot \hat{\alpha}, \id_{E}, 0]) = [E, \hat{v} \times \hat{u}', \hat{\alpha} \cdot \hat{u}, \pi, 0] \\
	& \hatPD \circ i^{*}(\hat{\alpha}) = [X, \hat{v}, i^{*}\hat{\alpha}, \id_{X}, 0] = [X, \hat{v}, i^{*}\hat{\alpha}, \pi \circ i, 0] = [E, \hat{v} \cdot \hat{u}, i_{!}(i^{*}\hat{\alpha}), \pi, 0] \\
	& \hspace{54pt} = [E, \hat{v} \times \hat{u}', \hat{\alpha} \cdot i_{!}(1), \pi, 0] = [E, \hat{v} \times \hat{u}', \hat{\alpha} \cdot \hat{u}, \pi, 0]
\end{align*}
and, for every $\hat{\alpha} \in \hat{k}(X)$:
	\[i_{*}\hatPD(\hat{\alpha}) = [X, \hat{v}, \hat{\alpha}, i] = [E, \hat{v} \cdot \hat{u}', i_{!}\hat{\alpha}, \id_{E}] = \hatPD(i_{!}\hat{\alpha}) = \hatPD(\Thom(\hat{\alpha}))
\]
About \eqref{CommDiagDiff}, given $\hat{\alpha} \in \hat{h}^{m-k}(X)$, we set $\hat{\mu} := \partial_{E/X} \circ \hatPD(\hat{\alpha})$. By definition, it is the unique class that satisfies conditions \eqref{Cond1Diff} and \eqref{Cond2Diff}, which become respectively:
\begin{align*}
	& \Thom'\hat{\mu} = \hatPD(\hat{\alpha}) = [X, \hat{v}, \hat{\alpha}, \id_{X}, 0] \\
	& R(\hat{\mu}) = (-1)^{n(m-k)} \, T_{R(\hat{\alpha})} \circ \biggl( \Td(\hat{u})^{-1} \wedge \int_{E/X} \biggr)
\end{align*}
Let us show that $\hat{\mu} = \hatPD \circ \pi^{*}(\hat{\alpha})$. Indeed:
\begin{align*}
	\Thom' \circ \hatPD \circ \pi^{*}(\hat{\alpha}) &= \Thom'[E, \hat{v} \times \hat{u}', \pi^{*}\hat{\alpha}, \id_{E}, 0] = (-1)^{n(m-k)} [E, \hat{v} \times \hat{u}', \hat{u} \cdot \pi^{*}\hat{\alpha}, \pi, 0] \\
	&= [E, \hat{v} \times \hat{u}', \pi^{*}\hat{\alpha} \cdot \hat{u}, \pi, 0] = [X, \hat{v}, \hat{\alpha} \cdot \pi_{!}(\hat{u}), \id_{X}, 0] = [X, \hat{v}, \hat{\alpha}, \id_{X}, 0]
\end{align*}
and, for any $\omega \in \Omega^{n-k}_{\vcpt}(E; \h_{\R})$:
\begin{align*}
	\bigl(R \circ \hatPD \circ \pi^{*}(\hat{\alpha})\bigr)(\omega) &= T_{(E, \hat{v} \times \hat{u}', \pi^{*}\hat{\alpha}, \id_{E})}(\omega) = \int_{E} \omega \wedge \pi^{*}\bigl(R(\hat{\alpha}) \wedge \Td(\hat{v}) \wedge \Td(\hat{u})^{-1}\bigr) \\
	&= (-1)^{(n-k)(m-k)} \int_{E} \pi^{*}\bigl(R(\hat{\alpha}) \wedge \Td(\hat{v}) \wedge \Td(\hat{u})^{-1}\bigr) \wedge \omega \\
	&= (-1)^{(n-k)(m-k)} \int_{X} R(\hat{\alpha}) \wedge \Td(\hat{v}) \wedge \Td(\hat{u})^{-1} \wedge \biggl( \int_{E/X} \omega \biggr) \\
	&= (-1)^{(n-k)(m-k) + (m-k)k} \int_{X} \biggl( \int_{E/X} \omega \biggr) \wedge R(\hat{\alpha}) \wedge \Td(\hat{u})^{-1} \wedge \Td(\hat{v}) \\	
	&= \biggl((-1)^{n(m-k)} \, T_{R(\hat{\alpha})} \circ \biggl( \Td(\hat{u})^{-1} \wedge \int_{E/X} \biggr)\biggr)(\omega)
\end{align*}
Lastly, for every $\hat{\alpha} \in \hat{h}^{k}_{\vcpt}(E)$:
	\[\textstyle \pi_{*}\hatPD(\hat{\alpha}) = [E, \hat{v} \cdot \hat{u}', \hat{\alpha}, \pi, 0] = [X, \hat{v}, \pi_{!}\hat{\alpha}, \id_{X}, 0] = \hatPD(\pi_{!}\hat{\alpha}) = \hatPD\bigl(\int_{E/X}\hat{\alpha}\bigr) \qedhere
\]

\SkipTocEntry \subsection{Proof of Theorem \ref{PropGysinHom}}\label{ProofPropGysinHom}

We have:
\begin{align*}
	(\varphi_{!} \hat{\alpha}) \pmb{\cdot} \hat{\lambda} & \textstyle \hspace{5pt} = \hspace{5pt} \bigl(\int_{X \times \R^{N}/X} j_{*}\phi_{*}(\pi^{*}\hat{\alpha} \cdot \hat{u})\bigr) \pmb{\cdot} \hat{\lambda} \\
	& \overset{\eqref{FundThVectBundle}}= (-1)^{N(m-l)} (\pi_{N})_{*}\bigl( j_{*}\phi_{*}(\pi^{*}\hat{\alpha} \cdot \hat{u}) \pmb{\cdot} \partial_{X \times \R^{N}/X} \hat{\lambda} \bigr) \\
	& \hspace{5pt} = \hspace{5pt} (-1)^{N(m-l)} (\pi_{N})_{*} j_{*} \phi_{*}\bigl( (\pi^{*}\hat{\alpha} \cdot \hat{u}) \pmb{\cdot} \phi^{*}j^{*}\partial_{X \times \R^{N}/X} \hat{\lambda} \bigr) \\
	& \overset{\eqref{FundThThom}}= (-1)^{N(m-l)+(N-d)(m-l)} (\pi_{N} j \phi i)_{*}(\hat{\alpha} \pmb{\cdot} \Thom'\phi^{*}j^{*}\partial_{X \times \R^{N}/X} \hat{\lambda}) \\
	& \overset{\eqref{GysinMapHom}}= (-1)^{d(m-l)} \varphi_{*}(\hat{\alpha} \pmb{\cdot} \varphi^{!}\hat{\lambda})
\end{align*}
By considering diagrams \eqref{CommDiagThom} and \eqref{CommDiagDiff}, we have:
\begin{align*}
	\varphi^{!} \circ \hatPD &= \Thom' \circ \phi^{*} \circ j^{*} \circ \partial_{X \times \R^{N}/X} \circ \hatPD = \hatPD \circ i^{*} \circ \phi^{*} \circ j^{*} \circ \pi^{*} \\
	&= \hatPD \circ (\pi \circ j \circ \phi \circ i)^{*} = \hatPD \circ (\pi \circ \iota)^{*}(\hat{\alpha}) = \hatPD \circ \varphi^{*}
\end{align*}
Lastly:
\begin{align*}
	\varphi_{*} \circ \hatPD(\hat{\alpha}) &= [Y, \hat{v}, \hat{\alpha}, \varphi, 0] = [X, \hat{u}, \varphi_{!}\hat{\alpha}, \id_{X}, 0] = \hatPD(\varphi_{!}\hat{\alpha}) \qedhere
\end{align*}



\end{document}